\documentclass[a4paper, oneside]{amsart}

\usepackage{geometry}
\RequirePackage[l2tabu, orthodox]{nag}
\usepackage[all, warning]{onlyamsmath}
\usepackage{color}
\usepackage[new]{old-arrows}

\usepackage{amsmath}
\usepackage{amsthm}
\usepackage{amsfonts}
\usepackage{amssymb}
\usepackage{amscd}

\usepackage{mathrsfs}
\usepackage{bm}
\usepackage{bbm}
\usepackage{stmaryrd}
\usepackage{braket}

\usepackage[all]{xy}
\usepackage{pb-diagram}
\usepackage{graphicx}

\usepackage{url}
\usepackage[dvipdfmx,setpagesize=false]{hyperref}

\theoremstyle{plain}
\newtheorem{thm}{Theorem}[section]
\newtheorem{lem}[thm]{Lemma}

\newtheorem{prop}[thm]{Proposition}
\newtheorem{cor}[thm]{Corollary}
\newtheorem{conj}[thm]{Conjecture}

\newtheorem{dfn-lem}[thm]{Definition-Lemma}
\newtheorem{ex-thm}[thm]{Example}
\newtheorem*{thm*}{Theorem}

\theoremstyle{definition}
\newtheorem{dfn}[thm]{Definition}

\newtheorem{rmk}[thm]{Remark}

\theoremstyle{remark}

\newcommand{\bZ}{\mathbb{Z}}
\newcommand{\Z}{\mathbb{Z}}
\newcommand{\Q}{\mathbb{Q}}
\newcommand{\R}{\mathbb{R}}
\newcommand{\C}{\mathbb{C}}

\newcommand{\PP}{\mathbb{P}}

\newcommand{\bC}{\mathbb{C}}
\newcommand{\bF}{\mathbb{F}}
\newcommand{\bH}{\mathbb{H}}
\newcommand{\bQ}{\mathbb{Q}}

\newcommand{\bbF}{\mathbb{F}}

\newcommand{\bbH}{\mathbb{H}}

\newcommand{\bbP}{\mathbb{P}}

\newcommand{\bbR}{\mathbb{R}}

\newcommand{\mcA}{\mathcal{A}}

\newcommand{\mcC}{\mathcal{C}}

\newcommand{\mcJ}{\mathcal{J}}

\newcommand{\mcM}{\mathcal{M}}

\newcommand{\mcO}{\mathcal{O}}

\newcommand{\cO}{\mathcal{O}}

\newcommand{\cE}{\mathcal{E}}
\newcommand{\cH}{\mathcal{H}}

\newcommand{\cM}{\mathcal{M}}
\newcommand{\cS}{\mathcal{S}}

\newcommand{\msF}{\mathscr{F}}

\newcommand{\msL}{\mathscr{L}}

\newcommand{\msR}{\mathscr{R}}
\newcommand{\msS}{\mathscr{S}}
\newcommand{\msT}{\mathscr{T}}

\newcommand{\msW}{\mathscr{W}}

\newcommand{\mfZ}{\mathfrak{Z}}

\newcommand{\mfa}{\mathfrak{a}}
\newcommand{\mfb}{\mathfrak{b}}

\newcommand{\mfz}{\mathfrak{z}}

\newcommand{\fz}{\mathfrak{z}}
\newcommand{\fZ}{\mathfrak{Z}}

\DeclareMathOperator{\Gal}{Gal}
\DeclareMathOperator{\Hom}{Hom}

\DeclareMathOperator{\Ext}{Ext}
\DeclareMathOperator{\Tor}{Tor}

\DeclareMathOperator{\Ann}{Ann}

\DeclareMathOperator{\image}{im}
\DeclareMathOperator{\coker}{coker}

\DeclareMathOperator{\sgn}{sgn}

\newcommand{\re}{\operatorname{Re}}
\newcommand{\im}{\operatorname{Im}}

\newcommand{\isomto}{\stackrel{\sim}{\longrightarrow}}

\newcommand{\bs}{\backslash}

\newcommand{\ra}{\rightarrow}

\newcommand{\setm}{\!-\!}

\newcommand{\ord}{\mathrm{ord}}

\numberwithin{equation}{section}

\newcommand{\Rpos}{\mathbb{R}_{>0}}

\newcommand{\brk}[1]{\langle#1\rangle}
\newcommand{\lr}[1]{\left(#1\right)}

\newcommand{\Eis}{\mathrm{Eis}}
\newcommand{\BS}{\mathrm{BS}}
\newcommand{\MS}{\mathcal{MS}}

\allowdisplaybreaks[3]

\begin{document}
\title{Harder's denominator problem for $\mathrm{SL}_2(\Z)$ and its applications}


\author{Hohto Bekki}
\address{Hohto Bekki \newline Max Planck Institute for Mathematics\\Vivatsgasse 7\\53111 Bonn\\Germany}
\email{bekki@mpim-bonn.mpg.de}

\author{Ryotaro Sakamoto}
\address{Ryotaro Sakamoto \newline Department of Mathematics\\University of Tsukuba\\1-1-1 Tennodai\\Tsukuba\\Ibaraki 305-8571\\Japan}
\email{rsakamoto@math.tsukuba.ac.jp}

\begin{abstract}
The aim of this paper is to give a full detail of the proof given by Harder of a theorem on the denominator of the Eisenstein class for $\mathrm{SL}_2(\Z)$ and to show that the theorem has some interesting applications including the proof of a recent conjecture by Duke on the integrality of the higher Rademacher symbols. 
We also present a sharp universal upper bound for the denominators of the values of partial zeta functions associated with narrow ideal classes of real quadratic fields in terms of the denominator of the values of the Riemann zeta function. 
\end{abstract}

\maketitle
\tableofcontents

\section{Introduction}

\subsection{The denominator of Eisenstein classes for $\mathrm{SL}_2(\Z)$}
Let $\Gamma:=\mathrm{SL}_2(\Z)$ be the modular group, 
and let $\bbH:=\{z \in \C \mid \im(z)>0\}$ be the upper half plane, on which $\Gamma$ acts by the linear fractional transformation. 
We denote by $Y:=\Gamma \bs \bbH$ the modular curve of level $\Gamma$, by $Y^{\BS}$ its Borel--Serre compactification, and by $\partial Y^{\BS}:=Y^{\BS}\setm Y$ the Borel--Serre boundary of $Y$. 
For any even integer $n \geq 2$, we define a left $\Gamma$-module $\mcM_n:=\mathrm{Sym}^n(\Z^2)$ to be the $n$-th symmetric power of $\Z^2$, and define $\mcM_n^{\vee}:=\Hom_{\Z}(\mcM_n, \Z)$ to be its dual $\Gamma$-module. 
Then the $\Gamma$-module $\mcM_n^{\vee}$ naturally defines a sheaf on $Y^{\BS}$ (which we also denote by $\mcM_n^\vee$), and we can consider the cohomology groups $H^{\bullet}(Y^{\BS}, \mcM_n^{\vee})$ and $H^{\bullet}(\partial Y^{\BS}, \mcM_n^{\vee})$. 
Since $\mcM_n^{\vee}$ has a left action of $M_2(\bZ)$, these cohomology groups   carry the structure of Hecke modules. 

The boundary $\partial Y^{\BS}$ is  identified with $\Gamma_\infty \bs \bbR$, where $\Gamma_\infty := \left\{\begin{pmatrix} 1&a\\0&1\end{pmatrix} \, \middle| \, a \in \bZ \right\}$. 
Hence it is easy to see that $\dim_{\bQ} H^{1}(\partial Y^{\BS}, \mcM_n^{\vee} \otimes \bQ) = 1$ and we have a natural generator  $\omega_n \in H^{1}(\partial Y^{\BS}, \mcM_n^{\vee})/(\text{torsion})$. 

Harder considered in his book \cite{CAG} a unique Hecke-equivariant section 
\[
H^{1}(\partial Y^{\BS}, \mcM_n^{\vee} \otimes \bQ) \longrightarrow H^{1}( Y^{\BS}, \mcM_n^{\vee} \otimes \bQ)
\]
of the canonical homomorphism $H^{1}(Y^{\BS}, \mcM_n^{\vee} \otimes \bQ) \longrightarrow H^{1}(\partial Y^{\BS}, \mcM_n^{\vee} \otimes \bQ)$ induced by the inclusion map $\partial Y^{\BS} \longhookrightarrow Y^{\BS}$, and he defined  the Eisenstein cohomology class 
\[
\Eis_n \in H^1(Y^{\BS}, \mcM_n^{\vee}\otimes \Q)
\]
to be the image of $\omega_n$ under this section. 
Then Harder studied the denominator $\Delta(\Eis_n)$ of the Eisenstein cohomology class $\Eis_n$, that is, the smallest positive integer $\Delta(\Eis_n)$ such that 
\[
\Delta(\Eis_n) \Eis_n \in H^1(Y^{\BS}, \mcM_n^{\vee})/(\text{torsion})
\subset H^1(Y^{\BS}, \mcM_n^{\vee}\otimes \Q). 
\]
See also the dissertation  \cite{Wang89} of Wang. 
As a result, the following theorem was obtained by Harder (see \cite[Staz 2, \S1]{HP92}). 

\begin{thm}[{Harder~{\cite[Theorem 5.1.2]{CAG}}}]\label{intro:main theorem 1}
For any even integer $n \geq 2$, we have
\[
\Delta(\Eis_n) = \textrm{the numerator of }\zeta(-1-n), 
\]
where $\zeta(s)$ denotes the Riemann zeta function\footnote{In the present paper, the numerator and the denominator of a rational number are always defined to be positive integers.}. 
\end{thm}

\begin{rmk}
    Haberland had proved a slightly weaker version of Theorem \ref{intro:main theorem 1} in \cite[pp. 272--273]{Haberland}. More precisely, let $p$ be a prime number and assume that $p>n$. Then Haberland obtained 
    \[
\ord_p(\Delta(\Eis_n)) \leq \ord_p(\textrm{the numerator of }\zeta(-1-n)). 
    \]
    Moreover, if we further assume that $p$ divides $\zeta(-1-n)$ and that there exists $\nu \in \{1, \dots, n-1\}$ such that $p \nmid \zeta(-\nu)\zeta(\nu-n)$, then he obtained the equality
    \[
\ord_p(\Delta(\Eis_n)) = \ord_p(\textrm{the numerator of }\zeta(-1-n)). 
    \]
\end{rmk}

\begin{rmk}
For any prime numbers $p$ and $\ell$ with $\ell \nmid p(p-1)$, 
the denominators over $\bZ_\ell$ of Eisenstein classes for $\Gamma_1(p)$ with a character was computed by Kaiser in the diploma thesis \cite{Kaiser90} (see also the paper \cite{Mahnkopf00} of Mahnkopf for the study of Eisenstein classes for $\Gamma_1(p^e\ell)$). 
Eisenstein classes for $\mathrm{GL}_2$ over totally real  fields have been studied by Maennel in the dissertation \cite{Maennel93}. 
Eisenstein classes for $\mathrm{GL}_2$ over imaginary quadratic fields have been studied by Harder in \cite{Harder81, Harder82}, Weselmann in \cite{Weselmann88}, Berger in  \cite{Berger08, Berger09}, 
and Branchereau in \cite{Branchereau23}. 
\end{rmk}

The purpose of the present paper is to report some arithmetic applications of Theorem \ref{intro:main theorem 1}  especially to the special values of partial zeta functions of real quadratic fields. 
However, since the book \cite{CAG} (which is available on Harder's web-page) is still under development, some of the important arguments and references in the proof of Theorem \ref{intro:main theorem 1} are currently not given completely. 
Taking this situation into account, we decided to also give the detailed proof of Theorem \ref{intro:main theorem 1}, which is another main purpose of the present paper.

\subsection{Reformulation in terms of the holomorphic Eisenstein series}
In view of applications, we interpret the above definition of the Eisenstein class $\Eis_n$ and Theorem \ref{intro:main theorem 1} in terms of the holomorphic Eisenstein series and the Eichler--Shimura homomorphism. 

In the following, let $n \geq 2$ be an even integer and let $M_{n+2}(\Gamma)$ denotes the space of modular forms of  level $\Gamma=\mathrm{SL}_2(\Z)$ and weight $n+2$. 
Then we have the Hecke-equivariant homomorphism 
\[
r \colon M_{n+2}(\Gamma) 
\longrightarrow
H^1(Y, \mcM_n^{\vee}\otimes \C) = H^1(Y^\BS, \mcM_n^{\vee}\otimes \C)
\]
called the Eichler--Shimura homomorphism which is defined by certain path integrals on $\bbH$ (see \S \ref{sec:Eichler--Shimura map} for the precise definition of the Eichler--Shimura homomorphism). 
Let 
\[
E_{n+2}(z) := 1 + \frac{2}{\zeta(-1-n)}\sum_{k=1}^{\infty}\sigma_{n+1}(k)q^k \in M_{n+2}(\Gamma),  \quad  q := e^{2\pi i z}
\]
denote the holomorphic Eisenstein series of weight $n+2$. 
Then the following is the reformulation of Theorem \ref{intro:main theorem 1} (\cite[Theorem 5.1.2]{CAG}) which will be proved in the present paper. 

\begin{thm}[{Lemma \ref{lem:eisenstein-pairing-eigen-rational}, Proposition \ref{prop:eisenstein-rational}, and Theorem \ref{thm:main result}}]\label{intro:reformulated theorem}
~
\begin{enumerate}
\item 
We have $r(E_{n+2})=\Eis_n$, i.e., 
the class $r(E_{n+2})$ coincides with the Eisenstein class $\Eis_n$. 
\item 
The denominator of $r(E_{n+2})$ is equal to the numerator of $\zeta(-1-n)$. 
\end{enumerate}
\end{thm}

\begin{rmk}
From the $q$-expansion of the Eisenstein series $E_{n+2}$, we see that the denominator of $E_{n+2}$ with respect to the integral structure coming from the $q$-expansion (the de Rham integral structure) is clearly the numerator of $\zeta(-1-n)$. An interesting and non-trivial point in Theorem \ref{intro:reformulated theorem} is that on the Eisenstein parts $\bQ E_{n+2}$ and $\Q\Eis_n$, the Betti integral structure 
coincides with the de Rham integral structure under the Eichler--Shimura homomorphism. Cf. \cite[\S 1.1]{harder_can_we} and Remark \ref{rem:differences integral structure (1)-ordinary (2)-de rham--betti}. 
\end{rmk}

\subsection{Strategy of the proof of Theorem \ref{intro:reformulated theorem}}

We review the strategy of the proof of Theorem \ref{intro:reformulated theorem}, which is based on Harder's argument in \cite{CAG}. 
First, note that for any prime number $p$, we have
\[
\ord_p(\Delta(\Eis_n))
=
\min\{\delta \in \Z_{\geq 0} \mid p^{\delta}\Eis_n \in H^1(Y^{\BS}, \mcM_{n}^{\vee} \otimes \Z_{(p)}) \}. 
\]
Therefore, it suffices to prove that 
\[
 \ord_p(\Delta(\Eis_n)) = \ord_p(\text{the numerator of } \zeta(-1-n))
\]
for each prime number $p$. 
Then the proof consists roughly of the following four parts. 
\begin{enumerate}
\item[(I)] First, we construct in \S\ref{sec:construction of the lift} a certain family of $p$-adically integral homology classes  
\[
\widetilde{T_p^{m}(C_{\nu}(\tau))} \in H_1(Y^{\BS}, \mcM_{n}\otimes \Z_{(p)})
\]
for any $\tau \in \bbH$, sufficiently large integer $m$, and  $1 \leq \nu \leq n-1$, 
where $H_1(Y^{\BS}, \mcM_{n}\otimes \Z_{(p)})$ is the cosheaf homology group associated with the $\Gamma$-module $\mcM_{n}\otimes \Z_{(p)}$. 
Note that the homology class $\widetilde{T_p^{m}(C_{\nu}(\tau))}$ is independent of the choice of $\tau \in \bbH$. 
See  Lemma \ref{lem:T_p-cycle-integral}. 
\item[(II)] 
Next, in \S \ref{sec:Period}, we compute  the $p$-adic limit of the value of the pairing 
\[
\lim_{m \to \infty}
\brk{\Eis_n, \widetilde{T_p^{m!}(C_{\nu}(\tau))}}, 
\]
where $\brk{~,~} \colon H^1(Y^{\BS}, \mcM_{n}^\vee) \times H_1(Y^{\BS}, \mcM_{n}) \longrightarrow \bZ$ is the pairing induced by $\mcM_{n}^\vee \times \cM_{n} \longrightarrow \bZ; (f, x) \mapsto f(x)$. 
 More precisely, we will show in Theorem \ref{thm:Tp^m and Eis} and Corollary \ref{cor:reformulation-of-THeorem-Tp^m-and-Eis-pairing} that this $p$-adic limit can be described in terms of the Kubota--Leopoldt $p$-adic $L$-functions, namely, for any integer $\nu \in \{1, \ldots, n-1\}$ we obtain the following interesting formula 
   \begin{align*}
   \lim_{m \to \infty}\brk{\Eis_n, \widetilde{T_p^{m!}(C_{\nu}(\tau))}} 
    =  
   \frac{1-p^{n+1}}{(1-p^\nu)(1-p^{n-\nu})}D_p(n, \nu), 
   \end{align*}
where 
\[
  D_p(n, \nu) := \frac{L_{p}(-\nu, \omega^{1+\nu})L_{p}(\nu-n, \omega^{n-\nu+1})}{L_{p}(-1-n, \omega^{n+2})} - L_{p}(-\nu, \omega^{1+\nu}) - L_{p}(\nu-n, \omega^{n-\nu+1}) 
\]
and 
$L_{p}(s, \omega^{a})$ denotes the  Kubota--Leopoldt $p$-adic $L$-function associated with the $a$-th power of the Teichm\"uller character $\omega$.

\item[(III)] 
In \S \ref{sec:denominator-cohomology-class}, we prove that 
the homology classes $\widetilde{T_p^{m!}(C_{1}(\tau))}, \ldots, \widetilde{T_p^{m!}(C_{n-1}(\tau))}$ give a generator of the ordinary part of the quotient group $(H_1(Y^{\BS}, \mcM_{n})/H_1(\partial Y^{\BS}, \mcM_{n})) \otimes \Z_{p}$ (see Proposition \ref{prop:homology_generator_ordinary}). 
Hence, since $\brk{\Eis_{n}, H_1(\partial Y^{\BS}, \mcM_{n})} = \bZ$,  we obtain 
\begin{align}\label{intro-eq:delta_p(n) = - min ord_p(D_p(n,nu))}
     \ord_p(\Delta(\Eis_n)) = \min \{a \in \bZ_{\geq0} \mid p^aD(n,\nu) \in \bZ_p \textrm{ for any integer } 1 \leq \nu \leq n-1 \}. 
\end{align}
See Corollary \ref{cor:ordinary class denominator}, and also Proposition \ref{prop:delta-p-Eisenstein=max-delta-p-n-nu}.

\item[(IV)] In \S\ref{sec:proof of main thm}, we show that the right hand side of \eqref{intro-eq:delta_p(n) = - min ord_p(D_p(n,nu))} is equal to the $p$-adic valuation of the numerator of $\zeta(-1-n)$. 
We devote \S\ref{sec:Eisenstein class} and \S\ref{sec:p-adic L} to the preparation for proving this fact. 
\end{enumerate}

\subsection{Applications to Duke's conjecture and to the special values of the partial zeta functions of real quadratic fields}

\subsubsection{Duke's conjecture}

In the paper \cite{Duke23}, Duke defined a certain map 
\[
\Psi_k: \Gamma=\mathrm{SL}_2(\Z) \longrightarrow \Q
\]
for each integer $k \geq 2$ called the higher Rademacher symbol which is a generalization of the classical Rademacher symbol, and he conjectured the integrality of the higher Rademacher symbol (\cite[Conjecture, p. 4]{Duke23}). 
As a first application of Theorem \ref{intro:reformulated theorem}, we prove this conjecture.

\begin{thm}[{Corollary \ref{cor:duke's conjecture}}]\label{intro-thm:duke's conjecture}
    Duke's conjecture holds true, that is, for any integer $k \geq 2$ and matrix $\gamma \in \Gamma$, we have 
\[
\Psi_k(\gamma) \in \Z. 
\]
\end{thm}

In fact, Duke proved in \cite[Lemma 6]{Duke23} that the higher Rademacher symbols can be written as the integral of the holomorphic Eisenstein series  along a certain homology cycles (see Proposition \ref{prop:rademacher integral}). 
Therefore, we can derive Theorem \ref{intro-thm:duke's conjecture} directly from Theorem \ref{intro:reformulated theorem}.

\begin{rmk}
Duke's conjecture is recently proved also by O'Sullivan in \cite{Sullivan23} using a more direct method. 
\end{rmk}

\subsubsection{The denominators of the partial zeta functions of real quadratic fields}

Next, we discuss the denominators of the partial zeta functions associated with narrow ideal classes of orders in real quadratic fields.

Let $F\subset \R$ be a real quadratic field, $\mcO \subset F$ be an order in $F$, and $\mcA \in Cl_{\mcO}^+$ be a narrow ideal class of $\mcO$. Then we have the associated partial zeta function
\[
\zeta_{\mcO}(\mcA,s)
=
\sum_{\mfa \subset \mcO, \mfa \in \mcA}
\frac{1}{N\mfa^s}, 
\]
which can be continued meromorphically to $\C$, and it is known that 
\[
\zeta_{\mcO}(\mcA, 1-k) \in \Q
\]
for any integer $k \geq 2$. We also define the positive integer $J_{2k}$ by 
\[
J_{2k} := \text{the denominator of } \zeta(1-2k). 
\]
Then in \S \ref{sec:denominator of partial zeta functions}, we obtain the following as another consequence of Theorem \ref{intro:reformulated theorem}. 

\begin{prop}
[Corollary \ref{cor:partial zeta denominator and J_n}]\label{intro:denominator of partial zeta}
Let $k \geq 2$ be an integer. 
Then the integer $J_{2k}$ gives a universal upper bound for the denominator of $\zeta_{\mcO}(\mcA, 1-k)$ with respect to  orders $\cO$ and narrow ideal classes $\mcA \in Cl_\cO^+$. 
In other words, we have 
\[
J_{2k}\zeta_{\mcO}(\mcA, 1-k) \in \Z 
\]
for all orders $\cO$ in all real quadratic fields and narrow ideal classes $\mcA \in Cl_\cO^+$. 
\end{prop}

In fact, one can construct a natural map 
$\mfz_{\cO, k} \colon Cl_\cO^+ \longrightarrow H_1(Y^\BS, \cM_{2k-2})$ for any integer $k \geq 2$ (see Definition \ref{def:definition of mfz}), and we show in Proposition \ref{prop:Hecke integral formula} that 
\[
\brk{\Eis_{2k-2}, \mfz_{\cO, k}(\mcA^{-1})} = (-1)^k\frac{\zeta_\cO(\mcA, 1-k)}{\zeta(1-2k)} = \pm \frac{J_{2k}\zeta_\cO(\mcA, 1-k)}{N_{2k}}. 
\]
Hence Proposition \ref{intro:denominator of partial zeta} follows  from Theorem \ref{intro:reformulated theorem}. 
See also Remark \ref{rem:duke's conjecture and integrality partial zetas are equivalent}  for the relation between Duke's conjecture and Proposition \ref{intro:denominator of partial zeta}.

Next, we discuss the sharpness of the universal upper bound obtained in Proposition \ref{intro:denominator of partial zeta}.

\begin{thm}[Corollary \ref{cor:sherpness of the universal bound of partial zeta functions}]\label{intro:sharpness}
The universal upper bound in Proposition \ref{intro:denominator of partial zeta} is sharp, that is, we have
\[
J_{2k}=\min\left\{ J \in \Z_{>0} \,\, \middle|\,\,   
\begin{aligned}
J\zeta_\mcO(\mcA, 1-k) \in \Z 
\text{ for all orders $\mcO$ in all real quadratic fields}\\ 
\text{and narrow ideal classes $\mcA \in Cl_{\mcO}^+$ } 
\end{aligned} \, 
\right\}. 
\]
\end{thm}

In order to derive Theorem \ref{intro:sharpness} from Theorem \ref{intro:reformulated theorem}, we need to show that the narrow ideal classes of orders in real quadratic fields produce sufficiently large submodule of  the homology group $H_1(Y^{\BS}, \mcM_{2k-2})$, and this will be done in \S \ref{sec:sharpness} using some techniques from Hida theory.

\begin{rmk}
    As for the denominator or the integrality of the special values of partial zeta functions of real quadratic fields, or more generally of totally real fields, many works have been done by 
Coates and Sinnott \cite{CS74-Stickelberger, CS74-p-adic, CS77}, 
Deligne and Ribet \cite{DR80}, Cassou-Nogu\`es \cite{Cassou-Nogues}, Charollois, Dasgupta, and Greenberg \cite{CDG15}, Beilinson, Kings, and Levin \cite{BKL18}, Bannai, Hagihara, Yamada, and Yamamoto \cite{BHYY22}, Bergeron, Charollois, and Garcia \cite{BCG20} etc., by using variety of methods including Hilbert modular forms, Shintani zeta functions \cite{Shintani76}, Sczech's Eisenstein cocycles~\cite{Sczech93}, etc. 
Actually, when $\cO = \cO_F$, the upper bound in Proposition \ref{intro:denominator of partial zeta} follows from these preceding works. 
More precisely, the results proved by Coates and Sinnott in \cite{CS77} or Deligne and Ribet in \cite{DR80} show that for any prime number $p$, we have 
\[
2^{-1}(1 - p^{2k})\zeta_{\cO_F}(\mcA, 1- k) \in \bZ[1/p], 
\]
which implies that 
\[
J_{2k}\zeta_{\cO_F}(\mcA, 1- k) \in \bZ 
\]
(see \cite[pp. 73, 75]{Zagier76}). 

One feature of the method in the present paper is that by using Theorem \ref{intro:reformulated theorem}, we capture not only the upper bound for the denominators of the partial zeta functions associated with any orders, but also the sharpness of the upper bound. 
\end{rmk}

\begin{rmk}
In the paper \cite{Zagier77}, Zagier proved a certain formula \cite[p. 149, Corollaire]{Zagier77} which explicitly computes the special values $\zeta_\cO(\mcA, 1-k)$ of partial zeta functions of orders of real quadratic fields at negative integers in a uniform way. 
Then by using this formula, he obtained a universal upper bound $d_k$ of the denominators of the values $\zeta_\cO(\mcA, 1-k)$ and examined its sharpness briefly. 
More precisely, he observed that the upper bound $d_k$ is not sharp and discussed how one can improve this upper bound when $k = 2, 3$ (see \cite[pp. 149--150]{Zagier77}). 
Theorem \ref{intro:sharpness} can be seen as the complete answer to this problem of determining the sharp universal upper bound for the denominators of $\zeta_\cO(\mcA, 1-k)$. 
\end{rmk}

\subsection*{Acknowledgements}
We would like to express our deepest gratitude to G\"unter Harder who explained to us many beautiful ideas and showed us his notes and manuscripts on the proof of his theorem. 
We would also like to thank him for encouraging us to write the proof of his theorem in this paper. 
We are also grateful to Herbert Gangl, Christian Kaiser, and Don Zagier for the fruitful discussions and many valuable comments during the study. 
In particular, Herbert Gangl and Don Zagier suggested us to conduct a numerical experiment to test the sharpness of Proposition \ref{intro:denominator of partial zeta}, and helped us to write PARI/GP programs for the experiment, which provided us with some important ideas to prove Theorem \ref{intro:sharpness}. 
Thanks are also due to Toshiki Matsusaka who drew our attention to Duke's conjecture, which became one of the main motivations in this study. 
This research has been carried out during the first author's stay at the Max Planck Institute for Mathematics in Bonn.

\section{Preliminaries and the Eisenstein class}

In this section, we give the definition of the Eisenstein class and 
explain  Theorem \ref{intro:reformulated theorem} (see Theorem \ref{thm:main result}).

Throughout this paper,  $n \geq 2$ denotes an even integer.

\subsection{Definitions of Modular curve and Borel--Serre compactification}
\label{sec:Definitions of Modular curve and Borel--Serre compactification}

Let 
\[
\bbH:=\{z \in \C \mid \im(z)>0\}
\]
denote the upper half plane, and let 
\[
\bbH^{\BS}
:=\bbH \sqcup 
\bigsqcup_{ r \in \PP^1(\Q)}\PP^1(\R) \setminus \{r\}
\]
be the Borel--Serre compactification of $\bbH$ (see \cite{Gor05} or \cite[\S1,2,7]{CAG}). 
We set 
\[
\Gamma:= \mathrm{SL}_2(\bZ).
\] 
The group $\Gamma$ acts on $\bbH$ and $\bbH^{\BS}$ by the linear fractional transformation as usual. 
We denote by
\begin{align*}
Y:=\Gamma \bs \bbH \,\,\, \textrm{ and } \,\,\, Y^{\BS} :=\Gamma \bs \bbH^{\BS}
\end{align*}
the modular curve of level $\mathrm{SL}_2(\Z)$ and its Borel--Serre compactification,  respectively. 
Moreover, we denote by $\partial Y^{\BS}:= Y^{\BS} \setm Y$ the boundary of $Y^{\BS}$. 
The boundary $\partial Y^\BS$ is homeomorphic to the circle $S^1$ and the fundamental group $\pi_1(\partial Y^\BS)$ can be identified with  $\Gamma_\infty := \left\{ \begin{pmatrix}
1&a\\0&1 \end{pmatrix}
\, \middle| \, a \in \bZ
 \right\}$. 

In the following, $\bbH^{?}$ (resp. $Y^{?}$) means that it is either $\bbH$ or $\bbH^{\BS}$ (resp. $Y$ or $Y^{\BS}$). 
Any left $\Gamma$-module  $\mcM$ can be regarded as (co)sheaf on $Y^{?}$ in a natural way, and hence we can consider the homology groups 
\begin{align*}
H_{\bullet}(Y^{?}, \mcM), \,\,\, 
H_{\bullet}(\partial Y^{\BS}, \mcM), \,\,\, 
H_{\bullet}(Y^{\BS}, \partial Y^{\BS}, \mcM), 
\end{align*}
which fit into the long exact sequence
\begin{align*}
\cdots
\ra
H_1(\partial Y^{\BS}, \mcM)
\ra
H_1(Y^{\BS}, \mcM)
\ra
H_1(Y^{\BS}, \partial Y^{\BS}, \mcM)
\ra
H_0(\partial Y^{\BS}, \mcM)
\ra
\cdots. 
\end{align*}
Similarly, we have the cohomology groups
\begin{align*}
H^{\bullet}(Y^{?}, \mcM), \,\,\, 
H^{\bullet}(\partial Y^{\BS}, \mcM), \,\,\, 
H^{\bullet}(Y^{\BS}, \partial Y^{\BS}, \mcM), 
\end{align*}
which fit into the long exact sequence
\begin{align*}
\cdots
\ra
H^1(Y^{\BS}, \partial Y^{\BS}, \mcM)
\ra
H^1(Y^{\BS}, \mcM)
\ra
H^1(\partial Y^{\BS}, \mcM)
\ra
H^2(Y^{\BS}, \partial Y^{\BS}, \mcM)
\ra
\cdots. 
\end{align*}

We note that the inclusion map $Y \longhookrightarrow Y^\BS$ induces  isomorphisms
\[
H_\bullet(Y, \mcM)
\isomto
H_\bullet(Y^{\BS}, \mcM)
\,\,\, \textrm{ and } \,\,\, 
H^\bullet(Y^{\BS}, \mcM)
\isomto
H^\bullet(Y, \mcM). 
\]
Moreover, if $\cM$ has an action of $M_2^+(\bZ) := \{\gamma \in M_2(\bZ) \mid \det(\gamma) > 0\}$, then these homology groups (resp. cohomology groups) carry the structure of Hecke modules, namely, for each prime number $p$ we have a Hecke operator $T_p$ (resp. $T_p'$) on these homology groups (resp. cohomology groups), and the above long exact sequences are compatible with the Hecke operators.

In Section \ref{sec:modular symbol}, we give a way to compute these (co)homology groups, 
and in Section \ref{sec:Hecke operator}, we give an explicit description of the Hecke operators.



\begin{rmk}\label{rem:sheafification-functor-non-exactness}
    As a sheaf on $Y$, the stalk $\mcM_{x}$ at $x \in Y$ coincides with  $\mcM^{\Gamma_{\tilde{x}}}$, where $\tilde{x} \in \bbH$ is a lift of $x \in Y$ and 
    $\Gamma_{\tilde{x}} := \{\gamma \in \Gamma \mid \gamma \tilde{x} = \gamma \}$. 
    This fact shows that a short exact sequence of left $\Gamma$-modules does not give a short exact sequence of sheaves on $Y$ in general, that is, the sheafification functor is not exact.  
    However, a short exact sequence of left $\bZ[1/6][\Gamma$]-modules induces  
    a short exact sequence of sheaves on $Y^\BS$ since the order of $\Gamma_{\tilde{x}}$ divides $6$. 
\end{rmk}

\subsection{Modular symbols and (co)homology}\label{sec:modular symbol}
\label{sec:Modular symbols and (co)homology}

Let $X \in \{\bbH, \bbH^\BS, \partial \bbH^\BS\}$ and 
let $(S_{\bullet}(X), \partial)$ denote the usual singular chain complex of $X$, i.e., $S_q(X)$ is the free abelian group generated by singular $q$-simplices in $X$ and 
$\partial \colon S_q(X) \longrightarrow S_{q-1}(X)$ is the boundary operator. 

The left action of $M_2^+(\bZ)$ on $X$ induces a left action of $M_2^+(\bZ)$ on $S_{\bullet}(X)$, and $(S_{\bullet}(X), \partial)$ is actually an $M_2^+(\bZ)$-equivariant complex. Then it is known that for any left $M_2^+(\bZ)$-module $\mcM$, which is also seen as a (co)sheaf on $\Gamma\bs X$, 
we have natural isomorphisms
\begin{align*}
    H_{\bullet}(\Gamma\bs X, \mcM) &\cong H_{\bullet}((S_{\bullet}(X)\otimes \mcM)_{\Gamma}), 
\\
H_{\bullet}(Y^{\BS},\partial Y^{\BS} \mcM) &\cong H_{\bullet}(((S_{\bullet}(\bbH^{\BS})/S_{\bullet}(\partial \bbH^{\BS})\otimes \mcM)_{\Gamma}), 
\end{align*}
where $(-)_{\Gamma}$ denotes the $\Gamma$-coinvariant functor. 
Here the left $M_2^+(\bZ)$-action on $S_{\bullet}(X) \otimes \mcM$ is defined by 
\[
\gamma \cdot (\sigma \otimes m) := \gamma\sigma \otimes \gamma m,   
\]
where $\sigma \in S_{\bullet}(X)$, $m \in \cM$, and $\gamma \in M_2^+(\bZ)$. 
Set
\[
\MS(X) := \coker(S_2(X) \stackrel{\partial}{\longrightarrow} S_1(X)). 
\]
For any elements $\alpha, \beta \in X$, we denote the equivalence class of a path from $\alpha$ to $\beta$ in $\MS(X)$ by 
\[
\{\alpha, \beta\} \in \MS(X). 
\]
The boundary map $\partial\colon S_1(X)\ra S_0(X)$ induces a $\Gamma$-homomorphism $\partial\colon \MS(X)\ra S_0(X)$, and we have a natural isomorphism
\begin{align*}
    H_1(\Gamma\bs X, \mcM)\cong \ker((\MS(X)\otimes \mcM)_{\Gamma} \overset{\partial}{\ra} (S_0(X)\otimes \mcM)_{\Gamma}). 
\end{align*}
Similarly, we also have natural isomorphisms
\begin{align*}
    H^{\bullet}(\Gamma\bs X, \mcM) &\cong H^{\bullet}(\Hom_{\Z}(S_{\bullet}(X), \mcM)^{\Gamma}), 
    \\
    H^{\bullet}(Y^{\BS},\partial Y^{\BS} \mcM) &\cong H^{\bullet}( \Hom_\bZ(S_{\bullet}(\bbH^{\BS})/S_{\bullet}(\partial \bbH^{\BS}), \mcM)^{\Gamma}), 
\end{align*}
where $(-)^{\Gamma}$ denotes the $\Gamma$-invariant functor. 
Here the left $M_2^+(\bZ)$-action on $\Hom_{\Z}(S_{\bullet}(X), \mcM)$ is defined by 
\[
(\gamma\phi)(\sigma) := \gamma(\phi(\tilde{\gamma}\sigma)),   
\]
where $\phi \in \Hom_{\Z}(S_{\bullet}(X), \mcM)$, $\sigma \in S_{\bullet}(X)$, and $\tilde{\gamma}$ is the adjugate  of $\gamma \in M_2^+(\bZ)$. 
Since  
\[
\ker\left(\Hom_{\Z}(S_{1}(X), \mcM) \ra \Hom_{\Z}(S_{2}(X), \mcM)\right) = \Hom_{\Z}(\MS(X), \mcM), 
\]
we have a natural isomorphism 
\[
H^1(\Gamma\bs X, \mcM) \cong  \coker(\Hom_{\Z}(S_0(X), \mcM)^{\Gamma}\ra\Hom_{\Z}(\MS(X), \mcM)^{\Gamma}). 
\]

\subsection{$M_2^+(\bZ)$-modules $\mcM_n$ and $\mcM_n^{\flat}$}

For any $2 \times 2$ matrix $\gamma=\begin{pmatrix}
    a&b\\
    c&d
\end{pmatrix}$, we denote the adjugate  of $\gamma$ by   
\[
\widetilde{\gamma}
:=
\begin{pmatrix}
    d&-b\\
    -c&a
\end{pmatrix}. 
\]
 Note that if $\gamma \in \Gamma$, then we have $\widetilde{\gamma}=\gamma^{-1}$. 

Let $\Z[X_1,X_2]$ denote the ring of polynomials of two variables over $\Z$, and we equip $\Z[X_1,X_2]$ with a left action of $M_2^+(\Z)$ by 
\begin{align*}
(\gamma P)(X_1,X_2)
:=&
P(dX_1-bX_2, -cX_1+aX_2)\\
=&
P\left( (X_1, X_2) \cdot  {}^{t}\widetilde{\gamma} \, \right),  
\end{align*}
where $P \in \Z[X_1,X_2]$ and $\gamma \in M_2^+(\Z)$. 
For each integer $0 \leq \nu \leq n$, we set
\begin{align*}
e_{\nu} :=
X_1^{\nu}X_2^{n-\nu} \,\,\, \textrm{ and } \,\,\, 
e_{\nu}^{\flat} :=
(-1)^{n-\nu}
\binom{n}{\nu}
X_1^{n-\nu}X_2^{\nu}. 
\end{align*}
We then define submodules $\mcM_n$ and $\mcM_n^\flat$ of $\bZ[X_1, X_2]$ by 
\begin{align*}
\mcM_n :=
\bigoplus_{\nu=0}^{n}
\Z
e_{\nu} \,\,\, \textrm{ and } \,\,\, 
\mcM_n^{\flat} :=
\bigoplus_{\nu=0}^{n}
\Z
e_{\nu}^{\flat}. 
\end{align*}
The $\bZ$-modules $\mcM_n$ and $\mcM_n^{\flat}$ are closed under the left action of $M_2^+(\Z)$ on $\bZ[X_1, X_2]$. In particular, both $\mcM_n$ and $\mcM_n^{\flat}$ are left $\Gamma$-modules. 
We also define the pairing
\[
\brk{~,~}
\colon 
\mcM_n^{\flat} \times \mcM_n \ra \Z
\]
by 
\[
\brk{e_{\nu}^{\flat}, e_{\mu}}=\delta_{\nu, \mu}, 
\]
where $\delta_{\nu,\mu}$ is the Kronecker delta. 
The pairing  $\brk{~,~}$ is  perfect and $M_2^+(\Z)$-equivariant in the sense that for any polynomials $P \in \mcM_n^{\flat}$ and $Q \in \mcM_n$ and matrix $\gamma \in M_2^+(\Z)$, we have 
\[
\brk{P, \gamma Q}=\brk{\widetilde{\gamma} P, Q}.  
\]
Hence the pairing $\brk{~,~}$ induces an $M_2^+(\Z)$-equivariant isomorphism
\[
\mcM_n^{\flat} \stackrel{\sim}{\longrightarrow} \mcM_n^{\vee}:=\Hom_{\Z}(\mcM_n, \Z), 
m' \mapsto \left(m \mapsto \brk{m',m}\right). 
\]
Here the left action of $M_2^+(\bZ)$  on $\mcM_n^{\vee}=\Hom_{\Z}(\mcM_n, \Z)$ is given by 
\[
(\gamma \phi)(Q) = \phi(\tilde{\gamma}Q), 
\]
where $\phi \in \Hom_{\Z}(\mcM_n, \Z)$, $Q \in \mcM_n$, and $\gamma \in M_2^+(\bZ)$.

\begin{rmk}
The left actions of $M_2^+(\Z)$ on $\mcM_n$ and $\mcM_n^{\flat}$ are slightly different from the left actions used in Harder's book \cite[(1.57)]{CAG}. However, since 
\[
\begin{pmatrix}
    &-1\\1&
\end{pmatrix}^{-1}
\begin{pmatrix}
   a &b\\c&d
\end{pmatrix}
\begin{pmatrix}
    &-1\\1&
\end{pmatrix}
= 
\begin{pmatrix}
    d&-b\\-c&a
\end{pmatrix}, 
\]
they are isomorphic as left $M_2^+(\Z)$-modules. Therefore, there are no essential differences. 
\end{rmk}

\subsection{Hecke operators}\label{sec:Hecke operator}

Let $X \in \{\bbH, \bbH^\BS, \partial \bbH\}$ and let $\cM$ be a left $M_2^+(\bZ)$-module. 
In this subsection, we define the Hecke operators on $H_{\bullet}(\Gamma\bs X, \mcM)$ and $H^{\bullet}(\Gamma\bs X, \mcM)$, explicitly. 

For each prime number $p$, we have the following double coset decomposition: 
\[
\Gamma \backslash \Gamma \begin{pmatrix}p & \\ & 1\end{pmatrix} \Gamma = \Gamma \begin{pmatrix}p & \\ & 1\end{pmatrix} \sqcup \bigsqcup_{j=0}^{p-1} \Gamma \begin{pmatrix}1 & j \\ & p\end{pmatrix}. 
\]
Hence the endomorphism 
\[
\cM \longrightarrow \cM; m \mapsto \begin{pmatrix}
    p&0\\
    0&1
\end{pmatrix}
m + 
\sum_{j=0}^{p-1}
\begin{pmatrix}
    1&j\\
    0&p
\end{pmatrix}
m
\]
 induces an endomorphism of $\cM_\Gamma$. 
Similarly, the endomorphism 
\[
\cM \longrightarrow \cM; m \mapsto \widetilde{\begin{pmatrix}
    p&0\\
    0&1
\end{pmatrix}}
m + 
\sum_{j=0}^{p-1}
\widetilde{\begin{pmatrix}
    1&j\\
    0&p
\end{pmatrix}}
m
\]
 induces an endomorphism of $\cM^\Gamma$. 

\begin{dfn}\label{dfn:Hecke operator}
Let $p$ be a prime number. 
\begin{enumerate}
\item We define the Hecke operator $T_{p}$ at $p$ on $S_{\bullet}(X) \otimes \mcM$  by
\begin{align*}
T_{p}(\sigma\otimes P)
&:=
\begin{pmatrix}
    p&0\\
    0&1
\end{pmatrix}
\sigma\otimes
\begin{pmatrix}
    p&0\\
    0&1
\end{pmatrix}
P
+
\sum_{j=0}^{p-1}
\begin{pmatrix}
    1&j\\
    0&p
\end{pmatrix}
\sigma\otimes
\begin{pmatrix}
    1&j\\
    0&p
\end{pmatrix}
P
\end{align*}
for any simplex $\sigma \in S_{\bullet}(X)$ and element $P\in \mcM$. 
The operator $T_p$ induces operators on $\MS(X)\otimes \mcM$ and $H_{\bullet}(\Gamma \bs X, \mcM)$, etc., also written as $T_p$.

\item 
We define the Hecke operator $T_{p}'$ at $p$ on $\Hom_{\Z}(S_{\bullet}(X), \mcM)$ by
\begin{align*}
(\phi|T_p')(\sigma)
:=
\widetilde{
\begin{pmatrix}
    p&0\\
    0&1
\end{pmatrix}
}
\left(\phi(
\begin{pmatrix}
    p&0\\
    0&1
\end{pmatrix}
\sigma
)\right)
+
\sum_{j=0}^{p-1}
\widetilde{
\begin{pmatrix}
    1&j\\
    0&p
\end{pmatrix}
}
\left(\phi(
\begin{pmatrix}
    1&j\\
    0&p
\end{pmatrix}
\sigma
)\right)
\end{align*}
for any homomorphism $\phi \in \Hom_{\Z}(S_{\bullet}(X), \mcM_n^{\flat})$ and simplex $\sigma \in S_{\bullet}(X)$. 
The operator $T_p'$ induces operators on 
$\Hom_{\Z}(\MS(X), \mcM)$ and $H^{\bullet}(\Gamma\bs X, \mcM)$, etc., also written as $T_p'$.
\end{enumerate}
\end{dfn}

For later use, we also define auxiliary operators $U_p$ and $V_p$ on $S_\bullet(X) \otimes \mcM$ by
\begin{align*}
V_p(\sigma \otimes P)
&:=
\begin{pmatrix}
p&0\\
0&1
\end{pmatrix}
\sigma \otimes
\begin{pmatrix}
p&0\\
0&1
\end{pmatrix}
P,
\\
U_p(\sigma \otimes P)
&:=
\sum_{j=0}^{p-1}
\begin{pmatrix}
1&j\\
0&p
\end{pmatrix}
\sigma \otimes
\begin{pmatrix}
1&j\\
0&p
\end{pmatrix}
P, 
\end{align*}
so that $T_p:=V_p+U_p$.

\begin{lem}\label{lem:V_pU_p=p^n+1}
The composite $V_pU_p$ acts on $(S_\bullet(X) \otimes \cM_n)_\Gamma$, and we have $V_pU_p = p^{n+1}$ as operators on $(S_\bullet(X) \otimes \cM_n)_\Gamma$. 
\end{lem}

\begin{proof}
    Since diagonal matrices act trivially on $X$, we have 
    \begin{align*}
       V_pU_p(\sigma\otimes P) &=
\sum_{j=0}^{p-1}
\begin{pmatrix}
p&pj\\
0&p
\end{pmatrix}
\sigma \otimes
\begin{pmatrix}
p&pj\\
0&p
\end{pmatrix}
P
\\
&= p^{n}
\sum_{j=0}^{p-1}
\begin{pmatrix}
1&j\\
0&1
\end{pmatrix}
\left(\sigma\otimes
P\right)
    \end{align*}
for any simplex $\sigma \in S_\bullet(X)$ and polynomial $P \in \mcM_n$. 
Since $\begin{pmatrix}
1&j\\
0&1
\end{pmatrix} \in \Gamma$ for any integer $j$, we obtain this  lemma from this equality. 
\end{proof}

\subsection{Formal duality}\label{sec:formal-duality}

Let $X \in \{\bH, \bH^\BS, \partial\bH^\BS\}$. 
As explained in \S\ref{sec:modular symbol}, the homology and cohomology groups  can be computed as
\begin{align*}
    H_\bullet(\Gamma \bs X, \mcM_n) &\cong H_\bullet((S_\bullet(X) \otimes \cM_n)_\Gamma), 
    \\
    H^\bullet(\Gamma \bs X, \mcM_n^{\flat}) &\cong  H_\bullet(\Hom_\bZ(S_\bullet(X),\cM_n^\flat)^\Gamma). 
\end{align*}
The pairing $\brk{~,~} \colon \cM^{\flat} \times \cM \longrightarrow \bZ$ induces a pairing 
\begin{align*}
\brk{~,~}\colon 
\Hom_{\Z}(S_\bullet(X), \mcM_n^{\flat}) \times S_\bullet(X)\otimes \mcM_n \longrightarrow \Z, 
\end{align*}
which is computed as
\[
\big\langle\phi, \sigma \otimes P \big\rangle := \brk{\phi(\sigma), P}. 
\]
Note that for any matrix $\gamma \in M_2^+(\bZ)$, we have 
\[
\big\langle \widetilde{\gamma}\phi, \sigma \otimes P \big\rangle = \brk{\widetilde{\gamma}\phi(\gamma\sigma), P} = \brk{\phi(\gamma\sigma), {\gamma}P} = \big\langle \phi, {\gamma}(\sigma \otimes P) \big\rangle. 
\]
Therefore, we have 
\[
\big\langle \phi|T_p', \sigma \otimes P \big\rangle = \big\langle \phi, T_p(\sigma \otimes P) \big\rangle. 
\]
In particular, we obtain a Hecke-equivariant pairing 
\[
\brk{~,~}\colon 
H^\bullet(\Gamma \bs X, \mcM_n^{\flat})\times H_\bullet(\Gamma \bs X, \mcM_n) \longrightarrow \Z,  
\]
which induces an isomorphism
\[
H^{\bullet}(\Gamma \bs X, \mcM_n^{\flat})/\text{(torsion)}
\overset{\sim}{\longrightarrow}
\Hom_{\Z}(H_{\bullet}(\Gamma\bs X, \mcM_n), \Z). 
\]



\subsection{Eichler--Shimura homomorphism}\label{sec:Eichler--Shimura map} 

Let $M_{n+2}(\Gamma)$ denote the space of modular forms of weight $n+2$ and level $\Gamma = \mathrm{SL}_2(\Z)$. 
We define a homomorphism 
\begin{align*}
r\colon M_{n+2}(\Gamma) \longrightarrow \Hom_{\Z} (\MS(\bbH), \mcM_n^{\flat}\otimes \C)  
\end{align*}
by 
\begin{align*}
r(f)(\{\alpha, \beta\}) := \int_{\alpha}^{\beta} f(z) (X_1-zX_2)^n \, dz. 
\end{align*}
for any modular form $f \in M_{n+2}(\Gamma)$ and $\{\alpha, \beta\} \in \MS(\bbH)$. 
It is well-known that 
\[
r(M_{n+2}(\Gamma)) \subset \Hom_{\Z} (\MS(\bbH), \mcM_n^{\flat}\otimes \C)^\Gamma, 
\]
and the homomorphism $r$ induces an injective homomorphism (called Eichler--Shimura homomorphism)
\begin{align*}
r \colon M_{n+2}(\Gamma) \longhookrightarrow  H^1(Y, \mcM_n^{\flat})\otimes \C = H^1(Y^\BS, \mcM_n^{\flat})\otimes \C. 
\end{align*}
See \cite[Section 5.3]{Bel21} for example.

\begin{rmk}
The definition of the Eichler--Shimura homomorphism shows that, 
for any element $\sigma = \{\alpha, \beta\}\otimes P \in \MS(\bbH)\otimes \mcM_n$ and modular form $f \in M_{n+2}(\Gamma)$, the pairing $\brk{r(f), \sigma}$ can be computed as  
\[
\brk{r(f), \sigma} = \int_{\alpha}^{\beta} f(z)P(z,1) \, dz. 
\]
\end{rmk}

\begin{rmk}
For each prime number $p$, the double coset operator $T_p'' := \Gamma \begin{pmatrix}
    1&\\&p
\end{pmatrix}\Gamma$ acts on the space $M_{n+2}(\Gamma)$ of modular forms from the right by using the weight $n+2$ slash operator $|[~]_{n+2}$.  
One can easily show that 
\[
r(f|[\gamma]_{n+2}) = \widetilde{\gamma} \cdot r(f)
\]
for any matrix $\gamma \in M_2^+(\bZ)$. 
Hence the Eichler--Shimura homomorphism is Hecke-equivariant, that is, for all prime numbers $p$, we have 
\[
r(f|T_p'')=r(f)|T_p'. 
\]
In other words, our Hecke operator $T_p'$ coincides with the usual one via the  Eichler--Shimura homomorphism. 
\end{rmk}

The following lemma is well-known (see \cite[Theorem 5.3.27]{Bel21} for example). 

\begin{lem}\label{lem:boundary-H^1-is-eisenstein}
    The Eichler--Shimura homomorphism induces a Hecke-equivariant isomorphism 
    \[
    M_{n+2}(\Gamma)/S_{n+2}(\Gamma) \stackrel{\sim}{\longrightarrow} H^1(\partial Y^\BS, \cM_{n}^\flat) \otimes \bC. 
    \]
    Here $S_{n+2}(\Gamma)$ denotes the space of cusp forms of weight $n+2$ and level $\Gamma$. 
\end{lem}

\subsection{Definition of the Eisenstein  class}

 In this subsection, we define the Eisenstein class $\Eis_n$ and explain its basic properties. 

We put $i := \sqrt{-1}$ and let $\sigma_{n+1}(k)$ denote the sum-of-positive-divisors function, namely, $\sigma_{n+1}(k) := \sum_{0<d\mid k}d^{n+1}$. 
Let 
\begin{align*}
E_{n+2}(z) := 1 +\frac{2}{\zeta(-1-n)}\sum_{k=1}^{\infty} \sigma_{n+1}(k) e^{2\pi i kz} \in M_{n+2}(\Gamma) 
\end{align*}
denote the normalized holomorphic Eisenstein series of weight $n+2$.

\begin{lem}\label{lem:eisenstein-pairing-eigen-rational}\
\begin{enumerate}
\item 
For any element $\tau \in \bbH$, we have
\[
\brk{r(E_{n+2}), \{\tau, \tau+1\}\otimes e_0} = 1. 
\]
\item 
For any prime number $p$, we have
\[
r(E_{n+2})|{T_p'}
=
(1+p^{n+1})
r(E_{n+2}). 
\]
\end{enumerate}
\end{lem}

\begin{proof}
 Claim (1) follows from the fact that the constant term of $E_{n+2}$ is $1$, and claim (2) follows from the fact that $E_{n+2}|{T_p''}=(1+p^{n+1})E_{n+2}$ and the Eichler--Shimura homomorphism is Hecke-equivariant. 
\end{proof}

\begin{dfn}
    We define the Eisenstein class $\Eis_n \in H^1(Y^{\BS}, \cM_n^{\flat}) \otimes \bC$ by 
    \[
    \Eis_n := r\left(E_{n+2}\right). 
    \]
\end{dfn}

\begin{rmk}
    The method of defining the Eisenstein class $\Eis_n$ in this paper differs from the method in Harder's book \cite[\S3.3.6 (3.74)]{CAG}.  However, thanks to Lemma \ref{lem:eisenstein-pairing-eigen-rational}, they  coincide. 
\end{rmk}

\subsection{Main theorem}

\begin{prop}\label{prop:eisenstein-rational}
   The Eisenstein class $\Eis_n$ is rational, that is,  $\Eis_n \in H^1(Y^\BS, \cM_n^{\flat}) \otimes \bQ$. 
\end{prop}

This proposition is proved in Corollary \ref{cor:eisenstein class is rational}.

\begin{dfn}
    For any $\Gamma$-module $\cM$, we define 
    \[
    H^1_{\mathrm{int}}(Y^\BS, \cM) := \image\left(H^1(Y^{\BS}, \cM) \longrightarrow H^1(Y^\BS, \cM) \otimes \bQ \right). 
    \]
\end{dfn}

Thanks to Proposition \ref{prop:eisenstein-rational}, 
we define  the denominator $\Delta(\Eis_n)$ of the Eisenstein class $\Eis_n$ with respect to the integral structure $H^1_{\mathrm{int}}(Y^\BS, \cM^\flat_n)$ by 
\[
\Delta(\Eis_n)
:=
\min
\{
\Delta \in \Z_{>0} 
\mid
\Delta \Eis_n \in H^1_{\mathrm{int}}(Y^{\BS}, \mcM_n^{\flat})\}. 
\]
Then the following is the main theorem which we want to prove in the present paper.  

\begin{thm}[{\cite[Theorem 5.1.2]{CAG}}]\label{thm:main result}\
The denominator $\Delta(\Eis_n)$ of the Eisenstein class $\Eis_n$ is equal to the numerator of the special value $\zeta(-1-n)$ of the Riemann zeta function. 
 \end{thm}

\begin{rmk}\label{rem:differences integral structure (1)-ordinary (2)-de rham--betti}
    \
    \begin{itemize}
        \item[(1)] Since $H^1(Y^{\BS}, \cM_n^\flat) \otimes \bQ = H^1(Y^{\BS}, \cM_n) \otimes \bQ$, we have another integral structure $H^1_{\mathrm{int}}(Y^{\BS}, \cM_n)$, and one can consider another denominator $\Delta'(\Eis_n)$ of the Eisenstein class $\Eis_n$: 
        \[
        \Delta'(\Eis_n) := \min \{\Delta \in \Z_{>0} \mid \Delta \Eis_n \in H^1_{\mathrm{int}}(Y^{\BS}, \mcM_n)\}. 
        \]
        However,  we show in Lemma \ref{lem:ordinary-flat-integral=non-flat-ntegral} that $\Delta(\Eis_n) = \Delta'(\Eis_n)$. 
        \item[(2)] By using the $q$-expansion at the cusp $i \infty$, one can regard $M_{n+2}(\Gamma)$ as a submodule of $\bC[[q]]$, and we obtain the de Rham rational structure of $M_{n+2}(\Gamma)$ by $M_{n+2}(\Gamma) \cap \bQ[[q]]$. 
        The rationality of the critical values of the $L$-function associated with a cusp form is obtained by studying the gap between the de Rham and Betti rational structures via the Eichler--Shimura homomorphism $r$.  
        However, Proposition \ref{prop:eisenstein-rational} shows that the Eisenstein parts of the two rational structures coincide. 
        Moreover,  Theorem \ref{thm:main result} says that the Eisenstein parts of the two integral structures $M_{n+2}(\Gamma) \cap \bZ[[q]]$ and $H^1_{\mathrm{int}}(Y^{\BS}, \mcM_n^\flat)$ coincide, namely, 
        \[
        r(\bQ E_{n+2}(z) \cap \bZ[[q]]) = \bQ \Eis_{n} \cap H^1_{\mathrm{int}}(Y^{\BS}, \mcM_n^\flat)
        \]
        since $2$ is a regular prime. Cf. \cite[\S 1.1]{harder_can_we}. 
    \end{itemize}
\end{rmk}

\section{Construction of the cycle $\widetilde{T_p^{m}(C_{\nu}(\tau))}$}\label{sec:construction of the lift}
Fix a prime number $p$. 
In this section, we construct a special homology cycle 
\[
\widetilde{T_p^{m}(C_{\nu}(\tau))} \in H_1(Y^\BS, \mcM_{n} \otimes \bZ_{(p)})  
\]
that is used to compute the $p$-part of the denominator $\Delta(\Eis_n)$ of the Eisenstein class $\Eis_n$. 

For any integer $1 \leq \nu \leq n-1$ and element $\tau \in \bbH^{\BS}$, we set
\[
C_{\nu}(\tau)
:=
\left\{
-\frac{1}{\tau}, \tau
\right\}
\otimes
e_{\nu} \in \MS(\bbH^{\BS}) \otimes  \mcM_n,  
\]
where recall that $e_{\nu}=X_1^{\nu}X_2^{n-\nu}$.

\subsection{Computation of $T_p^m(C_{\nu}(\tau))$}

Recall also the operators  
\begin{align*}
    V_p(\sigma \otimes P)
&:=
\begin{pmatrix}
p&0\\
0&1
\end{pmatrix}
\sigma \otimes
\begin{pmatrix}
p&0\\
0&1
\end{pmatrix}
P,\\
U_p(\sigma \otimes P)
&:=
\sum_{j=0}^{p-1}
\begin{pmatrix}
1&j\\
0&p
\end{pmatrix}
\sigma \otimes
\begin{pmatrix}
1&j\\
0&p
\end{pmatrix}
P 
\end{align*}
on $S_\bullet(\bbH^{\BS}) \otimes \mcM_n$. 
We have $T_p=V_p+U_p$. 
For each integer $m \geq 0$, set 
\begin{align*}
W_m := \sum_{k=0}^m U_p^kV_p^{m-k}. 
\end{align*}
Note that $W_0$ is the identity map. 
For any (commutative) ring $R$ and cycle $C \in S_\bullet(\bbH^{\BS}) \otimes (\cM_n \otimes R)$, we denote by $[C]$ the image of $C$ in $(S_\bullet(\bbH^{\BS}) \otimes (\cM_n \otimes R))_{\Gamma}$.

\begin{lem}\label{lem:expand hecke}

Let  $m\geq 1$ be an integer and $C \in S_\bullet(\bH^\BS) \otimes \cM_n$. 
\begin{itemize}
    \item[(1)]  We have $T_p([W_m(C)]) = [W_{m+1}(C)] + p^{n+1}[W_{m-1}(C)]$ and $T_p([W_0(C)]) = [W_1(C)]$. 
    \item[(2)] We have 
\end{itemize}
\begin{align*}
T_p^m([C])  = \sum_{A=0}^{\lfloor m/2 \rfloor}
C(m-A,A)
p^{(n+1)A}
[W_{m-2A}(C)], 
\end{align*}
where $\lfloor m/2 \rfloor$ is  the greatest integer less than or equal to $m/2$ and
\[
C(A, B) := \left(1 - \frac{B}{A+1}\right)\binom{A+B}{B} 
=
\binom{A+B}{B} - \binom{A+B}{B-1}
\in \bZ.  
\]
Here we assume $\binom{a}{b}=0$ if $b<0$. 
\end{lem}
\begin{proof}
Claim (1) follows from the fact that $V_pU_p = p^{n+1}$ proved in Lemma \ref{lem:V_pU_p=p^n+1}. 
Let us prove claim (2). 
For notational simplicity, we put 
    \[
    w_m := [W_{m}(C)].   
    \]
    Then claim (1) shows that we can write 
    \[
    T_p^m([C]) = \sum_{k=0}^m a_{k}^{(m)}w_{m-k} 
    \]
    with $a_k^{(m)} \in \Z$ such that 
\begin{align*}
    a_{0}^{(m)} &= a_{0}^{(m-1)} = \cdots = a_{0}^{(1)} = a_{0}^{(0)} = 1, 
    \\
    a_{1}^{(m)} &= a_{1}^{(m-1)} = \cdots = a_{1}^{(1)} = 0, 
\\
a_{k}^{(m)} &= a_{k}^{(m-1)} + p^{n+1} a_{k-2}^{(m-1)}, \,\,\, \quad (2 \leq k \leq m-1)
\\
a_{m}^{(m)} &= p^{n+1} a_{m-2}^{(m-1)}. 
\end{align*}
Therefore, we have $a_{2k+1}^{(m)} = 0$ for any integer $0 \leq k \leq (m-1)/2$, and hence 
\[
T_p^m([C]) = \sum_{k=0}^{\lfloor m/2 \rfloor} a_{2k}^{(m)}w_{m-2k}. 
\]
Let us show that $a_{2k}^{(m)} = C(m-k, k)p^{k(n+1)}$ for any integer $0 \leq k \leq m/2$ by induction on $m$.  
    When $m=1$ or $k=0$, this claim is clear. 
    If $m > 1$ and $1 \leq k \leq (m-1)/2$, then the induction hypothesis shows that 
    \begin{align*}
a_{2k}^{(m)} &= a_{2k}^{(m-1)} + p^{n+1} a_{2k-2}^{(m-1)}
\\
&= C(m-1-k, k)p^{k(n+1)} +  C(m-k, k+1)p^{k(n+1)}
\\
&= C(m-k, k)p^{k(n+1)}. 
        \end{align*}
      Moreover, if $m$ is even, i.e., $m = 2t$, then we have 
\begin{align*}
    a_{m}^{(m)} &= p^{n+1}a^{(m-1)}_{m-2} 
    = C(t, t-1)p^{t(n+1)} 
    = \frac{1}{t+1}\frac{(2t)!}{(t!) (t!)}p^{t(n+1)} 
    = C(t,t)p^{t(n+1)}. 
\end{align*}
\end{proof}

By definition, we have 
\begin{align*}
W_{m-2A}(C_{\nu}(\tau))
&=
\sum_{k=0}^{m-2A}
U_p^{k}V_p^{m-2A-k}(C_{\nu}(\tau))\\
&=
\sum_{k=0}^{m-2A}
U_p^{k}
\begin{pmatrix}
    p^{m-2A-k}&0\\
    0&1
\end{pmatrix}
\left(\left\{
-\frac{1}{\tau}, \tau
\right\}
\otimes
e_{\nu}\right)\\
&=
\sum_{k=0}^{m-2A}
p^{(n-\nu)(m-2A-k)}
U_p^{k}
\lr{
\left\{
-\frac{p^{m-2A-k}}{\tau}, p^{m-2A-k}\tau
\right\}
\otimes
e_{\nu}
}. 
\end{align*}

\begin{dfn}
Take elements $\tau_0, \tau_1 \in \bbH^{\BS}$. 
For any integers $\nu$, $j$, and $k$ satisfying $1\leq \nu \leq n-1$ and $0\leq j \leq p^k-1$, we define  
\begin{align*}
C_{\nu,k,j}(\tau_0, \tau_1)
&:=
\begin{pmatrix}
1&j\\
0&p^k
\end{pmatrix}
\left\{-\frac{1}{\tau_0}, \tau_1\right\}
\otimes
\begin{pmatrix}
1&j\\
0&p^k
\end{pmatrix}
e_{\nu} \\
&=
\left\{
\frac{-\frac{1}{\tau_0}+j}{p^k}, 
\frac{\tau_1+j}{p^k}
\right\}
\otimes
\begin{pmatrix}
1&j\\
0&p^k
\end{pmatrix}
e_{\nu} \in \cM\cS(\bH^{\BS}) \otimes \cM_n. 
\end{align*}
Note that we have
\[
U_p^k\left(
\left\{-\frac{1}{\tau_0}, \tau_1\right\}\otimes e_{\nu}
\right)
=
\sum_{j=0}^{p^k-1}
C_{\nu,k,j}(\tau_0, \tau_1).  
\]
\end{dfn}

To sum up, we obtain the following corollary. 

\begin{cor}\label{cor:T_p^m}
    We have
    \begin{align*}
&T_p^m([C_{\nu}(\tau)])
\\
&=
\sum_{A=0}^{\lfloor m/2 \rfloor}
C(m-A, A)
p^{(n+1)A}
\sum_{k=0}^{m-2A}
p^{(n-\nu)(m-2A-k)}
\sum_{j=0}^{p^k-1}
\left[C_{\nu,k,j}
\lr{
\frac{\tau}{p^{m-2A-k}}, p^{m-2A-k}\tau
}\right]. 
\end{align*}
\end{cor}

\subsection{Computation of the boundary $\partial C_{\nu,k,j}\lr{\tau_0, \tau_1}$}

Next, we compute the boundary  $\partial C_{\nu,k,j}\lr{\tau_0, \tau_1}$.

\begin{dfn}\label{dfn:l,d,b}\ 
\begin{enumerate}
\item 
For any integers $j$ and $N$ with $p \nmid j$ and $N>0$, we denote by $d_N(j)$ and $b_N(j)$ the integers uniquely determined by
\[
1 \leq d_N(j)< p^N \,\,\, \textrm{ and } \,\,\, jd_N(j)-p^N b_N(j) = 1. 
\]
We also put $d_0(j) := 0$ and $b_0(j):=-1$ for any integer $j$. 
\item 
    For any integers $k$ and $j$, we set 
    \[
    l_k(j) := \min\{\ord_p(j), k\}. 
    \]
    Note that $l_k(0)=k$. 
    We also put $j' := j/p^{l_{k}(j)}$. 
\end{enumerate}
\end{dfn}

In the following, 
for integers $j$ and $k$ with $0\leq j \leq p^k-1$, 
we often write as 
\begin{align}\label{eqn:l,d,b}
\begin{split}
l &:= l_k(j), 
\\
j' &:= j/p^{l_k(j)}, 
\\
d &:= d_{k-l_k(j)}(j'), 
\\
b &:= b_{k-l_k(j)}(j')
\end{split}
\end{align}
for simplicity.

\begin{dfn}\label{dfn:E}
For any integers $\nu$, $j$, and $k$ with $1 \leq \nu \leq n-1$ and $0\leq j \leq p^k-1$, we define homogeneous polynomials $E_{\nu,k,j}^{(1)}$ and $E_{\nu,k,j}^{(0)}$ in $\mcM_n$ by
\begin{align*}
E_{\nu,k,j}^{(1)}(X_1,X_2)  &:= (p^kX_1-jX_2)^{\nu}X_2^{n-\nu}, 
\\
E_{\nu,k,j}^{(0)}(X_1, X_2) &:= (-1)^{\nu+1}(p^lX_2)^{\nu}(p^{k-l}X_1+dX_2)^{n-\nu}. 
\end{align*}
\end{dfn}

\begin{lem}\label{lem:boundary of C}
We have
\begin{align*}
\partial 
[C_{\nu,k,j}(\tau_0, \tau_1)]
=
\left[
\left\{
\frac{\tau_1+j}{p^k}
\right\}
\otimes
E_{\nu,k,j}^{(1)}
+
\left\{
\frac{p^l\tau_0-d}{p^{k-l}}
\right\}
\otimes
E_{\nu,k,j}^{(0)}
\right] 
\end{align*}
in $(S_{0}(\bbH^{\BS})\otimes \mcM_n)_{\Gamma}$. 
\end{lem}

\begin{proof}
By definition, we have
\begin{align*}
\partial 
[C_{\nu,k,j}(\tau_0, \tau_1)]
= 
\left[
\begin{pmatrix}
1&j\\
0&p^k
\end{pmatrix}
\left\{\tau_1\right\}
\otimes
\begin{pmatrix}
1&j\\
0&p^k
\end{pmatrix}
e_{\nu}
-
\begin{pmatrix}
1&j\\
0&p^k
\end{pmatrix}
\left\{-\frac{1}{\tau_0}\right\}
\otimes
\begin{pmatrix}
1&j\\
0&p^k
\end{pmatrix}
e_{\nu}
\right]. 
\end{align*}
The definition of $E_{\nu,k,j}^{(1)}$ shows that 
\[
\begin{pmatrix}
1&j\\
0&p^k
\end{pmatrix}
\left\{\tau_1\right\}
\otimes
\begin{pmatrix}
1&j\\
0&p^k
\end{pmatrix}
e_{\nu} 
= \left\{\frac{\tau_1+j}{p^k}\right\} \otimes E_{\nu,k,j}^{(1)}. 
\]
Moreover, we have 
\begin{align*}
\begin{pmatrix}
1&j\\
0&p^k
\end{pmatrix}
\left\{-\frac{1}{\tau_0}\right\}
\otimes
\begin{pmatrix}
1&j\\
0&p^k
\end{pmatrix}
e_{\nu} 
&=
\begin{pmatrix}
1&j\\
0&p^k
\end{pmatrix}
\begin{pmatrix}
0&-1\\
1&0
\end{pmatrix}
\left\{\tau_0\right\}
\otimes
\begin{pmatrix}
1&j\\
0&p^k
\end{pmatrix}
e_{\nu}
\\
&=
\begin{pmatrix}
j'&b\\
p^{k-l}&d
\end{pmatrix}
\begin{pmatrix}
p^{l}&-d\\
0&p^{k-l}
\end{pmatrix}
\left\{\tau_0\right\}
\otimes
\begin{pmatrix}
1&j\\
0&p^k
\end{pmatrix}
e_{\nu}.
\end{align*}
Since 
\begin{align*}
\left[
\begin{pmatrix}
j'&b\\
p^{k-l}&d
\end{pmatrix}
\begin{pmatrix}
p^{l}&-d\\
0&p^{k-l}
\end{pmatrix}
\left\{\tau_0\right\}
\otimes
\begin{pmatrix}
1&j\\
0&p^k
\end{pmatrix}
e_{\nu}\right]
= \left[
\begin{pmatrix}
p^{l}&-d\\
0&p^{k-l}
\end{pmatrix}
\left\{\tau_0\right\}
\otimes
\begin{pmatrix}
d&-b\\
-p^{k-l}&j'
\end{pmatrix}
\begin{pmatrix}
1&j\\
0&p^k
\end{pmatrix}
e_{\nu}
\right]
\end{align*}
and $\begin{pmatrix}
d&-b\\
-p^{k-l}&j'
\end{pmatrix}
\begin{pmatrix}
1&j\\
0&p^k
\end{pmatrix}
e_{\nu} =   -E_{\nu,k,j}^{(0)}$, 
we obtain 
\begin{align*}
\left[\begin{pmatrix}
1&j\\
0&p^k
\end{pmatrix}
\left\{-\frac{1}{\tau_0}\right\}
\otimes
\begin{pmatrix}
1&j\\
0&p^k
\end{pmatrix}
e_{\nu} \right] 
= - \left[\left\{\frac{p^l\tau_0-d}{p^{k-l}}\right\} \otimes E_{\nu,k,j}^{(0)}\right]. 
\end{align*}

\end{proof}

\subsection{A cycle $\widetilde{T_p^m(C_{\nu}(\tau))}$ in $\cM\cS(\bH^\BS) \otimes (\cM_n \otimes \bQ)$}

In this subsection, we construct a cycle $\widetilde{T_p^m(C_{\nu}(\tau))}$ in $\cM\cS(\bH^\BS) \otimes (\cM_n \otimes \bQ)$ which is a lift of $T_p^m(C_{\nu}(\tau))$ and is $p$-adically integral for any sufficiently large integer $m$. 

\subsubsection{Bernoulli polynomials}
Since a key tool for constructing the cycle $\widetilde{T_p^m(C_{\nu}(\tau))}$ is the Bernoulli polynomials, we  briefly recall the basic properties of the Bernoulli polynomials.

Let $t$ be a non-negative integer. 
    We denote by $B_t(x)$ the $t$-th Bernoulli polynomial and by 
    \[
    B_t := B_t(0)
    \]
    the $t$-th Bernoulli number.  
For notational simplicity, we put 
\[
\widetilde{B}_t(x) := \frac{B_t(x)-B_t}{t}. 
\]
In this paper, we use the following well-known facts without any notice: 
\begin{gather*}
B_t(x) = \sum_{\mu=0}^{t}\binom{t}{\mu}B_{t-\mu}x^\mu, 
\\
\int_{x}^{x+1} B_t(x) \, dx= x^{t}, 
\\
\int_{a}^x B_t(x) \, dx= \frac{B_{t+1}(x) - B_{t+1}(a)}{t+1}, 
\\
\ord_p(B_t) \geq -1. 
\end{gather*}
The last fact is called the von Staudt--Clausen theorem. 
We note that the second and third facts imply that
\[
 \sum_{j=0}^{x-1}j^t =\widetilde{B}_{t+1}(x). 
\]

\subsubsection{$P^\dag$ and $P^\ddag$}
Set 
\[
\cM_{n, (p)}
:=
\bigoplus_{\mu = 0}^{n}\bZ_{(p)} X_1^{\mu}X_2^{n-\mu} 
\,\,\, \textrm{ and } \,\,\, 
\cM_{n, (p)}^{\circ}
:=
\bigoplus_{\mu = 0}^{n-1}\bZ_{(p)} X_1^{\mu}X_2^{n-\mu}, 
\]
and let 
\[
\dag\colon \mcM_{n, (p)}^{\circ} \otimes \Q \ra \cM_{n,(p)}\otimes \Q; 
P \mapsto P^{\dag}
\]
be the $\Q$-linear map defined by 
\[
(X_1^{\mu}X_2^{n-\mu})^{\dag}
:=
X_2^n\frac{B_{\mu+1}(X_1/X_2)-B_{\mu+1}}{\mu+1} = X_2^n \widetilde{B}_{\mu+1}(X_1/X_2).  
\]
For any integer $\mu \in \{0, \ldots, n-1\}$, we have $1 + \ord_p(\mu+1) \leq n$, and hence $p^nX_2^n \widetilde{B}_{\mu+1}(X_1/X_2) \in \cM_{n, (p)}$. 
This fact shows that 
\[
\dag(\cM_{n, (p)}^{\circ}) \subset \frac{1}{p^n}\cM_{n, (p)}. 
\]
Similarly, let 
\[
\ddag \colon \cM_{n, (p)}^{\circ} \otimes \Q \ra \cM_n \otimes \Q; 
P \mapsto P^{\ddag}
\]
be the $\Q$-linear map defined by 
\[
(X_1^{\mu}X_2^{n-\mu})^{\ddag}
:=
X_2^n\frac{(X_1/X_2)^{\mu+1}-B_{\mu+1}}{\mu+1}. 
\]
The following lemma follows from the definitions of $P^\dag$ and $P^\ddag$. 

\begin{lem}\label{lem:dagger-relations}
For any polynomial $P(X_1,X_2) \in \cM_n^{\circ} \otimes \Q$, we have 
\begin{align*}
    P^{\dag}(X_1+X_2, X_2)-P^{\dag}(X_1,X_2) &= P(X_1,X_2),  
\\
\int_{x}^{x+1}
P^{\dag}(z,1) \, dz
&=
P^{\ddag}(x,1),  
\\
\int_{x_1}^{x_2}
P(z,1) \, dz
&=
P^{\ddag}(x_2,1)
-
P^{\ddag}(x_1,1). 
\end{align*}
\end{lem}

\subsubsection{Definitions of polynomials $P^{(1)}_{\nu, k, j}$ and $P^{(0)}_{\nu, k, j}$}

\begin{dfn}\label{dfn:P}
For any integers $\nu$, $j$, and $k$ with $1 \leq \nu \leq n-1$ and $0\leq j \leq p^k-1$, we define polynomials $P_{\nu,k,j}^{(1)}$ and $P_{\nu,k,j}^{(0)}$ in $\mcM_n\otimes \Q$ by
\begin{align*}
P_{\nu,k,j}^{(1)} &:= E_{\nu,k,j}^{(1)\dag}, 
\\
P_{\nu,k,j}^{(0)} &:= E_{\nu,k,j}^{(0)\dag}. 
\end{align*}
\end{dfn}

\begin{lem}\label{lem:P-coeff order}
For each  $i \in \{0,1\}$,  we have $P_{\nu, k, j}^{(i)} \in p^{\min\{0, k-n\}}\cM_{n, (p)}$. 
\end{lem}
\begin{proof}
This lemma follows from the facts that  
$\dag(\cM_{n,(p)}^\circ) \subset p^{-n}\cM_{n,(p)}$ and  
$(X_2^n)^\dag = X_1X_2^{n-1}$ and that the coefficient of $X_1^{\mu}X_2^{n-\mu}$ in $E_{\nu,k,j}^{(i)}$ is divided by $p^k$ if $\mu\geq 1$. 
\end{proof}

\subsubsection{Definition of the cycle $\widetilde{T_p^m(C_{\nu}(\tau))}$}

Now we can define $\widetilde{C}_{\nu,k,j}(\tau_0, \tau_1)$ and $\widetilde{T_p^m(C_{\nu}(\tau))}$.

\begin{dfn}
Let $\tau_0, \tau_1 \in \bbH^{\BS}$. 
For any integers $\nu$, $k$, and $j$ with $1 \leq \nu \leq n-1$ and $0\leq j \leq p^k-1$, we define an element 
$\widetilde{C}_{\nu,k,j}(\tau_0, \tau_1) \in \cM\cS(\bH^{\BS}) \otimes (\cM_n \otimes \Q)$ by 
\begin{align*}
\widetilde{C}_{\nu,k,j}(\tau_0, \tau_1) :=
C_{\nu,k,j}(\tau_0, \tau_1) -
\left\{\frac{\tau_1+j}{p^k}, \frac{\tau_1+j}{p^k}+1\right\}
\otimes
P_{\nu,k,j}^{(1)} - \left\{\frac{p^{l}\tau_0-d}{p^{k-l}}, \frac{p^{l}\tau_0-d}{p^{k-l}}+1\right\}
\otimes P_{\nu,k,j}^{(0)}. 
\end{align*}
We also put 
\[
\widetilde{C}_{\nu,k,j}^{\rm int}(\tau_0, \tau_1) := p^{\max\{0, n-k\}}\widetilde{C}_{\nu,k,j}(\tau_0, \tau_1). 
\]
Note that $\widetilde{C}_{\nu,k,j}^{\rm int}(\tau_0, \tau_1) \in \cM\cS(\bH^{\BS}) \otimes \cM_{n, (p)}$ by Lemma \ref{lem:P-coeff order}. 

\end{dfn}

\begin{lem}\label{lem:partial=0-nu-k-j}
We have 
\[
\partial [\widetilde{C}_{\nu,k,j}^{\rm int}(\tau_0, \tau_1)] = 0 \,\,\, \textrm{ in } \,\,\,  (S_0(\bbH^{\BS})\otimes \mcM_{n, (p)})_{\Gamma}. 
\]
In particular, $[\widetilde{C}_{\nu,k,j}^{\rm int}(\tau_0, \tau_1)]$ defines a homology class 
\[
[\widetilde{C}_{\nu,k,j}^{\rm int}(\tau_0, \tau_1)] \in H_1(Y^{\BS}, \mcM_{n, (p)}). 
\]
\end{lem}
\begin{proof}
This lemma follows from Definitions \ref{dfn:E} and \ref{dfn:P} and Lemmas \ref{lem:boundary of C} and \ref{lem:dagger-relations}. 
\end{proof}

\begin{lem}\label{lem:indep-tau-C-nu-k-j}
    The homology class $[\widetilde{C}_{\nu,k,j}^{\rm int}(\tau_0, \tau_1)]$ is independent of the choices of $\tau_0$ and $\tau_1$. 
\end{lem}

\begin{proof}
Let $\tau_0'$ and $\tau_1'$ be another pair of points in $\bH^{\BS}$. 
We will prove that
\[
[\widetilde{C}_{\nu,k,j}^{\rm int}(\tau_0, \tau_1)]
= 
[\widetilde{C}_{\nu,k,j}^{\rm int}(\tau_0', \tau_1')]
\,\,\,  
\text{ in } \,\,\,  
H_1(Y^{\BS}, \mcM_{n, {(p)}}). 
\]
It suffices to construct an element $h \in S_2(\bbH^{\BS})\otimes \mcM_{n, {(p)}}$ such that
\[
\partial [h] 
=
{[\widetilde{C}_{\nu,k,j}^{\rm int}(\tau_0, \tau_1)] - [\widetilde{C}_{\nu,k,j}^{\rm int}(\tau_0', \tau_1')]} \,\,\,    \text{ in } \,\,\,  (\MS(\bbH^{\BS})\otimes \mcM_{n, {(p)}})_{\Gamma}.  
\]
{
For notational simplicity, set $q:=p^{\max \{0, n-k\}}$. 
}
First, since $\bbH^{\BS}$ is simply connected, there exist elements 
$h_1, h_2, h_3 \in S_{2}(\bbH^{\BS})$  such that
\footnotesize
\begin{align*}
\partial h_1
&=
\left\{
\frac{-\frac{1}{\tau_0}+j}{p^k}, 
\frac{\tau_1+j}{p^k}
\right\}
-
\left\{
\frac{-\frac{1}{\tau_0'}+j}{p^k}, 
\frac{\tau_1'+j}{p^k}
\right\}
-\left\{
\frac{\tau_1'+j}{p^k}
\frac{\tau_1+j}{p^k}
\right\}
+
\left\{
\frac{-\frac{1}{\tau_0'}+j}{p^k}, 
\frac{-\frac{1}{\tau_0}+j}{p^k}, 
\right\},
\\
\partial h_2 &=
\left\{\frac{\tau_1+j}{p^k}, \frac{\tau_1+j}{p^k}+1\right\} -
\left\{\frac{\tau_1'+j}{p^k}, \frac{\tau_1'+j}{p^k}+1\right\} - \left\{\frac{\tau_1'+j}{p^k}+1, \frac{\tau_1+j}{p^k}+1\right\} + \left\{\frac{\tau_1'+j}{p^k}, \frac{\tau_1+j}{p^k}\right\}, 
\\
\partial h_3
&=
\left\{\frac{p^l\tau_0-d}{p^{k-l}}, \frac{p^l\tau_0-d}{p^{k-l}}+1\right\}
- \left\{\frac{p^l\tau_0'-d}{p^{k-l}}, \frac{p^l\tau_0'-d}{p^{k-l}}+1\right\} 
- \left\{\frac{p^l\tau_0'-d}{p^{k-l}}+1, \frac{p^l\tau_0-d}{p^{k-l}}+1\right\} + \left\{\frac{p^l\tau_0'-d}{p^{k-l}}, \frac{p^l\tau_0-d}{p^{k-l}}\right\}. 
\end{align*}
\normalsize
Then we see that 
\[
h := {
{q}\left(h_1 \otimes
\begin{pmatrix}
    1&j\\
    0&p^k
\end{pmatrix}
e_{\nu} - h_2 \otimes P_{\nu,k,j}^{(1)} - h_3 \otimes P_{\nu,k,j}^{(0)}\right)}
\]
satisfies the desired property. 
Indeed, we have
\begin{align*}
\partial h = {{q}
\left( \widetilde{C}_{\nu,k,j}(\tau_0, \tau_1) - \widetilde{C}_{\nu,k,j}(\tau_0', \tau_1') -B^{(1)}-B^{(0)}\right)}, 
\end{align*}
where
\footnotesize
\begin{align*}
B^{(1)}
&:=
\left\{\frac{\tau_1'+j}{p^k}\frac{\tau_1+j}{p^k}\right\} \otimes
\begin{pmatrix}
    1&j\\
    0&p^k
\end{pmatrix}
e_{\nu}
-
\left\{\frac{\tau_1'+j}{p^k}+1, \frac{\tau_1+j}{p^k}+1\right\}\otimes P_{\nu,k,j}^{(1)}
+ \left\{\frac{\tau_1'+j}{p^k}, \frac{\tau_1+j}{p^k}\right\}\otimes P_{\nu,k,j}^{(1)}, 
\\
B^{(0)}
&:= -\left\{\frac{-\frac{1}{\tau_0'}+j}{p^k}, \frac{-\frac{1}{\tau_0}+j}{p^k}, \right\}\otimes
\begin{pmatrix}
    1&j\\
    0&p^k
\end{pmatrix}
e_{\nu}
- \left\{\frac{p^l\tau_0'-d}{p^{k-l}}+1, \frac{p^l\tau_0-d}{p^{k-l}}+1\right\} \otimes P_{\nu,k,j}^{(0)}
+ \left\{\frac{p^l\tau_0'-d}{p^{k-l}}, \frac{p^l\tau_0-d}{p^{k-l}}\right\}\otimes P_{\nu,k,j}^{(0)}. 
\end{align*}
\normalsize
Then the same computation as in the proof of Lemma \ref{lem:boundary of C} shows that
\footnotesize
\begin{align*}
{q}B^{(1)}
&\equiv
\left\{
\frac{\tau_1'+j}{p^k}
\frac{\tau_1+j}{p^k}
\right\}
\otimes {q}
\lr{
E_{\nu,k,j}^{(1)}(X_1, X_2)
-
P_{\nu,k,j}^{(1)}(X_1+X_2, X_2)
+
P_{\nu,k,j}^{(1)}(X_1, X_2)
}
=0, 
\\
{q}B^{(0)}
&\equiv
\left\{
\frac{p^l\tau_0'-d}{p^{k-l}}+1, 
\frac{p^l\tau_0-d}{p^{k-l}}+1
\right\}
\otimes {q}
\lr{
E_{\nu,k,j}^{(0)}(X_1, X_2)
-
P_{\nu,k,j}^{(0)}(X_1+X_2, X_2)
+
P_{\nu,k,j}^{(0)}(X_1, X_2)
}
=0,  
\end{align*}
\normalsize
where $\equiv$ means that it is an equality in the $\Gamma$-coinvariant $(\MS(\bbH^{\BS})\otimes \mcM_{n, {(p)}})_{\Gamma}$. 
Hence we obtain {$\partial [h] = [\widetilde{C}_{\nu,k,j}^{\rm int}(\tau_0, \tau_1)] - [\widetilde{C}_{\nu,k,j}^{\rm int}(\tau_0', \tau_1')]$}. 
\end{proof}

\begin{dfn}\label{def:widetilde-T_p^m-C-nu-tau}
For any integer $m > 0$ and element $\tau \in \bbH^{\BS}$, we define a cycle $\widetilde{T_p^m(C_{\nu}(\tau))} \in \MS(\bbH^{\BS}) \otimes (\cM_{n} \otimes \Q)$ by 
\begin{align*}
&\widetilde{T_p^m(C_{\nu}(\tau))}\\
&:=
\sum_{A=0}^{\lfloor m/2 \rfloor}
C(m-A,A)
p^{(n+1)A}
\sum_{k=0}^{m-2A}
p^{(n-\nu)(m-2A-k)}
\sum_{j=0}^{p^k-1}
\widetilde{C}_{\nu,k,j}\left(\frac{\tau}{p^{m-2A-k}}, p^{m-2A-k}\tau \right). 
\end{align*}
\end{dfn}

\begin{dfn}
 For any integer $\nu \in \{0, \ldots, n\}$, we define a homology class \[
 [C_\nu] \in H_1(Y^{\rm BS}, \partial Y^{\rm BS}, \cM_{n})\] to be the element represented by the cycle 
 \[
 C_\nu := \{0, i\infty\} \otimes e_\nu, 
 \]
 where $i\infty \in \partial \bbH^{\BS}$ is a point such that 
     \[
     i \infty:=\lim_{t\in \Rpos, t \to \infty} it. 
     \]
 \end{dfn}

\begin{lem}\label{lem:T_p-cycle-integral}
Let $1 \leq \nu \leq n-1$ be an integer. 
\begin{itemize}
    \item[(1)] $\partial [\widetilde{T_p^m(C_{\nu}(\tau))}] = 0$ in $(S_0(\bbH^{\BS})\otimes \mcM_n\otimes\Q)_{\Gamma}$. 
    \item[(2)] If $m \geq n$, then we have 
    $\widetilde{T_p^m(C_{\nu}(\tau))} \in \cM\cS(\bH) \otimes \cM_{n,(p)}$. 
    Hence $\widetilde{T_p^m(C_{\nu}(\tau))}$ defines a $p$-integral homology class $[\widetilde{T_p^m(C_{\nu}(\tau))}]$ in $H_1(Y^{\BS}, \cM_{n,(p)})$ which is independent of the choice of $\tau$. 
    \item[(3)] If $m \geq n$, then the image of the homology class $[\widetilde{T_p^m(C_{\nu}(\tau))}]$ under the homomorphism $H_1(Y^{\BS}, \cM_{n,(p)}) \longrightarrow H_1(Y^{\rm BS}, \partial Y^{\rm BS}, \cM_{n,(p)})$  is 
    $T_p^m([C_{\nu}])$. 
\end{itemize}
\end{lem}

\begin{proof}
Claim (1) follows from Lemma \ref{lem:partial=0-nu-k-j}. 
Let us show claim (2). 
By Lemma \ref{lem:P-coeff order}, we have $p^{\max\{0, n-k\}}\widetilde{C}_{\nu, k,j}(\tau_0, \tau_1) \in \cM\cS(\bH^\BS) \otimes \cM_{n, (p)}$. 
Moreover, since $1 \leq \nu \leq n-1$ and $2 \leq n$, we have 
\[
(n+1)A + (n - \nu)(m-2A-k) \geq (n+1)A + m-2A-k \geq m-k
\]
for any non-negative integer $A$. 
This fact shows that $\widetilde{T_p^m(C_{\nu}(\tau))} \in \cM\cS(\bH) \otimes \cM_{n,(p)}$ if $m \geq n$. 
It follows from Lemma \ref{lem:indep-tau-C-nu-k-j} that the homology class $[\widetilde{T_p^m(C_{\nu}(\tau))}]$ does not depend on the choice of $\tau$. 
    Claim (3) follows from the definition of $\widetilde{T_p^m(C_{\nu}(\tau))}$ and Lemma \ref{lem:expand hecke}. 

\end{proof}

\section{Period}\label{sec:Period}

The aim of this section is to compute the value
    \begin{align*}
        \brk{\Eis_{n}, \widetilde{T_p^{m!}(C_{\nu}(\tau))}} 
\end{align*} 
and its $p$-adic limit as $m \to \infty$. 
In this paper, we often consider the $p$-adic limit, and hence the symbol 
$\lim_{m \to \infty}$ will always mean the $p$-adic limit. 
The following is the main result of this section.

\begin{thm}\label{thm:Tp^m and Eis}
   For any integer $\nu \in \{1, \ldots, n-1\}$, we have 
   \begin{align*}
   \lim_{m \to \infty}\brk{\Eis_n, \widetilde{T_p^{m!}(C_{\nu}(\tau))}} 
   =
   (1-p^{n+1}) \left(\frac{1}{1-p^{n+1}}\frac{\zeta(-\nu)\zeta(\nu-n)}{\zeta(-1-n)} - 
   \frac{\zeta(-\nu)}{1-p^{n-\nu}}
   - 
   \frac{\zeta(\nu-n)}{1-p^{\nu}} 
   \right). 
   \end{align*}
\end{thm}

In fact, we show in \S\ref{sec:proof of thm Tp^m and Eis} that 
\begin{align}\label{eqn:limit of Wm}
\begin{split}
\lim_{m \to \infty}\sum_{k=0}^{m}
       p^{(n-\nu)(m-k)}
       &\sum_{j=0}^{p^k-1}\brk{\Eis_n, \widetilde{C}_{\nu,k,j}\left(\frac{\tau}{p^{m-k}}, p^{m-k}\tau \right)} \\
       &= 
       \frac{1}{1-p^{n+1}}\frac{\zeta(-\nu)\zeta(\nu-n)}{\zeta(-1-n)} - 
   \frac{\zeta(-\nu)}{1-p^{n-\nu}}- 
   \frac{\zeta(\nu-n)}{1-p^{\nu}}.  
\end{split}
\end{align}
Hence Theorem \ref{thm:Tp^m and Eis} follows from \eqref{eqn:limit of Wm} and the following lemma.

\begin{lem}\label{lem:limit lemma}
Suppose that the $p$-adic limit 
\[
\lim_{m \to \infty}\sum_{k=0}^{m}
       p^{(n-\nu)(m-k)}
       \sum_{j=0}^{p^k-1}\brk{\Eis_n, \widetilde{C}_{\nu,k,j}\left(\frac{\tau}{p^{m-k}}, p^{m-k}\tau \right)} 
\]
exists. 
    We then have 
    \begin{align*}
       \lim_{m \to \infty}
       \brk{\Eis_n, \widetilde{T_p^{m!}(C_{\nu}(\tau))}} 
       = 
       (1-p^{n+1})\lim_{m \to \infty}\sum_{k=0}^{m}
       p^{(n-\nu)(m-k)}
       \sum_{j=0}^{p^k-1}\brk{\Eis_n, \widetilde{C}_{\nu,k,j}\left(\frac{\tau}{p^{m-k}}, p^{m-k}\tau \right)}. 
    \end{align*}
\end{lem}
\begin{proof}
For notational simplicity, we put 
\[
\msW^{(m)} := \sum_{k=0}^{m}
       p^{(n-\nu)(m-k)}
       \sum_{j=0}^{p^k-1}\brk{\Eis_n, \widetilde{C}_{\nu,k,j}\left(\frac{\tau}{p^{m-k}}, p^{m-k}\tau \right)}
\,\,\, \textrm{ and } \,\,\,  \msW := \lim_{m \to \infty}\msW^{(m)}. 
 \]
We then have 
\[
\brk{\Eis_n, \widetilde{T_p^{m!}(C_{\nu}(\tau))}}  = 
\sum_{A=0}^{m!/2}
C(m!-A,A)
p^{(n+1)A}
\msW^{(m!-2A)}. 
\]
Take a positive integer $Q$. Then there is a positive integer $r$ such that 
$\msW^{(s)} - \msW \in p^Q\bZ_p$ for any integer $s \geq r$. 
Hence we have 
\begin{align*}
    &\sum_{A=0}^{\lfloor m/2 \rfloor}C(m-A,A)p^{(n+1)A}(\msW^{(m-2A)} - \msW) 
\\
&\equiv   
\sum_{A=\lfloor (m-r)/2 \rfloor + 1}^{\lfloor m/2 \rfloor}C(m-A,A)p^{(n+1)A}(\msW^{(m-2A)} - \msW) \pmod{p^Q\bZ_p}. 
\end{align*}
The sequence $(\msW^{(m)} - \msW)_{m= 0}^{\infty}$ is bounded in $\bQ_p$, and hence for any sufficiently large integer  $m$, we have 
\[
\sum_{A=\lfloor (m-r)/2 \rfloor + 1}^{\lfloor m/2 \rfloor}C(m-A,A)p^{(n+1)A}(\msW^{(m-2A)} - \msW)  \in p^Q\bZ_p. 
\]
This implies that 
\[
\lim_{m \to \infty}
\brk{\Eis_n, \widetilde{T_p^{m!}(C_{\nu}(\tau))}}  = \lim_{m \to \infty}
\sum_{A=0}^{m!/2}
C(m!-A,A)
p^{(n+1)A}
\msW. 
\]
Since $C(m!-A,A) = \binom{m!}{A} - \binom{m!}{A-1}$ (note that $\binom{m}{-1} = 0$), we have 
\begin{align*}
\sum_{A=0}^{ m!/2}C(m!-A,A) p^{(n+1)A} 
&= \sum_{A=0}^{ m!/2}\binom{m!}{A} p^{(n+1)A} - \sum_{A=0}^{ m!/2 - 1}\binom{m!}{A} p^{(n+1)(A+1)} 
\\
&= (1-p^{n+1}) \sum_{A=0}^{ m!/2  - 1}\binom{m!}{A} p^{(n+1)A} 
+ \binom{m!}{ m!/2} p^{(n+1) m!/2 }. 
\end{align*}
Since 
\begin{align*}
\sum_{A=0}^{ m!/2  - 1}\binom{m!}{A} p^{(n+1)A}  &\equiv \sum_{A=0}^{ m! }\binom{m!}{A} p^{(n+1)A}  \pmod{p^{m!/2}}
\\
&= (1+p^{n+1})^{m!}, 
\end{align*}
we obtain that 
\[
\lim_{m \to \infty}\sum_{A=0}^{ m!/2}C(m!-A,A) p^{(n+1)A} = (1-p^{n+1}) \lim_{m \to \infty} (1+p^{n+1})^{m!} = 1-p^{n+1}, 
\]
which completes the proof. 
\end{proof}

Therefore, it remains to prove \eqref{eqn:limit of Wm}, and this will be done in  Proposition \ref{prop:main term limit}.

\subsection{$\brk{\Eis_n, \widetilde{C}_{\nu,k,j}\left(\tau_0, \tau_1\right)}$}

We start with computing the value $\brk{\Eis_n, \widetilde{C}_{\nu,k,j}\left(\tau_0, \tau_1\right)}$. 
In this subsection, we fix integers $k$ and $j$ with $0 \leq j \leq p^k - 1$. Recall that  
\[
l := l_k(j) \,\,\, \textrm{ and } \,\,\, d := d_{k - l_k(j)}(j') = d_{k - l_k(j)}(j/p^{l_k(j)}) 
\]
are taken as in Definition \ref{dfn:l,d,b} and \eqref{eqn:l,d,b} and  that 
\begin{align*}
\widetilde{C}_{\nu,k,j}
\left(
\tau_0, \tau_1
\right)
=
\left\{
\frac{-\frac{1}{\tau_0}+j}{p^k}, 
\frac{\tau_1+j}{p^k}
\right\}
\otimes
(p^kX_1-jX_2)^{\nu}X_2^{n-\nu}
&-
\left\{
\frac{\tau_1+j}{p^k}, 
\frac{\tau_1+j}{p^k}+1
\right\}
\otimes
P_{\nu,k,j}^{(1)}\\
&-
\left\{
\frac{p^l\tau_0-d}{p^{k-l}}, 
\frac{p^l\tau_0-d}{p^{k-l}}+1
\right\}
\otimes
P_{\nu,k,j}^{(0)}. 
\end{align*}
Hence we have
\small
\begin{align}
\begin{split}\label{eq:paring-C_nu-k-j and Eis}
    \brk{\Eis_n, \widetilde{C}_{\nu,k,j}
\left(
\tau_0, \tau_1
\right)}
&=
\int_{\frac{-\frac{1}{\tau_0}+j}{p^k}}^
{\frac{\tau_1+j}{p^k}}
E_{n+2}(z)
(p^kz-j)^{\nu}\, dz\\
&\quad -
\int_{\frac{\tau_1+j}{p^k}}^{\frac{\tau_1+j}{p^k}+1}
E_{n+2}(z)
P_{\nu,k,j}^{(1)}(z,1)\, dz
-
\int_{\frac{p^l\tau_0-d}{p^{k-l}}}^{\frac{p^l\tau_0-d}{p^{k-l}}+1}
E_{n+2}(z)
P_{\nu,k,j}^{(0)}(z,1)\, dz\\
&=
\frac{1}{p^k}
\int_{-\frac{1}{\tau_0}}^{\tau_1}
E_{n+2}\left(\frac{z+j}{p^k}\right)
z^{\nu}\, dz\\
&\quad -
\int_{\frac{\tau_1+j}{p^k}}^{\frac{\tau_1+j}{p^k}+1}
E_{n+2}(z)
P_{\nu,k,j}^{(1)}(z,1)\, dz
-
\int_{\frac{p^l\tau_0-d}{p^{k-l}}}^{\frac{p^l\tau_0-d}{p^{k-l}}+1}
E_{n+2}(z)
P_{\nu,k,j}^{(0)}(z,1)\, dz. 
\end{split}
\end{align}
\normalsize
The definition of the Eisenstein series $E_{n+2}$ shows that 
\begin{align*}
E_{n+2}\left(\frac{z+j}{p^k}\right)-1
&=
O(e^{-2\pi\im(z)}) 
\quad \text{ for } \im(z) \geq 1, 
\\
E_{n+2}\left(\frac{z+j}{p^k}\right)-\frac{p^{l(n+2)}}{z^{n+2}}
&=
O
\left(
\frac{p^{l(n+2)}}{|z|^{n+2}}e^{-2\pi p^{2l-k}/\im(z)}
\right) 
\quad \text{ for } \im(z)\leq 1, 
\end{align*}
where $f(z) = O(g(z))$ means that there is a constant $C$ which does not depend on $k$ and $j$ such that $|f(z)| \leq Cg(z)$. 
Set
\begin{align*}
\msL_{k,j}(s)
:=
\int_{0}^{\infty}
\lr{
E_{n+2}\lr{\frac{iy+j}{p^k}}-1
}
y^{s}dy. 
\end{align*}

We have the following Lemma \ref{lem:meromorphic continuation} and Proposition \ref{prop:C_nu,k,j and Eis-pairing computation}, whose proofs will be given in \S\ref{sec:Computation of the first term} and \S\ref{sec:proof of prop cpmputation pairing}, respectively.

\begin{lem}\label{lem:meromorphic continuation}
The function 
$\msL_{k,j}(s)$  converges for $\re(s)>n+1$, and continued to a meromorphic function on $\C$. 
Moreover, it has at most simple poles at $s=-1$ and $s=n+1$. In particular, $\msL_{k,j}(s)$ is holomorphic at $s=\nu$ for any integer $1 \leq \nu \leq n-1$. 
\end{lem}

\begin{prop}\label{prop:C_nu,k,j and Eis-pairing computation}
We have
\[
\brk{\Eis_n, \widetilde{C}_{\nu,k,j}
\left(
\tau_0, \tau_1
\right)}
=
\frac{i^{\nu+1}}{p^k}
\msL_{k,j}(\nu)
-
E_{\nu,k,j}^{(1) \ddag}
\lr{
\frac{j}{p^k}, 1
}
-
E_{\nu,k,j}^{(0) \ddag}
\lr{
-\frac{d}{p^{k-l}}, 1
}. 
\]
\end{prop}

Note that since $\partial [\widetilde{C}_{\nu,k,j}\left(\tau_0, \tau_1\right)]=0$ by Lemma \ref{lem:partial=0-nu-k-j}, the value $\brk{\Eis_n, \widetilde{C}_{\nu,k,j}
\left(
\tau_0, \tau_1
\right)}$ 
does not depend on the choices of $\tau_0$ and $\tau_1$. 
Therefore, in the following we take $\tau_0=it_0$ and $\tau_1=it_1$ for $t_0, t_1 \in \Rpos$.

\subsubsection{Computation of the first term of \eqref{eq:paring-C_nu-k-j and Eis}}\label{sec:Computation of the first term}

Here we compute the first term of \eqref{eq:paring-C_nu-k-j and Eis}: 
\[
\frac{1}{p^k}
\int_{\frac{i}{t_0}}^{it_1}
E_{n+2}\left(\frac{z+j}{p^k}\right)
z^{\nu} \, dz. 
\]
This integral is transformed as follows:
\begin{align*}
&\frac{1}{p^k}
\int_{\frac{i}{t_0}}^{it_1}
E_{n+2}\left(\frac{z+j}{p^k}\right)
z^{\nu}\, dz
\\
&= \frac{1}{p^k}
\int_{\frac{i}{t_0}}^{\infty}
\left(
E_{n+2}\left(\frac{z+j}{p^k}\right)-1
\right)
z^{\nu} \, dz
+
\frac{1}{p^k}
\int_{\frac{i}{t_0}}^{it_1}
z^{\nu}\, dz
-
\frac{1}{p^k}
\int_{it_1}^{\infty}
\left(
E_{n+2}\left(\frac{z+j}{p^k}\right)-1
\right)
z^{\nu} \,
dz\\
&=
\frac{i^{\nu+1}}{p^k}
\left\{
\int_{\frac{1}{t_0}}^{\infty}
\lr{
E_{n+2}\lr{\frac{iy+j}{p^k}}-1
}
y^{\nu}\,dy
+
\int_{\frac{1}{t_0}}^{t_1}
y^{\nu}\,dy
-
\int_{t_1}^{\infty}
\lr{
E_{n+2}\lr{\frac{iy+j}{p^k}}-1
}
y^{\nu}\,dy
\right\}. 
\end{align*}
Set
\begin{align*}
\msS_{k,j}(t_0,t_1,s)
&:=
\int_{\frac{1}{t_0}}^{\infty}
\lr{
E_{n+2}\lr{\frac{iy+j}{p^k}}-1
}
y^{s} \,dy
+
\int_{\frac{1}{t_0}}^{t_1}
y^{s}\,dy, 
\\
\msR_{k,j}^{(1)}(t_1,s)
&:=
\int_{t_1}^{\infty}
\lr{
E_{n+2}\lr{\frac{iy+j}{p^k}}-1
}
y^{s}\,dy, 
\\
\msR_{k,j}^{(0)}(t_0,s)
&:=
\int_{0}^{\frac{1}{t_0}}
\lr{
E_{n+2}\lr{
\frac{iy+j}{p^k}
}
-
\frac{p^{l(n+2)}}{(iy)^{n+2}}
}
y^s\,dy. 
\end{align*}
Then we have 
\begin{align*}
\frac{1}{p^k}
\int_{\frac{i}{t_0}}^{it_1}
E_{n+2}\left(\frac{z+j}{p^k}\right)
z^{\nu}\,dz 
=
\frac{i^{\nu+1}}{p^k}
\lr{
\msS_{k,j}(t_0,t_1,\nu)
-
\msR_{k,j}^{(1)}(t_1,\nu)}.
\end{align*}
Now, we see that 
\begin{itemize}
    \item the first terms of $\msS_{k,j}(t_0,t_1,s)$ and $\msR^{(i)}_{k,j}(t_i,s)$ converge for all $s \in \C$, 
    \item the second term of $\msS_{k,j}(t_0,t_1,s)$ is meromorphic and has at most simple pole at $s=-1$. 
\end{itemize}
In addition, we also see that 
\begin{align*}
\msS_{k,j}(t_0,t_1,s)
&=
\msL_{k,j}(s)
-
\msR_{k,j}^{(0)}(t_0,s)
-
\int_{0}^{\frac{1}{t_0}}
\frac{p^{l(n+2)}}{(iy)^{n+2}}y^s\,dy
+
\int_{0}^{t_1}
y^s\,dy \\
&=
\msL_{k,j}(s)
-
\msR_{k,j}^{(0)}(t_0,s)
-
\frac{p^{l(n+2)}}{i^{n+2}}
\frac{1}{s-n-1}
\frac{1}{t_0^{s-n-1}}
+
\frac{1}{s+1}
t_1^{s+1}. 
\end{align*} 
In particular, all of these functions are meromorphically continued to $s \in \C$ and are holomorphic at $s=\nu$. This proves Lemma \ref{lem:meromorphic continuation}, and moreover, we get the following. 

\begin{lem}
    \begin{align*}
&\frac{1}{p^k}
\int_{\frac{i}{t_0}}^{it_1}
E_{n+2}\left(\frac{z+j}{p^k}\right)
z^{\nu}\,dz\\
&=
\frac{i^{\nu+1}}{p^k}
\left\{
\msL_{k,j}(\nu)
+
\frac{t_1^{\nu+1}}{\nu+1}
+
\frac{p^{l(n+2)}}{i^{n+2}}
\frac{t_0^{n-\nu+1}}{n-\nu+1}
-
\msR_{k,j}^{(1)}(t_1,\nu)
-
\msR_{k,j}^{(0)}(t_0,\nu)
\right\}. 
\end{align*}
\end{lem}

\subsubsection{Computation of the second and the third terms of \eqref{eq:paring-C_nu-k-j and Eis}}\label{sec: Computation of the second and the third terms}

Here we compute
\[
\int_{\frac{\tau_1+j}{p^k}}^{\frac{\tau_1+j}{p^k}+1}
E_{n+2}(z)
P_{\nu,k,j}^{(1)}(z,1) \, dz
+
\int_{\frac{p^l\tau_0-d}{p^{k-l}}}^{\frac{p^l\tau_0-d}{p^{k-l}}+1}
E_{n+2}(z)
P_{\nu,k,j}^{(0)}(z,1) \, dz. 
\]
We put 
\begin{align*}
\msT_{\nu,k,j}(\tau_0, \tau_1)
:=
\int_{\frac{\tau_1+j}{p^k}}^{\frac{\tau_1+j}{p^k}+1}
\lr{
E_{n+2}(z)-1
}
P_{\nu,k,j}^{(1)}(z,1) \, dz
+
\int_{\frac{p^l\tau_0-d}{p^{k-l}}}^{\frac{p^l\tau_0-d}{p^{k-l}}+1}
\lr{
E_{n+2}(z)-1
}
P_{\nu,k,j}^{(0)}(z,1)\, dz. 
\end{align*}
Note that $\msT_{\nu,k,j}(\tau_0, \tau_1) \to 0$ as $\tau_0, \tau_1 \to \infty$. 
On the other hand by Lemma \ref{lem:dagger-relations}, we have
\begin{align*}
\int_{\frac{\tau_1+j}{p^k}}^{\frac{\tau_1+j}{p^k}+1}
P_{\nu,k,j}^{(1)}(z,1) \, dz
&=
E_{\nu,k,j}^{(1)\ddag}
\lr{
\frac{\tau_1+j}{p^k}, 1
}\\
&=
E_{\nu,k,j}^{(1)\ddag}
\lr{
\frac{j}{p^k}, 1
}
+
\int_{\frac{j}{p^k}}^{\frac{\tau_1+j}{p^k}}
E_{\nu,k,j}^{(1)}(z,1)
dz\\
&=
E_{\nu,k,j}^{(1)\ddag}
\lr{
\frac{j}{p^k}, 1
}
+
\int_{\frac{j}{p^k}}^{\frac{\tau_1+j}{p^k}}
(p^kz-j)^{\nu} \, 
dz\\
&=
E_{\nu,k,j}^{(1)\ddag}
\lr{
\frac{j}{p^k}, 1
}
+
\frac{1}{p^k}
\frac{(it_1)^{\nu+1}}{\nu+1}. 
\end{align*}
Similarly we have
\begin{align*}
\int_{\frac{p^l\tau_0-d}{p^{k-l}}}^{\frac{p^l\tau_0-d}{p^{k-l}}+1}
P_{\nu,k,j}^{(0)}(z,1)\,dz
=
E_{\nu,k,j}^{(0)\ddag}
\lr{
-\frac{d}{p^{k-l}}, 1
}
+
(-1)^{\nu+1}
\frac{p^{l(n+2)}}{p^k}
\frac{(it_0)^{n-\nu+1}}{n-\nu+1}. 
\end{align*}

\subsubsection{Proof of Proposition \ref{prop:C_nu,k,j and Eis-pairing computation}}\label{sec:proof of prop cpmputation pairing}

By combining the computations in subsections \ref{sec:Computation of the first term} and \ref{sec: Computation of the second and the third terms}, we find
\begin{align*}
&\brk{\Eis_n, \widetilde{C}_{\nu,k,j}
\left(
\tau_0, \tau_1
\right)}\\
&=
\frac{i^{\nu+1}}{p^k}
\msL_{\nu,k,j}(\nu)
+
\frac{1}{p^k}
\frac{(it_1)^{\nu+1}}{\nu+1}
+
(-1)^{\nu-1}
\frac{p^{l(n+2)}}{p^k}
\frac{(it_0)^{n-\nu+1}}{n-\nu+1}
\\
&\quad-
E_{\nu,k,j}^{(1)\ddag}
\lr{
\frac{j}{p^k}, 1
}
-
\frac{1}{p^k}
\frac{(it_1)^{\nu+1}}{\nu+1}
-
E_{\nu,k,j}^{(0)\ddag}
\lr{
-\frac{d}{p^{k-l}}, 1
}
-
(-1)^{\nu+1}
\frac{p^{l(n+2)}}{p^k}
\frac{(it_0)^{n-\nu+1}}{n-\nu+1}\\
&\quad -
\frac{i^{\nu+1}}{p^k}
\msR_{k,j}^{(1)}(t_1,\nu)
-
\frac{i^{\nu+1}}{p^k}
\msR_{k,j}^{(0)}(t_0,\nu)
-\msT_{k,j}(\tau_0,\tau_1), 
\end{align*}
and hence
\begin{align*}
\brk{\Eis_n, \widetilde{C}_{\nu,k,j}
\left(
\tau_0, \tau_1
\right)} 
&=
\frac{i^{\nu+1}}{p^k}
\msL_{k,j}(\nu)
-
E_{\nu,k,j}^{(1)\ddag}
\lr{
\frac{j}{p^k}, 1
}
-
E_{\nu,k,j}^{(0)\ddag}
\lr{
-\frac{d}{p^{k-l}}, 1
}\\
&\quad -
\frac{i^{\nu+1}}{p^k}
\msR_{k,j}^{(1)}(t_1,\nu)
-
\frac{i^{\nu+1}}{p^k}
\msR_{k,j}^{(0)}(t_0,\nu)
-\msT_{k,j}(\tau_0,\tau_1). 
\end{align*}
Since the value $\brk{\widetilde{C}_{\nu,k,j}
\left(
\tau_0, \tau_1
\right), 
\Eis_n}$ does not depend on $t_0$ and $t_1$, we can take the limit $t_0, t_1 \to \infty$. 
Then the last three terms vanish and we obtain the desired identity. 
\qed

\subsection{Summation over $j$}

In this subsection, we compute the sum 
\begin{align*}
\sum_{j=0}^{p^k-1}
\brk{\Eis_n, \widetilde{C}_{\nu,k,j}
\left(
\tau_0, \tau_1
\right)} 
=
\sum_{j=0}^{p^k-1}
\left\{
\frac{i^{\nu+1}}{p^k}
\msL_{k,j}(\nu)
-
E_{\nu,k,j}^{(1)\ddag}
\lr{
\frac{j}{p^k}, 1
}
-
E_{\nu,k,j}^{(0)\ddag}
\lr{
-\frac{d}{p^{k-l}}, 1
}
\right\}. 
\end{align*}
We keep using the abbreviation 
\[
l := l_k(j) \,\,\, \textrm{ and } \,\,\, d := d_{k - l_k(j)}(j') = d_{k - l_k(j)}(j/p^{l_k(j)}) 
\]
that are actually depending on $k$ and $j$. 
Recall that $\widetilde{B}_t(x) = (B_t(x) - B_t)/t$. 

\begin{lem}\label{lem:L, E^1, E^0}
\
\begin{enumerate}
\item 
We have
\begin{align*}
\frac{i^{\nu+1}}{p^k}\sum_{j=0}^{p^k-1}
\msL_{k,j}(\nu)
=
\frac{(1-p^{(n+1)(k+1)})-(1-p^{(n+1)k})p^{n-\nu}}{1-p^{n+1}}
\frac{\zeta(-\nu)\zeta(\nu-n)}{\zeta(-1-n)}. 
\end{align*}
\item 
We have
\begin{align*}
\sum_{j=0}^{p^k-1}
E_{\nu,k,j}^{(1)\ddag}
\lr{
\frac{j}{p^k}, 1
}
=
\frac{(-1)^{\nu}}{(\nu+1)}
\frac{1}{p^k}
\widetilde{B}_{\nu + 2}(p^k)
+
\sum_{j=0}^{p^k-1}
E_{\nu,k,j}^{(1), \ddag}
\lr{
0, 1
}. 
\end{align*}
\item 
We have
\begin{align*}
&\sum_{j=0}^{p^k-1}
E_{\nu,k,j}^{(0)\ddag}
\lr{
-\frac{d}{p^{k-l}}, 1
}\\
&=
\frac{(-1)^{\nu}}{n-\nu+1}
\frac{1}{p^k}
\sum_{l'=0}^{k-1}
p^{l'(\nu+1)}
\lr{
\widetilde{B}_{n-\nu+2}(p^{k-l'})
-p^{n-\nu+1}
\widetilde{B}_{n-\nu+2}(p^{k-l'-1})
} 
+
\sum_{j=0}^{p^k-1}
E_{\nu,k,j}^{(0)\ddag}
\lr{
0, 1
}. 
\end{align*}
\end{enumerate}
\end{lem}

\begin{proof}
\item[(1)]
Recall that 
\begin{align*}
\msL_{k,j}(s)
=
\int_{0}^{\infty}
\lr{
E_{n+2}\lr{\frac{iy+j}{p^k}}-1
}
y^{s}dy. 
\end{align*}
Hence we have 
\begin{align*}
\frac{1}{p^k}\sum_{j=0}^{p^k-1}
\msL_{k,j}(s)
&=
\frac{1}{p^k}\sum_{j=0}^{p^k-1}
\frac{2}{\zeta(-1-n)}
\sum_{\mu=1}^{\infty}
\sigma_{n+1}(\mu)
e^{\frac{2\pi i \mu j}{p^k}}
\int_{0}^{\infty}
e^{-\frac{2\pi \mu y}{p^k}}y^{s+1}\frac{dy}{y}\\
&=
\frac{1}{p^k}\frac{2}{\zeta(-1-n)}\sum_{j=0}^{p^k-1}
\sum_{\mu=1}^{\infty}
\sigma_{n+1}(\mu)
e^{\frac{2\pi i \mu j}{p^k}}
\frac{p^{k(s+1)}}{(2\pi \mu)^{s+1}}\Gamma(s+1)\\
&=
\frac{2}{\zeta(-1-n)}
\frac{\Gamma(s+1)p^{ks}}{(2\pi)^{s+1}}
\sum_{\mu=1}^{\infty}
\left(\sum_{j=0}^{p^k-1}e^{\frac{2\pi i \mu j}{p^k}}\right)
\frac{\sigma_{n+1}(\mu)}{\mu^{s+1}}\\
&=
\frac{2\Gamma(s+1)}{\zeta(-1-n)(2\pi)^{s+1}}
\sum_{\mu=1}^{\infty}
\frac{\sigma_{n+1}(p^k \mu)}{\mu^{s+1}}. 
\end{align*}
For notational simplicity, we put 
\[
\cE_p(s) := 1-p^{-s}. 
\]
We then have 
\begin{align*}
    \sum_{a=0}^{\infty} \frac{\sigma_{n+1}(p^{k+a})}{p^{a(s+1)}} &= \sum_{a=0}^{\infty}\frac{1}{p^{a(s+1)}}\frac{1-p^{(k+a+1)(n+1)}}{1-p^{n+1}}
    \\
    &= \frac{\cE_p(1+s)^{-1} - p^{(k+1)(n+1)}\cE_p(s-n)^{-1}}{1-p^{n+1}}
    \\
    &= \frac{\cE_p(s-n) - p^{(k+1)(n+1)}\cE_p(1+s)}{1-p^{n+1}}\cE_p(1+s)^{-1}\cE_p(s-n)^{-1}. 
\end{align*}
Hence the well-known relation that   
$\zeta(1+s)\zeta(s-n) = \sum_{a=1}^{\infty}\sigma_{n+1}(a) a^{-(s+1)}$ implies that 
\begin{align*}
    \frac{1}{p^k}\sum_{j=0}^{p^k-1}
\msL_{k,j}(s)
&= \frac{2\Gamma(s+1)}{\zeta(-1-n)(2\pi)^{s+1}}\frac{\cE_p(s-n) - p^{(k+1)(n+1)}\cE_p(1+s)}{1-p^{n+1}}  \zeta(s+1)\zeta(s-n). 
\end{align*}
By setting $s=\nu$ and using the functional equation of the Riemann zeta function (see \cite[p.29]{Hida93} for example), we find
\begin{align*}
\frac{i^{\nu+1}}{p^k}\sum_{j=0}^{p^k-1}
\msL_{k,j}(\nu) 
&=
\frac{(\cE_p(\nu-n) - p^{(k+1)(n+1)}\cE_p(\nu+1))}{1-p^{n+1}}
\frac{\zeta(-\nu)\zeta(\nu-n)}{\zeta(-1-n)}
\\
&= 
\frac{(1-p^{(n+1)(k+1)})-(1-p^{(n+1)k})p^{n-\nu}}{1-p^{n+1}}
\frac{\zeta(-\nu)\zeta(\nu-n)}{\zeta(-1-n)}. 
\end{align*}

\item[(2)] 
By using Lemma \ref{lem:dagger-relations}, we find
\begin{align*}
\sum_{j=0}^{p^k-1}
E_{\nu,k,j}^{(1), \ddag}
\lr{
\frac{j}{p^k}, 1
}
&=
\sum_{j=0}^{p^k-1}
\lr{
E_{\nu,k,j}^{(1), \ddag}
\lr{
\frac{j}{p^k}, 1
}
-
E_{\nu,k,j}^{(1), \ddag}
\lr{
0, 1
}
}
+
\sum_{j=0}^{p^k-1}
E_{\nu,k,j}^{(1), \ddag}
\lr{
0, 1
}\\
&=
\sum_{j=0}^{p^k-1}
\int_{0}^{\frac{j}{p^k}}
E_{\nu,k,j}^{(1)}(z,1)dz
+
\sum_{j=0}^{p^k-1}
E_{\nu,k,j}^{(1), \ddag}
\lr{
0, 1
}\\
&=
\sum_{j=0}^{p^k-1}
\int_{0}^{\frac{j}{p^k}}
(p^kz-j)^{\nu}dz
+
\sum_{j=0}^{p^k-1}
E_{\nu,k,j}^{(1), \ddag}
\lr{
0, 1
}\\
&=
\frac{(-1)^{\nu}}{(\nu+1)}
\frac{1}{p^k}
\sum_{j=0}^{p^k-1}
j^{\nu+1}
+
\sum_{j=0}^{p^k-1}
E_{\nu,k,j}^{(1), \ddag}
\lr{
0, 1
}. 
\end{align*}
Therefore, claim (2) follows from the fact that $\sum_{j=0}^{p^k-1}j^{\nu+1} = \widetilde{B}_{\nu+2}(p^k)$.

\item[(3)] 
Lemma \ref{lem:dagger-relations} shows that 
\begin{align*}
E_{\nu,k,j}^{(0)\ddag}
\lr{
-\frac{d}{p^{k-l}}, 1
}
&=
-(-1)^{\nu}p^{l\nu}
\int_{0}^{-\frac{d}{p^{k-l}}}
(p^{k-l}z+d)^{n-\nu}
dz
+
E_{\nu,k,j}^{(0)\ddag}
\lr{
0, 1
}\\
&=
\frac{(-1)^{\nu}}{n-\nu+1}
\frac{1}{p^k}
p^{l(\nu+1)}
d^{n-\nu+1}
+
E_{\nu,k,j}^{(0)\ddag}
\lr{
0, 1
}. 
\end{align*}
Since $d_0(0) = 0$, $n-\nu+1 \geq 2$, and the map $a \mapsto d_{N}(a)$ induces a permutation on the set $\{1 \leq a < p^{N} \mid (a,p^N) = 1 \}$,  we find
\begin{align*}
\frac{1}{p^k}
\sum_{j=0}^{p^k-1} 
p^{l(\nu+1)}
d^{n-\nu+1}
&= \frac{1}{p^k}
\sum_{l'=0}^{k-1}
p^{l'(\nu+1)}
\sum_{\substack{d'=1\\(d',p^{k-l'})=1}}^{p^{k-l'}-1}
d'^{n-\nu+1}
\\
&= \frac{1}{p^k}
\sum_{l'=0}^{k-1}
p^{l'(\nu+1)}
\lr{
\widetilde{B}_{n-\nu+2}(p^{k-l'})
-p^{n-\nu+1}
\widetilde{B}_{n-\nu+2}(p^{k-l'-1})
}. 
\end{align*}
\end{proof}

\begin{lem}\label{lem:sum-detail_bel-E^0-E^1}\
\begin{enumerate}

\item 
We have
\begin{align*}
\sum_{j=0}^{p^k-1}
E_{\nu,k,j}^{(1)\ddag}(0,1)
= (-1)^{\nu+1}
\sum_{\mu=0}^{\nu}
\binom{\nu}{\mu}
(-1)^{\mu}p^{k\mu}
\frac{B_{\mu+1}}{\mu+1}
\widetilde{B}_{\nu-\mu+1}(p^k). 
\end{align*}

\item 
We have
\begin{multline*}
\sum_{j=0}^{p^k-1}
E_{\nu,k,j}^{(0)\ddag}
\lr{
0, 1
}
= (-1)^{\nu}p^{k\nu}\frac{B_{n-\nu+1}}{n-\nu+1} +
(-1)^{\nu}
\sum_{\mu=0}^{n-\nu}
\begin{pmatrix}
n-\nu\\
\mu
\end{pmatrix}
p^{k\mu}
\frac{B_{\mu+1}}{\mu+1}
\sum_{l'=0}^{k-1}
p^{l'(\nu-\mu)}\\
\times
\lr{
\widetilde{B}_{n-\nu-\mu+1}(p^{k-l'}) - p^{n-\nu-\mu}
\widetilde{B}_{n-\nu-\mu+1}(p^{k-l'-1})}.  
\end{multline*}
\end{enumerate}
\end{lem}

\begin{proof}
\item[(1)] Note that $(X_1^{\mu}X_2^{n-\mu})^{\ddag}(0,1) = -B_{\mu+1}/(\mu+1)$. 
Since $E_{\nu,k,j}^{(1)}(X_1,X_2) = (p^kX_1-jX_2)^{\nu}X_2^{n-\nu}$, we have 
\[
E_{\nu,k,j}^{(1)\ddag}
\lr{
0, 1
}
= \sum_{\mu=0}^{\nu}\binom{\nu}{\mu}p^{k\mu}\left(-\frac{B_{\mu+1}}{\mu+1}\right)(-j)^{\nu-\mu}. 
\]
Hence we have 
\begin{align*}
\sum_{j=0}^{p^k-1}
E_{\nu,k,j}^{(1)\ddag}(0,1)
= 
(-1)^{\nu+1}
\sum_{\mu=0}^{\nu}
\binom{\nu}{\mu}
(-1)^{\mu}p^{k\mu}
\frac{B_{\mu+1}}{\mu+1}
\widetilde{B}_{\nu-\mu+1}(p^k).
\end{align*}

\item[(2)] Since $E_{\nu,k,j}^{(0)}(X_1,X_2) = (-1)^{\nu+1}(p^{l}X_2)^{\nu}(p^{k-l}X_1+dX_2)^{n-\nu}$,  we have 
\[
E_{\nu,k,j}^{(0)\ddag}(0,1)
= (-1)^{\nu+1}p^{l\nu}\sum_{\mu=0}^{n-\nu}\binom{n-\nu}{\mu}p^{\mu(k-l)}\left(-\frac{B_{\mu+1}}{\mu+1}\right)d^{n-\nu-\mu}. 
\]
First note that in the case where $j=0$, since $d=d_0(0)=0$, we find
\[
E_{\nu,k,0}^{(0)\ddag}(0,1) = (-1)^{\nu}p^{k\nu}\frac{B_{n-\nu+1}}{n-\nu+1}.  
\]
Then by using the same argument as in the proof of Lemma \ref{lem:L, E^1, E^0}(3), we also obtain
\begin{align*}
\sum_{j=1}^{p^k-1} E_{\nu,k,j}^{(0)\ddag}(0,1) &= 
(-1)^{\nu}\sum_{l'=0}^{k-1}\sum_{\substack{d' = 1 \\ (d',p^{k-l'})=1}}^{p^{k-l'}-1}p^{l'\nu}\sum_{\mu=0}^{n-\nu}\binom{n-\nu}{\mu}p^{\mu(k-l')}\frac{B_{\mu+1}}{\mu+1}d'^{n-\nu-\mu}
\\
&= 
(-1)^{\nu}\sum_{l'=0}^{k-1} p^{l'\nu}\sum_{\mu=0}^{n-\nu}\binom{n-\nu}{\mu}p^{\mu(k-l')}\frac{B_{\mu+1}}{\mu+1}
 \\
&\quad\quad\quad \times 
\lr{
\widetilde{B}_{n-\nu-\mu+1}(p^{k-l'}) - p^{n-\nu-\mu}
\widetilde{B}_{n-\nu-\mu+1}(p^{k-l'-1})}.
\end{align*}
This completes the proof. 
\end{proof}

\subsection{Summation over $k$ and the $p$-adic limits}\label{sec:proof of thm Tp^m and Eis}

In this subsection, we compute the value 
\[
\msW^{(m)} := \sum_{k=0}^m
p^{(n-\nu)(m-k)}
\sum_{j=0}^{p^k-1}
\brk{\Eis_n, \widetilde{C}_{\nu,k,j}
\left(
\frac{\tau}{p^{m-k}}, p^{m-k}\tau
\right)}
\]
and its $p$-adic limit as $m \to \infty$. 
This enables us to complete the proof of Theorem \ref{thm:Tp^m and Eis}.

We keep the notation in the previous sections. 
Proposition \ref{prop:C_nu,k,j and Eis-pairing computation} shows that 
\[
\brk{\Eis_n, \widetilde{C}_{\nu,k,j}
\left(
\tau_0, \tau_1
\right)}
=
\frac{i^{\nu+1}}{p^k}
\msL_{k,j}(\nu)
-
E_{\nu,k,j}^{(1) \ddag}
\lr{
\frac{j}{p^k}, 1
}
-
E_{\nu,k,j}^{(0) \ddag}
\lr{
-\frac{d}{p^{k-l}}, 1
}, 
\]
and hence
\[
\msW^{(m)} = \sum_{k=0}^m
p^{(n-\nu)(m-k)}
\sum_{j=0}^{p^k-1}
\left(
\frac{i^{\nu+1}}{p^k}
\msL_{k,j}(\nu)
-
E_{\nu,k,j}^{(1) \ddag}
\lr{
\frac{j}{p^k}, 1
}
-
E_{\nu,k,j}^{(0) \ddag}
\lr{
-\frac{d}{p^{k-l}}, 1
}\right). 
\]
We set 
\begin{align*}
    \msW_1^{(m)} &:= \sum_{k=0}^m p^{(n-\nu)(m-k)}\sum_{j=0}^{p^k-1}\frac{i^{\nu+1}}{p^k}\msL_{k,j}(\nu), 
    \\
    \msW_2^{(m)} &:= \sum_{k=0}^m p^{(n-\nu)(m-k)} \sum_{j=0}^{p^k-1}
E_{\nu,k,j}^{(1) \ddag}
\lr{
\frac{j}{p^k}, 1
}, 
    \\
    \msW_3^{(m)} &:= \sum_{k=0}^m
p^{(n-\nu)(m-k)}
\sum_{j=0}^{p^k-1}
E_{\nu,k,j}^{(0) \ddag}
\lr{
-\frac{d}{p^{k-l}}, 1
}, 
\end{align*}
so that we have $\msW^{(m)} = \msW^{(m)}_1 - \msW^{(m)}_2 - \msW^{(m)}_3$.

\begin{lem}\label{lem:FT1}
We have 
\[
\lim_{m \to \infty} \msW_1^{(m)} = \frac{1}{1-p^{n+1}}\frac{\zeta(-\nu)\zeta(\nu-n)}{\zeta(-1-n)}. 
\]
\end{lem}
\begin{proof}
Lemma \ref{lem:L, E^1, E^0}(1) shows that 
    \begin{align*}
\msW_1^{(m)} 
&= 
\frac{1}{1-p^{n+1}}\frac{\zeta(-\nu)\zeta(\nu-n)}{\zeta(-1-n)}
\sum_{k=0}^m 
p^{(n-\nu)(m-k)}
((1-p^{(n+1)(k+1)})-(1-p^{(n+1)k})p^{n-\nu}). 
\end{align*}
Moreover, we have 
    \begin{align*}
\sum_{k=0}^m 
p^{(n-\nu)(m-k)}
&
((1-p^{(n+1)(k+1)}) - (1-p^{(n+1)k})p^{n-\nu}) 
\\
&= 
p^{(n-\nu) m}(1-p^{n-\nu})
\sum_{k=0}^{m}p^{-(n-\nu) k} 
-  
p^{(n-\nu) m}(p^{n+1} - p^{n-\nu})
\sum_{k=0}^{m}p^{(\nu+1) k}
\\
&= 
p^{(n-\nu) m}(1-p^{n-\nu}) 
\frac{1-p^{-(n-\nu)(m+1)}}{1-p^{-(n-\nu)}} 
-  
p^{(n-\nu) m}(p^{n+1} - p^{n-\nu}) 
\frac{1-p^{(\nu+1)(m+1)}}{1-p^{\nu+1}}
\\
&= 
1-p^{(n-\nu)(m+1)} 
-  
p^{(n-\nu) m}(p^{n+1} - p^{n-\nu}) 
\frac{1-p^{(\nu+1)(m+1)}}{1-p^{\nu+1}}
\\
&\to 1 \quad\quad (m \to \infty). 
\end{align*}

\end{proof}

\begin{lem}\label{lem:p-adic limit-ber-poly}
Let $s$ and $t$ be integers with $t > 0$. 
\begin{itemize}
    \item[(1)] For any positive integer $u>0$, we have 
    \[
    \lim_{m \to \infty} \sum_{k=0}^m  \frac{p^{u(m-k)}}{p^{k-s}}\widetilde{B}_t(p^{k-s}) = \frac{B_{t-1}}{1 - p^u}. 
    \]
    \item[(2)] For any integer $u$ and any $\varepsilon >0$ with $u + \varepsilon > 0$, we have 
    \[
        \lim_{m \to \infty} p^{m\varepsilon}\sum_{k=0}^m  \frac{p^{u(m-k)}}{p^{k-s}}\widetilde{B}_t(p^{k-s}) = 0. 
    \]
    \end{itemize}
\end{lem}
\begin{proof}
    We have 
    \begin{align*}
        \sum_{k=0}^m p^{u(m-k)} \frac{1}{p^{k-s}}\widetilde{B}_t(p^{k-s})
        &= \frac{1}{t}\sum_{k=0}^m p^{u(m-k)} \sum_{\mu=1}^{t}\binom{t}{\mu}B_{t-\mu}p^{(\mu-1)(k-s)} 
        \\
        &= \frac{1}{t}\sum_{\mu=1}^{t}\binom{t}{\mu}B_{t-\mu}\sum_{k=0}^m p^{u(m-k) +(\mu-1)(k-s)} 
        \\
        &= \frac{1}{t} \sum_{\mu=1}^{t}\binom{t}{\mu}B_{t-\mu}
        p^{u m - (\mu-1)s} \frac{1-p^{(\mu-1-u)(m+1)}}{1-p^{\mu-1-u}}
\\
&=  \frac{1}{t}\sum_{\mu=1}^{t}\binom{t}{\mu}B_{t-\mu}
         \frac{p^{u m - (\mu-1)s} - p^{- (\mu-1)s + (m+1)(\mu-1) -u}}{1-p^{\mu-1-u}}. 
        \end{align*}
        Hence claim (2) is clear. Moreover if $u > 0$, then all the terms with $\mu \geq 2$ vanish as $m\to \infty$, and therefore, 
\[
 \lim_{m \to \infty} \sum_{k=0}^m  \frac{p^{u(m-k)}}{p^{k-s}}\widetilde{B}_t(p^{k-s}) 
 = B_{t-1} \frac{-p^{-u}}{1-p^{-u}} = \frac{B_{t-1}}{1 - p^u}. 
\]
\end{proof}

\begin{lem}\label{lem:FT2}
We have 
        \[
    \lim_{m \to \infty} \msW_2^{(m)} 
    = \frac{(-1)^{\nu}}{1-p^{n-\nu}}\frac{B_{\nu + 1}}{\nu+1}  
     = \frac{1}{1-p^{n-\nu}}\zeta(-\nu). 
    \]
\end{lem}

\begin{proof}
Lemma \ref{lem:L, E^1, E^0}(2) shows that
    \begin{align*}
          \msW_2^{(m)} &= \sum_{k=0}^m p^{(n-\nu)(m-k)} \sum_{j=0}^{p^k-1}
E_{\nu,k,j}^{(1) \ddag}
\lr{
\frac{j}{p^k}, 1
}
\\
&= \sum_{k=0}^m p^{(n-\nu)(m-k)} \left(\frac{(-1)^{\nu}}{(\nu+1)}
\frac{1}{p^k}
\widetilde{B}_{\nu+2}(p^k)
+
\sum_{j=0}^{p^k-1}
E_{\nu,k,j}^{(1), \ddag}
\lr{
0, 1
}\right).
    \end{align*}
    Hence Lemma \ref{lem:p-adic limit-ber-poly}(1) implies that 
    \[
    \lim_{m \to \infty} \msW_2^{(m)} 
    = \frac{(-1)^{\nu}}{1-p^{n-\nu}}\frac{B_{\nu + 1}}{\nu+1} + 
    \lim_{m \to \infty} \sum_{k=0}^m p^{(n-\nu)(m-k)} \sum_{j=0}^{p^k-1}
E_{\nu,k,j}^{(1), \ddag}(0,1). 
    \]
    Moreover, by Lemma \ref{lem:sum-detail_bel-E^0-E^1}, we have  
\begin{align*}
  \sum_{k=0}^m p^{(n-\nu)(m-k)} \sum_{j=0}^{p^k-1}E_{\nu,k,j}^{(1), \ddag}(0,1) 
  &= 
  \sum_{k=0}^m p^{(n-\nu)(m-k)} \sum_{\mu=0}^{\nu}
(-1)^{\nu+\mu+1}
\binom{\nu}{\mu}p^{k(\mu+1)}
\frac{B_{\mu+1}}{\mu+1}
\frac{\widetilde{B}_{\nu-\mu+1}(p^k)}{p^k}
\\
&= \sum_{\mu=0}^{\nu}(-1)^{\nu+\mu+1}\binom{\nu}{\mu}\frac{B_{\mu+1}}{\mu+1}
p^{m(\mu+1)}\sum_{k=0}^m p^{(n-\nu-\mu-1)(m-k)}  \frac{\widetilde{B}_{\nu-\mu+1}(p^k)}{p^k}. 
\end{align*}
Since $\mu+1 + n- \nu-\mu-1 = n- \nu > 0$, 
Lemma \ref{lem:p-adic limit-ber-poly}(2) shows that 
\[
\lim_{m \to \infty}p^{m(\mu+1)}\sum_{k=0}^m p^{(n-\nu-\mu-1)(m-k)}  \frac{\widetilde{B}_{\nu-\mu+1}(p^k)}{p^k} = 0. 
\]
Hence we conclude that 
    \[
    \lim_{m \to \infty} \msW_2^{(m)} 
    = \frac{(-1)^{\nu}}{1-p^{n-\nu}}\frac{B_{\nu + 1}}{\nu+1}. 
    \]
\end{proof}

\begin{lem}\label{lem:ber-poly-limit-FT3}
    Let $b$, $s$, $t$, and $u$ be integers with $s \geq 0$ and $t \geq 0$. 
\begin{itemize}
    \item[(1)] If $s > 0$ and $u > 0$, then 
\begin{align*}
 \lim_{m \to \infty}\sum_{k=0}^m p^{u(m-k)}  \sum_{l'=0}^{k-1}p^{l's} \left(\frac{\widetilde{B}_{t+1}(p^{k-l'})}{p^{k-l'}}- p^b \frac{\widetilde{B}_{t+1}(p^{k-l'-1})}{p^{k-l'-1}}\right) 
 = \frac{B_{t}}{1-p^{s}} \frac{1-p^b}{1-p^{u}}. 
\end{align*}
    \item[(2)] For any $\varepsilon > 0$ with $u + \varepsilon > 0$, we have 
    \begin{align*}
 \lim_{m \to \infty} p^{\varepsilon m}\sum_{k=0}^m p^{u(m-k)}  \sum_{l'=0}^{k-1}p^{l's} \left(\frac{\widetilde{B}_{t+1}(p^{k-l'})}{p^{k-l'}}- p^b \frac{\widetilde{B}_{t+1}(p^{k-l'-1})}{p^{k-l'-1}}\right) 
 = 0. 
\end{align*}
\end{itemize}
\end{lem}
\begin{proof}
We have 
    \begin{align*}
        &\sum_{k=0}^m p^{u(m-k)}  \sum_{l'=0}^{k-1}p^{l's} \left(\frac{\widetilde{B}_{t+1}(p^{k-l'})}{p^{k-l'}}- p^b \frac{\widetilde{B}_{t+1}(p^{k-l'-1})}{p^{k-l'-1}}\right) 
        \\
        &= \frac{1}{t+1}\sum_{k=0}^m p^{u(m-k)}  \sum_{l'=0}^{k-1}p^{l's} 
        \sum_{\mu=1}^{t+1}\binom{t+1}{\mu}B_{t+1-\mu}
        \times
        (p^{(k-l')(\mu-1)} - p^{(k-l'-1)(\mu-1) + b})
        \\
        &= \frac{1}{t+1}\sum_{\mu=1}^{t+1}\binom{t+1}{\mu}B_{t+1-\mu}
        \times
        (1-p^{b-(\mu-1)})p^{um}
        \sum_{k=0}^m p^{k(\mu-1 - u)}  \sum_{l'=0}^{k-1}p^{l'(s- \mu + 1)} 
        \\
        &= \frac{1}{t+1}\sum_{\mu=1}^{t+1}\binom{t+1}{\mu}B_{t+1-\mu}
        \times
        (1-p^{b-(\mu-1)})p^{um}
        \sum_{k=0}^m p^{k(\mu-1 - u)}  \frac{1-p^{k(s-\mu+1)}}{1-p^{s-\mu+1}}
        \\
        &= \frac{1}{t+1}\sum_{\mu=1}^{t+1}\binom{t+1}{\mu}\frac{B_{t+1-\mu}}{1-p^{s-\mu+1}}\times(1-p^{b-(\mu-1)})
        \left(
        \frac{p^{um} - p^{(m+1)(\mu-1) - u}}{1-p^{\mu-1-u}} - 
        \frac{p^{um}-p^{(m+1)s-u}}{1-p^{s-u}} 
        \right), 
    \end{align*}
     which implies this lemma in the same way as in the proof of Lemma \ref{lem:p-adic limit-ber-poly}. 
\end{proof}

\begin{lem}\label{lem:FT3}
We have 
    \[
    \lim_{m \to \infty} \msW_3^{(m)} 
= \frac{(-1)^{\nu}}{1-p^{\nu}}\frac{B_{n-\nu+1}}{n-\nu+1} 
= \frac{1}{1-p^{\nu}}\zeta(\nu-n). 
    \]
\end{lem}
\begin{proof}
    Recall that in Lemma \ref{lem:L, E^1, E^0}(3) we have shown that 
    \begin{multline*}
            \sum_{j=0}^{p^k-1}
E_{\nu,k,j}^{(0)\ddag}
\lr{
-\frac{d}{p^{k-l}}, 1
}
\\
=
\frac{(-1)^{\nu}}{n-\nu+1}
\frac{1}{p^k}
\sum_{l'=0}^{k-1}
p^{l'(\nu+1)}
\lr{
\widetilde{B}_{n-\nu+2}(p^{k-l'})
-p^{n-\nu+1}
\widetilde{B}_{n-\nu+2}(p^{k-l'-1})
}
+
\sum_{j=0}^{p^k-1}
E_{\nu,k,j}^{(0)\ddag}(0,1). 
    \end{multline*}
Since $1 \leq \nu \leq n-1$, Lemma \ref{lem:ber-poly-limit-FT3}(1) shows that   
\begin{align*}
    & \sum_{k=0}^m p^{(n-\nu)(m-k)}
\sum_{l'=0}^{k-1}
p^{l'\nu}
\lr{
\frac{\widetilde{B}_{n-\nu+2}(p^{k-l'})}{p^{k-l'}}
-p^{n-\nu}
\frac{\widetilde{B}_{n-\nu+2}(p^{k-l'-1})}{p^{k-l'-1}}
}
\\
&\to  \frac{B_{n-\nu + 1}}{1-p^{\nu}} \quad\quad (m \to \infty). 
\end{align*}
By Lemma \ref{lem:sum-detail_bel-E^0-E^1}, we have 
\begin{align*}
\sum_{j=0}^{p^k-1}
E_{\nu,k,j}^{(0)\ddag}
\lr{
0, 1
}
=
(-1)^{\nu}p^{k\nu}\frac{B_{n-\nu+1}}{n-\nu+1} 
&+
(-1)^{\nu}
\sum_{\mu=0}^{n-\nu}
\begin{pmatrix}
n-\nu\\
\mu
\end{pmatrix}
p^{k\mu}
\frac{B_{\mu+1}}{\mu+1}
\sum_{l'=0}^{k-1}
p^{l'(\nu-\mu)}
\\
&\quad \times\lr{
\widetilde{B}_{n-\nu-\mu+1}(p^{k-l'})
-
p^{n-\nu-\mu}
\widetilde{B}_{n-\nu-\mu+1}(p^{k-l'-1})}. 
\end{align*}
Since $(n-\nu)(m-k)+k\mu = (n-\nu-\mu)(m-k)+m\mu$ and $0 \leq \mu \leq n-\nu$, all the terms with $\mu \geq 1$  vanish when $m \to \infty$, and hence
\begin{align*}
&\lim_{m \to \infty}\sum_{k=0}^m p^{(n-\nu)(m-k)}\sum_{j=0}^{p^k-1}
E_{\nu,k,j}^{(0)\ddag}(0,1) 
\\
&= 
\lim_{m\to \infty}
\sum_{k=0}^m (-1)^{\nu} p^{(n-\nu)(m-k) + k\nu}\frac{B_{n-\nu+1}}{n-\nu+1}
\\
&+
(-1)^{\nu}B_1\lim_{m \to \infty}\sum_{k=0}^m p^{(n-\nu)(m-k)}
\sum_{l'=0}^{k-1}
p^{l'\nu}
\left(
\widetilde{B}_{n-\nu+1}(p^{k-l'})
-
p^{n-\nu}
\widetilde{B}_{n-\nu+1}(p^{k-l'-1})\right). 
\end{align*}
Since 
\begin{align*}
\lim_{m\to \infty}
\sum_{k=0}^m p^{(n-\nu)(m-k)+k\nu} =
\lim_{m\to \infty}
\frac{p^{(n-\nu)(m+1)}-p^{\nu(m+1)}}{p^{n-\nu}-p^{\nu}}
=0, 
\end{align*}
the first limit vanishes. 
Moreover, Lemma \ref{lem:ber-poly-limit-FT3}(2) implies that 
\begin{align*}
&\sum_{k=0}^m p^{(n-\nu)(m-k)}
\sum_{l'=0}^{k-1}
p^{l'\nu}
\left(
\widetilde{B}_{n-\nu+1}(p^{k-l'})
-
p^{n-\nu}
\widetilde{B}_{n-\nu+1}(p^{k-l'-1})\right)   
\\
&= p^{m}\sum_{k=0}^m p^{(n-\nu-1)(m-k)}
\sum_{l'=0}^{k-1}
p^{l'(\nu-1)}
\left(
\frac{\widetilde{B}_{n-\nu+1}(p^{k-l'})}{p^{k-l'}}
-
p^{n-\nu-1}
\frac{\widetilde{B}_{n-\nu+1}(p^{k-l'-1})}{p^{k-l'-1}}\right)
\\
&\to 0 \quad\quad (m \to \infty), 
\end{align*}
which implies this lemma. 
\end{proof}

By Lemmas \ref{lem:FT1}, \ref{lem:FT2}, and \ref{lem:FT3}, we obtain the following.

\begin{prop}\label{prop:main term limit}
    We have 
\begin{align*}
\lim_{m \to \infty} \sum_{k=0}^m
p^{(n-\nu)(m-k)}
&\sum_{j=0}^{p^k-1}
\brk{\Eis_n, \widetilde{C}_{\nu,k,j}
\left(
\frac{\tau}{p^{m-k}}, p^{m-k}\tau
\right)} 
\\
&=    \frac{1}{1-p^{n+1}}\frac{\zeta(-\nu)\zeta(\nu-n)}{\zeta(-1-n)} - 
   \frac{\zeta(-\nu)}{1-p^{n-\nu}}- 
   \frac{\zeta(\nu-n)}{1-p^{\nu}}.  
        \end{align*}
\end{prop}
By the argument in the beginning of \S\ref{sec:Period}, this completes the proof of Theorem \ref{thm:Tp^m and Eis}.

\subsection{Rationality of the Eisenstein class $\Eis_n$}

The  computations (in the proofs of) Proposition \ref{prop:C_nu,k,j and Eis-pairing computation} and Lemmas \ref{lem:L, E^1, E^0} and \ref{lem:sum-detail_bel-E^0-E^1} imply the following proposition as a special case. 

\begin{prop}\label{prop:pairing-rational}
      For any integer $\nu \in \{1, \ldots, n-1\}$, we have 
   \begin{align*}
   \brk{\Eis_n, \widetilde{C_{\nu}(\tau)}} 
   =
  \frac{\zeta(-\nu)\zeta(\nu-n)}{\zeta(-1-n)} - \zeta(-\nu) - \zeta(\nu-n) \in \bQ. 
   \end{align*}
\end{prop}
\begin{proof}
By Proposition \ref{prop:C_nu,k,j and Eis-pairing computation}, 
we have
\[
\brk{\Eis_n, \widetilde{C_{\nu}(\tau)}}
= \brk{\Eis_n, \widetilde{C}_{\nu,0,0}}
=
i^{\nu+1}\msL_{0,0}(\nu)
-
E_{\nu,0,0}^{(1) \ddag}
\lr{
0, 1
}
-
E_{\nu,0,0}^{(0) \ddag}
\lr{0, 1
}. 
\]
By Lemma \ref{lem:L, E^1, E^0}(1), we have 
\[
i^{\nu+1}\msL_{0,0}(\nu) = \frac{\zeta(-\nu)\zeta(\nu-n)}{\zeta(-1-n)}. 
\]
Moreover, in the proof of Lemma \ref{lem:sum-detail_bel-E^0-E^1}, we showed that 
\[
E_{\nu,0,0}^{(1) \ddag}
\lr{
0, 1
} = \zeta(-\nu)
\,\,\, \textrm{ and } \,\,\, 
E_{\nu,0,0}^{(0) \ddag}
\lr{0, 1
} = \zeta(\nu-n). 
\]
\end{proof}

The following lemma will be well-known to experts. 
For instance, Harder mentioned in \cite[\S5.1.3]{CAG} that this is proved by Gebertz in her diploma thesis. 
Here we give a proof for the completeness of the paper.

    \begin{lem}\label{lem:relative homology generator}
   The relative homology group $H_1(Y^\BS, \partial Y^\BS, \cM_n)$ is generated by the set $\{[C_\nu] \mid 0 \leq \nu \leq n\}$.      
    \end{lem}
\begin{proof}
Note that the relative homology group $H_1(Y^\BS, \partial Y^\BS, \cM_n)$ can be computed as
\begin{align*}
&H_1(Y^\BS, \partial Y^\BS, \cM_n)\\
&=
\ker\lr{
((\MS(\bbH^{\BS})/\MS(\partial \bbH^\BS))\otimes \mcM_n)_{\Gamma}
\overset{\partial}{\ra}
((S_0(\bbH^{\BS})/S_0(\partial \bbH^{\BS}))\otimes \mcM_n)_{\Gamma}
}. 
\end{align*}
Let $[\sigma] \in H_1(Y^\BS, \partial Y^\BS, \cM_n)$ be a class represented by a $1$-chain
\[
\sigma=
\sum_{j}\{a_j, b_j\}\otimes P_j 
\in 
(\MS(\bbH^{\BS})/\MS(\partial \bbH^\BS)) \otimes \mcM_n, 
\]
where $a_j, b_j \in \bbH^{\BS}$ and $P_j \in \mcM_n$. 
The condition that 
\[
\partial \sigma=0 
\,\,\, \text{ in }\,\,\, 
((S_0(\bbH^{\BS})/S_0(\partial \bbH^{\BS}))\otimes \mcM_n)_{\Gamma} 
\]
implies that 
\begin{align}\label{eq:relative-homology-eq1}
\sum_j \{b_j\}\otimes P_j -\{a_j\}\otimes P_j
=
\sum_{k} (\gamma_k-1)(\{d_k\}\otimes Q_k)
+
\sum_{l} \{c_l\} \otimes R_l
\end{align}
in $S_0(\bbH^{\BS})\otimes \mcM_n$
for some $\gamma_k \in \Gamma, d_k \in \bbH^{\BS}, c_l \in \partial \bbH^{\BS}$, and $Q_k, R_l \in \mcM_n$. 
Then we can rewrite the identity \eqref{eq:relative-homology-eq1} as
\begin{align*}
\sum_{\tau \in \bbH^{\BS}}
\{\tau\}\otimes
\lr{
\sum_{\substack{j \\ b_j=\tau}}P_j
-\sum_{\substack{j\\ a_j=\tau}}P_j
-\sum_{\substack{k\\ \gamma_k d_k=\tau}}\gamma_k Q_k
+\sum_{\substack{k\\ d_k=\tau}}Q_k
-\sum_{\substack{l\\ c_l=\tau}}R_l
}
=0 
\end{align*}
in $S_0(\bbH^{\BS})\otimes \mcM_n$. 
Since $S_0(\bbH^{\BS}) = \bigoplus_{\tau \in \bbH^\BS}\bZ\{\tau\}$, this shows that for any $\tau \in \bbH^{\BS}$, we have
\begin{align}\label{eq:relative-homology-eq2}
\sum_{\substack{j\\ b_j=\tau}}P_j
-\sum_{\substack{j\\ a_j=\tau}}P_j
-\sum_{\substack{k\\ \gamma_k d_k=\tau}}\gamma_k Q_k
+\sum_{\substack{k\\ d_k=\tau}}Q_k
-\sum_{\substack{l\\ c_l=\tau}}R_l
=0  \,\,\, \textrm{ in } \,\,\,  \mcM_n. 
\end{align}

Now, let $i \infty \in \partial \bbH^{\BS}$ denote the point defined by
\[
i \infty := \lim_{t\in \Rpos, t \to \infty} it. 
\]
Then the identity \eqref{eq:relative-homology-eq2} implies that in $\MS(\bbH^{\BS}) \otimes \mcM_n$, we have
\begin{align}\label{eq:relative-homology-eq3}
\sum_{\tau \in \bbH^{\BS}}
\{i\infty, \tau\}\otimes
\lr{
\sum_{\substack{j\\ b_j=\tau}}P_j
-\sum_{\substack{j\\ a_j=\tau}}P_j
-\sum_{\substack{k\\ \gamma_k d_k=\tau}}\gamma_k Q_k
+\sum_{\substack{k\\ d_k=\tau}}Q_k
-\sum_{\substack{l\\ c_l=\tau}}R_l
}
=0.  
\end{align}
Using the identity \eqref{eq:relative-homology-eq3}, in $((\MS(\bbH^{\BS})/\MS(\partial \bbH^\BS))\otimes \mcM_n)_{\Gamma}$, we compute 
\begin{align*}
[\sigma]
&=
\sum_j [\{i\infty, b_j\}\otimes P_j] -[\{i\infty, a_j\}\otimes P_j]\\
&=
\sum_{k} 
[
\{i\infty, \gamma_k d_k\}\otimes \gamma_k Q_k]
-
\sum_k [\{i \infty, d_k\} \otimes Q_k]
+
\sum_{l} [\{i\infty, c_l\} \otimes R_l]\\
&=
\sum_{k} 
[(\gamma_k-1)
\lr{\{i\infty, d_k\}\otimes Q_k
}]
+
\sum_k [\{i \infty, \gamma_k i \infty\} \otimes \gamma_k Q_k]
+
\sum_{l} [\{i\infty, c_l\} \otimes R_l]\\
&=
\sum_k [\{i \infty, \gamma_k i \infty\} \otimes \gamma_k Q_k]
+
\sum_{l} [\{i\infty, c_l\} \otimes R_l]. 
\end{align*}
Moreover, (as we are considering the relative homology classes) we may replace $c_l$ with any point in the same connected component of $\partial \bbH^{\BS}$ as $c_l$, 
and in particular, we may replace $c_l$ by $g_l i \infty$ for some $g_l \in \Gamma$. 
Thus we conclude that the class $[\sigma] \in H_1(Y^\BS, \partial Y^\BS, \cM_n)$ is represented by a $1$-chain of the form
\[
\sigma' 
=
\sum_{m} \{i \infty,\gamma_m' i \infty\}\otimes R_{m}'
\]
for some $\gamma_m' \in \Gamma$ and $R_m' \in \mcM_n$. 
Now the lemma follows from the facts that the group $\Gamma = \mathrm{SL}_2(\bZ)$ is generated by matrices $\begin{pmatrix}
    0&-1\\1&0
\end{pmatrix}$ and $\begin{pmatrix}
    1&1\\0&1
\end{pmatrix}$ 
and that $[\{i \infty, \gamma_1 \gamma_2 i \infty\} \otimes P] = [\{i \infty, \gamma_1 i \infty\} \otimes P] + [\{i \infty,  \gamma_2 i \infty\} \otimes \gamma_1^{-1} P]$ for any 
$\gamma_1, \gamma_2 \in \Gamma$ and $P \in \cM_n$. 
\end{proof}


Recall that $\Gamma_{\infty}$ is the subgroup of $\Gamma$ generated by $\begin{pmatrix}
        1&1\\0&1
    \end{pmatrix}$. 
Since $\mathbb{P}^1(\bbR) = \bbR \cup \{\infty\}$, 
 the boundary $\partial Y^\BS$ {can be} identified with $\Gamma_\infty \backslash \bbR$. 
Hence we obtain the following lemma. 

\begin{lem}\label{lem:boundary-homology-0}
We have an identification $H_0(\partial Y^\BS, \cM_{n}) =(\cM_{n})_{\Gamma_\infty}$, and 
hence $H_0(\partial Y^\BS, \cM_{n}) \otimes \bQ$ is a $1$-dimensional $\Q$-vector space generated by $[e_n]$, where $e_n = X_1^n$. 
\end{lem}

\begin{lem}\label{lem:kernel-boundary-rational}
The kernel of the boundary homomorphism 
\[
\partial\colon H_1(Y^\BS, \partial Y^\BS \cM_n) \otimes \bQ \longrightarrow H_0(\partial Y^\BS, \cM_n) \otimes \bQ
\]
is generated by the set $\{[C_\nu] \mid 1 \leq \nu \leq n-1\}$.
\end{lem}
\begin{proof}
Let $\sigma \in \ker(H_1(Y^\BS, \partial Y^\BS \cM_n) \otimes \bQ \longrightarrow H_0(\partial Y^\BS, \cM_n) \otimes \bQ)$. 
By Lemma\ref{lem:relative homology generator}, we can write as
$
\sigma 
=
\sum_{\nu=0}^n a_{\nu} [C_{\nu}] 
$ 
for some rational  numbers $a_{0}, \dots, a_n \in \Q$. 
Then by using Lemma \ref{lem:boundary-homology-0}, we find that
\[
0 = \partial \sigma
=
\sum_{\nu=0}^n
a_{\nu}
[\{i \infty\}\otimes (e_{\nu}-(-1)^{n-\nu}e_{n-\nu})]
=
- (a_0 -a_n) [\{i \infty\}\otimes e_n]. 
\]
Therefore, Lemma \ref{lem:boundary-homology-0} shows that  $a_0=a_n$. On the other hand we see that 
\[
[C_0]=-[C_{\nu}]
\,\,\, \textrm{ in } \,\,\,  
H_1(Y^{\BS}, \partial Y^{\BS}, \mcM_{n}).
\]
This proves the lemma. 
\end{proof}

\begin{cor}\label{cor:eisenstein class is rational}
We have $\Eis_n \in H^1(Y^\BS, \cM_n) \otimes \bQ$. 
\end{cor}
\begin{proof}
    By Lemma \ref{lem:kernel-boundary-rational}, the homology group $H_1(Y^\BS, \cM_n) \otimes \bQ$ is generated by the image of $H_1(\partial Y^\BS, \cM_n)$ and the set $\{[\widetilde{C}_\nu(\tau)] \mid 1 \leq \nu \leq n-1\}$. 
    Therefore, by Lemma \ref{lem:eisenstein-pairing-eigen-rational} and Proposition \ref{prop:pairing-rational} we have 
    \[
    \brk{\Eis_n, H_1( Y^\BS, \cM_n) \otimes \bQ} \subset \bQ, 
    \]
    which implies $\Eis_n \in H^1(Y^\BS, \cM_n) \otimes \bQ$. 
\end{proof}

\section{Denominator of an ordinary cohomology class}\label{sec:denominator-cohomology-class}

In order to study the denominator $\Delta(\Eis_n)$ of the Eisenstein class $\Eis_n$, in this section, we interpret the denominator $\Delta(\Eis_n)$ in terms of the values $\brk{\Eis_n, [\widetilde{T_p^m C_\nu(\tau)}]}$ of the pairing between the Eisenstein class $\Eis_n$ and the cycles $[\widetilde{T_p^m C_\nu(\tau)}] \in H_1(Y^\BS, \cM_{n, (p)})$.

\subsection{Definition of the ordinary part}
Let $p$ be a prime number and $M$ a finitely generated $\bZ_{p}$-module with an endomorphism $f \colon M \longrightarrow M$. 
In this subsection, we introduce the notion of the $f$-ordinary part of $M$. 

Since $M$ is a finitely generated $\bZ_{p}$-module, the $p$-adic limit 
\[
e_{f} := \lim_{m \to \infty}f^{m!} \in \mathrm{End}_{\bZ_p}(M)
\]
always exists, and $e_f^2 = e_f$.  
We define the $f$-ordinary part $M_{\rm ord}$ of $M$ by 
\[
M_{\rm ord} := e_{f}M, 
\]
and we say that $m \in M$ is ($f$-)ordinary if $m \in M_\ord$, that is, $e_f m = m$. 
We also put $M_{\textrm{nilp}} := (1-e_{f})M$. 
We then have $M = M_{\ord} \oplus M_{\rm nilp}$. 

The following lemma follows from the fact that $e_f^2 = e_f$, 

\begin{lem}
The functor $M \mapsto M_{\ord}$ is exact. 
\end{lem}

\subsection{Denominator of a cohomology class}

Recall 
\[
H^1_{\rm int}(Y^{\BS}, \cM_n^\flat) = \image\left(H^1(Y^{\BS}, \cM_n^\flat) \longrightarrow H^1(Y^{\BS}, \cM_n^\flat) \otimes \bQ\right). 
\]

\begin{dfn}\label{dfn:denominator-of-general-cohomology-class}
For any cohomology class $c \in H^1(Y^{\BS}, \cM_n^\flat) \otimes \bQ$, we define the denominator $\Delta(c) \in \bZ_{>0}$ of $c$ by 
\[
\Delta(c) := \min\{\Delta \in \bZ_{>0} \mid \Delta c \in H^1_{\rm int}(Y^{\BS}, \cM_n^\flat)\}, 
\]
    and for each prime number $p$, we set 
    \[
    \delta_p(c) := \ord_p(\Delta(c)) \,\,\, \textrm{ and } \,\,\, \Delta_p(c) := p^{\delta_p(c)}. 
    \]
\end{dfn}

\begin{lem}
    Let $c \in  H^1(Y^{\BS}, \cM_n) \otimes \bQ$ be a cohomology class. 
    We have 
    \[
    \Delta(c) = \min\{\Delta \in \bZ_{>0} \mid 
    \Delta \brk{c, H_1(Y^{\BS}, \cM_n)} 
    \subset \bZ\}. 
    \]
    Moreover, for any prime number $p$, we have 
    \[
    \delta_p(c) = \min\{\delta \in \bZ_{\geq 0} \mid p^{\delta} 
   \brk{c,H_1(Y^{\BS}, \cM_n \otimes \bZ_p)}
    \subset \bZ_p \}. 
    \]
\end{lem}

\begin{proof}
This lemma follows immediately from the formal duality (see \S \ref{sec:formal-duality}): 
\[
H^{\bullet}(\Gamma \bs X, \mcM_n^{\flat})/\text{(torsion)}
\overset{\sim}{\longrightarrow}
\Hom_{\Z}(H_{\bullet}(\Gamma\bs X, \mcM_n), \Z). 
\]
\end{proof}

\subsection{Denominator of an ordinary cohomology class}
 In this subsection, we fix a prime number $p$. 

\begin{dfn}\
\begin{itemize}
    \item[(1)] We put $\cM_{n, p} := \cM_{n} \otimes \bZ_p$  and $\cM_{n, p}^\flat := \cM_{n}^\flat \otimes \bZ_p$. 
\item[(2)] For any finitely generated $\bZ_p$-algebra, we put 
\begin{align*}
    H_\bullet^\ord(Y^\BS, \cM_{n,p} \otimes R) &:= e_{T_p}H_\bullet(Y^\BS, \cM_{n,p} \otimes R), 
    \\
    H_\bullet^\ord(Y^\BS, \partial Y^\BS, \cM_{n,p} \otimes R) &:= e_{T_p}H_\bullet(Y^\BS, \partial Y^\BS,  \cM_{n,p} \otimes R), 
    \\
    H_\bullet^\ord(\partial Y^\BS, \cM_{n,p} \otimes R) &:= e_{T_p}H_\bullet(\partial Y^\BS, \cM_{n,p} \otimes R). 
\end{align*}
\item[(3)] For any finitely generated $\bZ_p$-algebra $R$, we put 
\begin{align*}
    H^\bullet_\ord(Y^\BS, \cM_{n,p}^\flat \otimes R) &:= e_{T_p'}H^\bullet(Y^\BS, \cM_{n,p}^\flat \otimes R), 
    \\
    H^\bullet_\ord(Y^\BS, \partial Y^\BS, \cM_{n,p}^\flat \otimes R) &:= e_{T_p'}H^\bullet(Y^\BS, \partial Y^\BS,  \cM_{n,p}^\flat \otimes R), 
    \\
    H^\bullet_\ord(\partial Y^\BS, \cM_{n,p}^\flat \otimes R) &:= e_{T_p'}H^\bullet(\partial Y^\BS, \cM_{n,p}^\flat \otimes R). 
\end{align*}
\end{itemize}
\end{dfn}

\begin{lem}\label{lem:e_p-pairing-ordinary}
   Let $c \in  H^1(Y^{\BS}, \cM_{n,p}) \otimes \bQ_p$ be a cohomology class.  
   For any homology class $C \in H_1(Y^{\BS}, \cM_{n,p}) \otimes \bQ_p$, we have 
\[
\brk{c, e_{T_p}C}  = \lim_{m \to \infty} \brk{c, T_p^{m!}C}  = \lim_{m \to \infty} \brk{c|T_p'^{m!}, C}
=\brk{e_{T_p'}c, C}. 
\]
In particular, if $c$ is ordinary, that is, $e_{T_p'}c = c$, then 
\[
\brk{c, C} = \brk{c, e_{T_p}C}, 
\]
and hence
\[
    \delta_p(c) = \min\{\delta \in \bZ_{\geq 0} \mid p^{\delta} \brk{c, H_1^{\ord}(Y^{\BS}, \cM_{n, p})} \subset \bZ_p \}. 
\]
\end{lem}
\begin{proof}
    This lemma follows from the facts that the pairing $\brk{~,~}$ is continuous and $\brk{c|T_p', C} = \brk{c, T_pC}$. 
\end{proof}

Recall the identification $H_0(\partial Y^\BS, \cM_{n,p}) = (\cM_{n,p})_{\Gamma_{\infty}}$ by Lemma \ref{lem:boundary-homology-0}. 

\begin{lem}\label{lem:H_0-ordinary-structure}
    The $\bZ_p$-module $H_0^{\ord}(\partial Y^\BS, \cM_{n,p}) = e_{T_p}(\cM_{n,p})_{\Gamma_{\infty}}$ is free of rank $1$ and is generated by $e_{T_p}[e_n]$. 
\end{lem}
\begin{proof}
    Let $k \in \{0, \ldots, n\}$ be an integer. Since $e_k = X_1^kX_2^{n-k}$, we have
    \[
    T_p([e_k]) = p^{n-k}[e_k] + \sum_{j=0}^{p-1} \sum_{k'=0}^{k} \binom{k}{k'} p^{k'}(-j)^{k-k'}[e_{k'}]. 
    \]
    In particular, we have $T_p([e_0]) = (p^{n} + p)[e_0]$, and hence $e_{T_p}[e_0] = 0$. 
    Therefore, inductively, we obtain that $e_{T_p}[e_k] = 0$ for any integer $k \in \{0, \ldots, n-1\}$ and that $e_{T_p}[e_n]=[e_n]$, 
    which implies that the $\bZ_p$-module $e_{T_p}(\cM_{n,p})_{\Gamma_{\infty}}$ is generated by  $e_{T_p}[e_n]$. 
    Hence by Lemma \ref{lem:boundary-homology-0}, $e_{T_p}(\cM_{n,p})_{\Gamma_{\infty}}$ is a free $\bZ_p$-module of rank $1$.
     \end{proof}

Recall $[C_\nu] = [\{0, i \infty\} \otimes e_\nu] \in H_1(Y^\BS, \partial Y^\BS, \cM_n)$. 

\begin{prop}\label{prop:homology_generator_ordinary}
    For any integer $m \geq 0$, 
    the $\bZ_p$-module $H_{1}^{\ord}(Y^{\BS}, \cM_{n, p})$ is generated by 
    (the image of) $H_1^\ord(\partial Y^\BS, \cM_{n,p})$ and a set of lifts of $e_{T_p}T_p^m[C_\nu]$ $(1 \leq \nu \leq n-1)$. 
\end{prop}
\begin{proof}
    By definition, we have an exact sequence of Hecke modules
    \begin{align*}
            0 \longrightarrow H_1(\partial Y^\BS, \, &\cM_{n,p}) \longrightarrow H_1( Y, \cM_{n,p}) 
            \longrightarrow 
            H_1( Y^\BS, \partial Y^\BS, \cM_{n,p}) \stackrel{\partial}{\longrightarrow} H_0(\partial Y^\BS, \cM_{n,p}) \longrightarrow 0. 
     \end{align*}
     By Lemma \ref{lem:relative homology generator}, the ordinary part $H_1^{\ord}( Y^\BS, \partial Y^\BS, \cM_{n,p})$ of the relative homology group is generated by the set $\{e_{T_p}[C_\nu] \mid 0 \leq \nu \leq n\}$. 
     Then by the same argument as in Lemma \ref{lem:kernel-boundary-rational} using Lemma \ref{lem:H_0-ordinary-structure} instead of Lemma \ref{lem:boundary-homology-0}, we find that 
     the kernel of the boundary map 
     $\partial \colon H_1^{\ord}( Y^\BS, \partial Y^\BS, \cM_{n,p}) \longrightarrow H_0^{\ord}(\partial Y^\BS, \cM_{n,p})$ is generated by the set $\{e_{T_p}[C_\nu] \mid 1 \leq \nu \leq n-1\}$. 
     Since the homomorphism 
    \[
    T_p \colon H_1^\ord(Y^\BS, \partial Y^\BS, \cM_{n,p}) \longrightarrow H_1^\ord(Y^\BS, \partial Y^\BS, \cM_{n,p})
    \]
    is an isomorphism and the boundary map $\partial$ is Hecke equivariant,  
    it follows that the kernel of the boundary map is generated by the set $\{e_{T_p}T_p^m[C_\nu] \mid 1 \leq \nu \leq n-1\}$. 
    This fact together with the above exact sequence implies this proposition. 
\end{proof}

\begin{cor}\label{cor:ordinary class denominator}
     Let $m \geq n$ be an integer. 
     For each integer $\nu \in \{1, \ldots, n-1\}$, recall the cycle  $[\widetilde{T_p^m C_\nu(\tau)}] \in H_1(Y^\BS, \cM_{n, (p)})$ defined in Definition \ref{def:widetilde-T_p^m-C-nu-tau} (see also Lemma \ref{lem:T_p-cycle-integral}(2)). 
     Then for any ordinary cohomology class $c \in H^1_{\rm ord}(Y^{\BS}, \cM_{n,p}) \otimes \bQ_p$ satisfying  
     $\brk{c, H_1(\partial Y^{\rm BS}, \cM_n)} \subset \bZ_p$, we have  
\[
\delta_p(c) = \min\{ \delta \in \bZ_{\geq0} \mid p^\delta \brk{c, [\widetilde{T_p^m C_\nu(\tau)}]} \in \bZ_p \textrm{ for any integer } 1 \leq \nu \leq n-1 \}. 
\]
\end{cor}

\begin{proof}
     By Lemma \ref{lem:T_p-cycle-integral} (3), we see that $e_{T_p} [\widetilde{T_p^m C_\nu(\tau)}] \in H_1^\ord(Y^\BS, \cM_{n,p})$ maps to $e_{T_p} T_p^m [C_\nu(\tau)] \in H_1^\ord(Y^\BS, \partial Y^\BS, \cM_{n,p})$ under the homomorphism
     \[
     H_1^\ord(Y^\BS, \cM_{n,p})\longrightarrow  H_1^\ord(Y^\BS, \partial Y^\BS, \cM_{n,p}). 
     \]
     Therefore, Proposition \ref{prop:homology_generator_ordinary} shows that 
     the $\bZ_p$-module $H_1^\ord(Y^\BS, \cM_{n,p})$ is generated by 
     the image of $H_1^\ord(\partial Y^{\rm BS}, \cM_n)$ and the set 
     $\{e_{T_p} [\widetilde{T_p^m C_\nu(\tau)}] \mid 1 \leq \nu \leq n-1\}$. 
    On the other hand, by Lemma \ref{lem:e_p-pairing-ordinary}, we have 
    \[
    \brk{c, e_{T_p}[\widetilde{T_p^m C_\nu(\tau)}]} = \brk{c, [\widetilde{T_p^m C_\nu(\tau)}]}. 
    \]
    Now, since $\brk{c, H_1(\partial Y^{\BS}, \cM_n)} \subset \bZ_p$ by assumption, Lemma \ref{lem:e_p-pairing-ordinary} shows that 
\[
\delta_p(c) = \min\{ \delta \in \bZ_{\geq 0} \mid p^\delta \brk{c, [\widetilde{T_p^m C_\nu(\tau)}]} \in \bZ_p \textrm{ for any integer } 1 \leq \nu \leq n-1 \}. 
\]
\end{proof}

\begin{cor}\label{cor:ordinary Eisenstein denominator}
    For any integer $m \geq n$, we have  
\[
\delta_p(\Eis_n) = \min\{ \delta \in \bZ_{\geq0} \mid p^\delta {\brk{\Eis_n, [\widetilde{T_p^m C_\nu(\tau)}]}} \in \bZ_p \textrm{ for any integer } 1 \leq \nu \leq n-1 \}. 
\]
\end{cor}
\begin{proof}
    By Lemma \ref{lem:eisenstein-pairing-eigen-rational}, we have 
    $\brk{\Eis_n, H_1(\partial Y^{\rm BS}, \cM_n)} \subset \bZ_p$ and 
    $e_{T_p'}\Eis_n = \Eis_n$, 
    this corollary follows from Corollary \ref{cor:ordinary class denominator}. 
\end{proof}

\section{Relation between the denominators of the Eisenstein classes}\label{sec:Eisenstein class}

Recall that $\Delta_p(\Eis_n)$ denotes the $p$-part of the denominator $\Delta(\Eis_n)$ of the Eisenstein class (see Definition \ref{dfn:denominator-of-general-cohomology-class}). 
In this section, we fix a prime number $p \geq 5$ and discuss another expression for the denominator $\Delta(\Eis_n)$ of the Eisenstein class $\Eis_n$. Moreover, we study a relation  of the denominators $\Delta_p(\Eis_n)$ and $\Delta_p(\Eis_{n'})$ of the Eisenstein classes when $n$ and $n'$ are $p$-adically close.

\subsection{Structure of the ordinary part of cohomology groups}

In this subsection, we study the structure of the ordinary part of cohomology groups. 
Results similar to those obtained in this subsection can be found in the paper \cite{Hida86} of Hida. 
In the papers \cite{Hida86, Hida88}, Hida studied the ordinary part of cohomology groups for $\Gamma' \bs \bbH$ in the case that $\Gamma'/\{\pm 1\}$ is torsion-free.  
However, in the present paper we consider the group $\Gamma = \mathrm{SL}_2(\bZ)$ which has torsion elements other than $\pm \begin{pmatrix}
    1&\\&1
\end{pmatrix}$. 
Hence, for the completeness of the present paper, we give the details of the proof of all the necessary facts. 

Note that since we assume that $p \geq 5$, any short exact sequence of $\Gamma$-modules induces a long exact sequence in cohomology.

\begin{lem}\label{lem:ordinary-flat-integral=non-flat-ntegral}
The inclusion map $\cM_{n,p}^\flat \longhookrightarrow \cM_{n,p}$ induces an isomorphism $H^{\bullet}_\ord(Y^\BS, \cM_{n,p}^\flat) \stackrel{\sim}{\longrightarrow} H^{\bullet}_\ord(Y^\BS, \cM_{n,p})$. 
\end{lem}
\begin{proof}
    It suffices to prove that $e_{T_p'}H^{i}(Y^\BS, \cM_{n,p}/\cM_{n,p}^\flat) = 0$ for any integer $i \geq 0$. 
    Since $X_1^n, X_2^n \in \cM_{n, p}^\flat$, any element in $\cM_{n,p}/\cM_{n,p}^\flat$ can be represented by a polynomial of the form $X_1X_2f(X_1, X_2)$, where $f \in \cM_{n-2, p}$.  
    Hence the fact that 
    \begin{align*}
        \widetilde{\begin{pmatrix}
        p&\\&1
    \end{pmatrix}} \cdot X_1X_2f(X_1, X_2) &= pX_1X_2f(pX_1, X_2), 
    \\
    \widetilde{\begin{pmatrix}
        1&j\\&p
    \end{pmatrix}} \cdot X_1X_2f(X_1, X_2) &= p(X_1+jX_2)X_2f(X_1+jX_2, pX_2)
    \end{align*} 
    shows that $e_{T_p'} c = 0$ for any element $c \in \Hom_\bZ(S_\bullet(\bH^\BS), \cM_{n,p}/\cM_{n,p}^\flat)$. 
    In particular, we have $e_{T_p'}H^{i}(Y^\BS, \cM_{n,p}/\cM_{n,p}^\flat) = 0$. 
\end{proof}

Thanks to  Lemma \ref{lem:ordinary-flat-integral=non-flat-ntegral}, in the following, we focus on the ordinary cohomology groups with coefficient $\mcM_{n,p}$.

\begin{lem}\label{lem:action of Hecke operator at p on mod p polynomials is nilpotent modulo X_2}   
For any polynomial $f(X_1, X_2) \in \cM_n/p\cM_n$, we have $f|({T_p'})^2  \in \bF_p  X_2^n$. 
\end{lem}
\begin{proof}
By definition, we have 
\begin{align*}
(f|{T_p'})(X_1, X_2) = f(0, X_2) + \sum_{j=0}^{p-1} f(X_1+jX_2, 0). 
\end{align*}
Hence we see that $(f|{T_p'})(X_1, 0) = \sum_{j=0}^{p-1} f(X_1, 0) = 0$, and we obtain 
\[
(f|({T_p'})^2)(X_1, X_2) = (f|{T_p'})(0, X_2) \in \bF_p  X_2^n. 
\]
\end{proof}

The boundary $\partial Y^\BS$ is of real dimension $1$, and hence $H^{2}(\partial Y^\BS, \cM)$ vanishes for any $\Gamma$-module $\cM$. 
Therefore, for any integer $r \geq 0$, the short exact sequence  $0 \longrightarrow \cM_{n,p} \stackrel{\times p^r}{\longrightarrow} \cM_{n,p} \longrightarrow \cM_{n,p}/p^r \cM_{n,p} \longrightarrow 0$ induces an isomorphism 
\begin{align}\label{eq:reduction-isomo-H^1-ordinary-boundary}
    H^{1}_{\ord}(\partial Y^\BS, \cM_{n,p}) \otimes \bZ_p/(p^r) \stackrel{\sim}{\longrightarrow} H^{1}_{\ord}(\partial Y^\BS, \cM_{n,p}/p^r \cM_{n,p}). 
\end{align}

\begin{lem}\label{lem:H^1-ordinary-boundary-torsion-free}
    The ordinary part $H^{1}_{\ord}(\partial Y^\BS, \cM_{n, p})$ is torsion-free. 
\end{lem}
\begin{proof}
    By using the exact sequence of $M_2^+(\bZ)$-modules 
\[
0 \longrightarrow \cM_{n, p} 
\stackrel{\times p}{\longrightarrow}
\cM_{n, p} \longrightarrow \cM_{n, p}/p\cM_{n, p} \longrightarrow 0, 
\]
we obtain an isomorphism of Hecke modules 
\[
\coker\left(\cM_{n, p}^{\Gamma_\infty}/p \cM_{n, p}^{\Gamma_\infty} 
\longrightarrow (\cM_{n, p}/p\cM_{n, p})^{\Gamma_\infty}\right) \stackrel{\sim}{\longrightarrow} H^{1}(\partial Y^\BS, \cM_{n, p})[p],  
\]
where for an abelian group $M$ we write $M[p] := \ker(M \stackrel{\times p}{\longrightarrow} M)$ for the subgroup of $p$-torsion elements of $M$. 
A direct computation shows that $\cM_{n, p}^{\Gamma_\infty} = \bZ_p X_2^n$. 
Hence Lemma \ref{lem:action of Hecke operator at p on mod p polynomials is nilpotent modulo X_2} implies that $H^{1}_\ord(\partial Y^\BS, \cM_{n, p})[p] = \coker\left(\bF_pX_2^n
\longrightarrow e_{T_p'}(\cM_{n, p}/p\cM_{n, p})^{\Gamma_\infty}\right)  = 0$. 
\end{proof}

\begin{cor}\label{cor:H^1-ordinary-boundary}
\
\begin{itemize}
    \item[(1)] The ordinary part $H^{1}_{\ord}(\partial Y^\BS, \cM_{n, p})$ is a free $\bZ_p$-module of rank $1$. 
    \item[(2)] We have a canonical isomorphism $H^{1}_{\mathrm{int}}(\partial Y^\BS, \cM_{n}) \otimes \bZ_p \stackrel{\sim}{\longrightarrow} H^{1}_{\ord}(\partial Y^\BS, \cM_{n, p})$. 
    \item[(3)]  We have $c|T_{\ell}' = (1+\ell^{n+1})c$  for any element $c \in H^{1}_{\ord}(\partial Y^\BS, \cM_{n, p})$ and prime number $\ell$. 
\end{itemize}
\end{cor}

\begin{proof}
By Lemma \ref{lem:H^1-ordinary-boundary-torsion-free}, we have a surjective homomorphism   
\[
H^{1}_{\mathrm{int}}(\partial Y^\BS, \cM_{n}) \otimes \bZ_p \longrightarrow H^{1}_\ord(\partial Y^\BS, \cM_{n, p}). 
\]
Lemma \ref{lem:boundary-H^1-is-eisenstein} shows that $H^{1}_{\mathrm{int}}(\partial Y^\BS, \cM_{n}) \cong \bZ$ and $c|T_{\ell}' = (1+\ell^{n+1})c$  for any element $c \in H^{1}_{\mathrm{int}}(\partial Y^\BS, \cM_{n})$ and prime number $\ell$. 
These facts imply this corollary. 
\end{proof}

\begin{prop}
\label{prop:H^1-ord-torsion-free}
The ordinary part $H^{1}_{\rm ord}(Y^\BS, \cM_{n, p})$ is torsion-free. 
\end{prop}

\begin{proof}
Since $H^0(Y^{\BS}, \mcM_n)=\mcM_n^{\Gamma}=0$, the exact sequence 
\[
0 \longrightarrow \cM_{n, p}  \stackrel{\times p}{\longrightarrow} \cM_{n, p}  \longrightarrow  \cM_{n, p}/p \cM_{n, p}\longrightarrow 0
\]
implies that 
\[
H^0_{\ord}(Y^\BS, \cM_{n, p}/p \cM_{n, p}) = H^{1}_{\ord}(Y^\BS, \cM_{n, p})[p]. 
\]
Since $\bF_p X_2^n \cap H^0(Y^\BS, \cM_{n, p}/p \cM_{n, p}) = (\bF_p X_2^n)\cap (\mcM_{n,p}/p\mcM_{n,p})^\Gamma = 0$, 
Lemma \ref{lem:action of Hecke operator at p on mod p polynomials is nilpotent modulo X_2} shows that  the module 
$H^0_{\ord}(Y^\BS, \cM_{n, p}/p \cM_{n, p})$ vanishes. 
\end{proof}

\begin{lem}\label{lem:H^2_ord(Y^BS, partial Y^BS, cM) = H^2_ord(Y^BS, cM) = 0}
   We have $H^2_\ord(Y^\BS, \partial Y^\BS, \cM_{n, p}) = H^2_\ord(Y^\BS,  \cM_{n, p})  = 0$. 
\end{lem}

\begin{proof}
Since the boundary $\partial Y^\BS$ is homeomorphic to the circle, we have $H^2(\partial Y^\BS,  \cM_{n, p}) = 0$. 
Hence the canonical homomorphism $H^2_\ord(Y^\BS, \partial Y^\BS, \cM_{n, p}) \longrightarrow H^2_\ord(Y^\BS,  \cM_{n, p})$ is surjective. 
Therefore, we only need to show that $H^2_\ord(Y^\BS, \partial Y^\BS, \cM_{n, p})  = 0$. 

Moreover, since $Y$ is a two-dimensional real manifold, we have $H^3(Y^\BS, \partial Y^\BS, \cM_{n, p}) = 0$, and hence the short exact sequence $0 \longrightarrow \cM_{n, p} \stackrel{\times p}{\longrightarrow} \cM_{n, p} \longrightarrow \cM_{n, p} /p\cM_{n, p} \longrightarrow 0$ induces an isomorphism 
\[
H^2(Y^\BS, \partial Y^\BS, \cM_{n, p}) \otimes \bZ/(p) \stackrel{\sim}{\longrightarrow} H^2(Y^\BS, \partial Y^\BS, \cM_{n, p} /p\cM_{n, p}). 
\]
Therefore, it suffices to prove that $H^2_\ord(Y^\BS, \partial Y^\BS, \cM_{n, p} /p\cM_{n, p}) = 0$. 

For notational simplicity, set $\overline{\cM}_{n, p} := \cM_{n, p} /p\cM_{n, p}$. 
Let 
\[ 
\msF  \in S_2(\bbH^{\BS})
\] 
be a representative of a fundamental class of $H_2(Y^{\BS}, \partial Y^{\BS}, \Z) \cong \bZ$. 
Then it is known that 
the homomorphism $\Hom_\bZ(S_2(\bbH^{\BS}), \overline{\cM}_{n, p}) \longrightarrow \overline{\cM}_{n, p}; \phi \mapsto \phi(\msF)$ induces an isomorphism 
\begin{align}\label{eq:evaluation-at-fundamental-class}
\mathrm{ev}_{\msF} \colon H^2(Y^\BS, \partial Y^\BS, \overline{\cM}_{n, p}) \stackrel{\sim}{\longrightarrow} (\overline{\cM}_{n, p})_\Gamma. 
\end{align}
See \cite[Proposition 8.2]{Shimura-book-arith-autom-funct} or \cite[Proposition 1, \S6.1]{Hida93} for example. 

We will show that for any $[\phi] \in H^2(Y^\BS, \partial Y^\BS, \overline{\cM}_{n, p})$, we have $[\phi]|_{T_p'}=0$. 
By \eqref{eq:evaluation-at-fundamental-class}, it suffices to show that $\mathrm{ev}_{\msF}([\phi]|_{T_p'})\equiv 0$. 
Here we use $\equiv$ to emphasize that it is an identity in $(\overline{\cM}_{n, p})_\Gamma$. 
We then compute 
\begin{align*}
\mathrm{ev}_{\msF}([\phi]|_{T_p'})
&\equiv \phi|_{T_p'}(\msF)(X_1, X_2)
\\
&=
\widetilde{\begin{pmatrix}
    p&0\\
    0&1
\end{pmatrix}}
\phi\left(\begin{pmatrix}
    p&0\\
    0&1
\end{pmatrix}\msF\right)(X_1, X_2)
+
\sum_{j=0}^{p-1}
\widetilde{\begin{pmatrix}
    1&j\\
    0&p
\end{pmatrix}}
\phi\left(\begin{pmatrix}
    1&j\\
    0&p
\end{pmatrix}\msF \right)(X_1, X_2)
\\
&\equiv
\phi\left(\begin{pmatrix}
    p&0\\
    0&1
\end{pmatrix}\msF \right)(0, X_2)
+
\sum_{j=0}^{p-1}
\phi\left(\begin{pmatrix}
    1&j\\
    0&p
\end{pmatrix}\msF \right)(X_1+jX_2, 0). 
\end{align*}
Put
\[
a_{p}:=\phi\left(\begin{pmatrix}
    p&0\\
    0&1
\end{pmatrix}\msF\right)(0, 1) \in \bF_p, 
\quad
a_{j}:=\phi\left(\begin{pmatrix}
    1&j\\
    0&p
\end{pmatrix}\msF\right)(1, 0)
 \in \bbF_p. 
\]
Then we find that
\begin{align*}
\mathrm{ev}_{\msF}([\phi]|_{T_p'}) 
&\equiv
a_{p}X_2^n
+
\sum_{j=0}^{p-1}
a_{j}(X_1+jX_2)^n
\\
&=
a_{p}X_2^n
+
\sum_{j=0}^{p-1}
a_{j}
\begin{pmatrix}
    1&1\\
    0&1
\end{pmatrix}^{-j}\begin{pmatrix}
    0&1\\
    -1&0
\end{pmatrix}
X_2^n
\\
&\equiv
a_{p}X_2^{n}
+
\sum_{j=0}^{p-1}
a_{j}X_2^{n}
\\
&=
\left(a_{p}
+
\sum_{j=0}^{p-1}
a_{j}\right)\lr{\begin{pmatrix}
    1&-1\\
    0&1
\end{pmatrix}-1}X_1X_2^{n-1}
\\
&\equiv
0. 
\end{align*}
\end{proof}

For any $\Gamma$-module $\cM$, we define the inner cohomology $H^1_{!}(Y^\BS, \cM)$ by 
\[
H^1_{!}(Y^\BS, \cM) := \image\left(H^1(Y^\BS, \partial Y^\BS \cM) \longrightarrow H^1(Y^\BS, \cM)\right) 
\]
and,  when $\cM$ is a finitely generated $\bZ_p$-module, we put 
\[
H^1_{!, \ord}(Y^\BS, \cM) := e_{T_p'}H^1_{!}(Y^\BS, \cM). 
\]
Then the following corollary follows from Lemma \ref{lem:H^2_ord(Y^BS, partial Y^BS, cM) = H^2_ord(Y^BS, cM) = 0} and the isomorphism \eqref{eq:reduction-isomo-H^1-ordinary-boundary}. 

\begin{cor}\label{cor:exact-seq-ordinary-H^1-inner-boundary}
Let $r$ be a non-negative integer and $\cM \in \{\cM_{n, p}, \cM_{n, p}/p^r\cM_{n, p}\}$. 
Then we have a natural exact sequence of Hecke modules: 
\begin{align*}
    0 \longrightarrow H^{1}_{!, \ord}(Y^\BS, \cM) \longrightarrow H^{1}_{\ord}(Y^\BS, \cM)  \longrightarrow  H^{1}_{\ord}(\partial Y^\BS,  \cM) \longrightarrow 0. 
\end{align*}
\end{cor}
\begin{proof}

For notational simplicity, set $\cM_{n,p}/p^r:=\cM_{n,p}/p^r\cM_{n,p}$. 
Consider the natural commutative diagram
\begin{align*}
\xymatrix{
0\ar[r]&
H^{1}_{!, \ord}(Y^\BS, \cM_{n,p})\ar[r]\ar[d]&
H^{1}_{\ord}(Y^\BS, \cM_{n,p})\ar[r]\ar[d]&
H^{1}_{\ord}(\partial Y^\BS, \cM_{n,p})\ar[r]\ar[d]&
0\\
0\ar[r]&
H^{1}_{!, \ord}(Y^\BS, \cM_{n,p}/p^r)\ar[r]&
H^{1}_{\ord}(Y^\BS, \cM_{n,p}/p^r)\ar[r]&
H^{1}_{\ord}(\partial Y^\BS, \cM_{n,p}/p^r)\ar[r]&
0.  
}
\end{align*}
The upper row is exact by Lemma \ref{lem:H^2_ord(Y^BS, partial Y^BS, cM) = H^2_ord(Y^BS, cM) = 0}. 
Moreover, by \eqref{eq:reduction-isomo-H^1-ordinary-boundary}, the right vertical map is surjective, and the bottom row is also exact. 
\end{proof}

\begin{cor}\label{cor:conpatible-reduction-ordinary-inner}
For any integer $r \geq 0$, the canonical homomorphism  $\cM_{n,p} \longrightarrow \cM_{n,p}/p^r \cM_{n,p}$ induces isomorphisms
\begin{align*}
H^{1}_{\ord}(Y^\BS, \cM_{n,p}) \otimes \bZ_{p}/(p^{r}) &\stackrel{\sim}{\longrightarrow} H^{1}_{\ord}(Y^\BS, \cM_{n, p}/p^r \cM_{n,p}), 
\\
H^{1}_{!, \ord}(Y^\BS, \cM_{n, p}) \otimes \bZ_{p}/(p^{r})  &\stackrel{\sim}{\longrightarrow}  H^{1}_{!, \ord}(Y^\BS, \cM_{n, p}/p^r \cM_{n,p}).  
\end{align*}
\end{cor}
\begin{proof}
The exact sequence 
\[
0 \longrightarrow \cM_{n, p}  \stackrel{\times p^r}{\longrightarrow} \cM_{n, p}  \longrightarrow  \cM_{n, p}/p^r \cM_{n, p}\longrightarrow 0 
\]
shows that we have an exact sequence 
\begin{align*}
    0 \longrightarrow H^{1}_{\ord}(Y^\BS, \cM_{n,p}) \otimes \bZ_{p}/(p^{r})  {\longrightarrow} H^{1}_{\ord}(Y^\BS, \cM_{n, p}/p^r \cM_{n,p})  \longrightarrow  H^{2}_{\ord}(Y^\BS,  \cM_{n,p})[p^r] \longrightarrow 0. 
\end{align*}
Hence by Lemma \ref{lem:H^2_ord(Y^BS, partial Y^BS, cM) = H^2_ord(Y^BS, cM) = 0}, we obtain the first isomorphism.

By Corollary \ref{cor:H^1-ordinary-boundary}(1), we see that $\Tor_{1}^{\Z_p}(H^{1}_{\ord}(\partial Y^\BS, \cM_{n, p}), \Z_p/(p^r))=0$, and hence Corollary \ref{cor:exact-seq-ordinary-H^1-inner-boundary} for $\cM = \cM_{n,p}$ shows that 
\small
\begin{align*}
    H^{1}_{!, \ord}(Y^\BS, \cM_{n, p}) \otimes \bZ_p/(p^r) = \ker\left( H^{1}_{\ord}(Y^\BS, \cM_{n, p}) \otimes \bZ_p/(p^r) \longrightarrow H^{1}_{\ord}(\partial Y^\BS, \cM_{n, p}) \otimes \bZ_p/(p^r)\right). 
\end{align*}
\normalsize
Hence the second isomorphism follows from the first isomorphism, the isomorphism \eqref{eq:reduction-isomo-H^1-ordinary-boundary}, and Corollary \ref{cor:exact-seq-ordinary-H^1-inner-boundary} for $\cM = \cM_{n,p}/p^r \cM_{n,p}$.   
\end{proof}

\begin{thm}\label{thm:harder-hida}
For any positive integers $r$ and $n'$ with $n \equiv n' \pmod{(p-1)p^{r-1}}$, we have the following canonical isomorphism of exact sequences which is $T_\ell'$-equivalent for any prime number $\ell \neq p$: 
\begin{align*}
\xymatrix@C=20pt{
0 \ar[r] &  H^{1}_{!, \ord}(Y^\BS, \cM_{n, p}/ p^{r}) \ar[d]^-\cong \ar[r] & H^{1}_{\ord}(Y^\BS, \cM_{n, p}/ p^{r})  \ar[d]^-\cong \ar[r] &    H^{1}_{\ord}(\partial Y^\BS, \cM_{n, p}/ p^{r}) \ar[d]^-\cong \ar[r] & 0
\\
0 \ar[r] &   H^{1}_{!, \ord}(Y^\BS, \cM_{n', p}/ p^{r})   \ar[r] &  H^{1}_{\ord}(Y^\BS, \cM_{n', p}/ p^{r})  \ar[r] &  H^{1}_{\ord}(\partial Y^\BS, \cM_{n', p}/ p^{r}) \ar[r] & 0, 
}
\end{align*}
where  $\cM_{n, p}/p^r := \cM_{n, p}/ p^{r}\cM_{n, p}$. 
\end{thm}

\begin{proof}
    Theorem \ref{thm:harder-hida}  follows from  the results proved by Hida in \cite{Hida86} (see also \cite{Harder11}). In the following, we briefly explain how we derive Theorem \ref{thm:harder-hida} from  Hida's results in \cite{Hida86}.

First, note that since $p \geq 5$, we have canonical isomorphisms between a sheaf cohomology on $Y^\BS$ and a group cohomology of $\Gamma$: 
\begin{align}\label{eq:isom sheaf and group cohomologies}
H^{1}(Y^\BS, \cM_{m, p}/p^r) &\stackrel{\sim}{\longrightarrow} H^{1}(\Gamma, \cM_{m, p}/p^r). 
\end{align}
Moreover, the inner cohomology group $H^{1}_{!}(Y^\BS, \cM_{m, p}/p^r)$ corresponds to the parabolic subgroup $H^1_P(\Gamma, \cM_{m, p}/p^r)$ of $H^{1}(\Gamma, \cM_{m, p}/p^r)$ under the isomorphism \eqref{eq:isom sheaf and group cohomologies} (see \cite[(4.1a)]{Hida86} for the definition of the parabolic subgroup). 

Let $m \in \{n, n'\}$. 
Hida showed in \cite[Proposition 4.7]{Hida86} that  we have  isomorphisms 
    \begin{align}\label{eq:isom-T_p-ordinary-and-U_p-ordinary}
    \begin{split}
                    e_{T_p'}H^1(\Gamma, \cM_{m, p}/p^r ) &\stackrel{\sim}{\longrightarrow}  e_{U_p'}H^1(\Gamma_0(p^r), \cM_{m, p}/p^r ); x \mapsto e_{U_p'}  \mathrm{res}(x), 
            \\
             e_{T_p'}H^1_P(\Gamma, \cM_{m, p}/p^r ) &\stackrel{\sim}{\longrightarrow}  e_{U_p'}H^1_P(\Gamma_0(p^r), \cM_{m, p}/p^r ); x \mapsto e_{U_p'}  \mathrm{res}(x)
                 \end{split}
        \end{align}
    which are $T_\ell'$-equivariant for any prime number $\ell \neq p$. 
    Here $\mathrm{res}$ denotes the restriction map.

Let $L_{m, r}$ denote the $\Gamma_0(p^r)$-module whose underlying abelian group is $\bZ_p/(p^r)$ and the $\Gamma_0(p^r)$-action is given  by the homomorphism 
$\Gamma_0(p^r) \longrightarrow (\bZ_p/(p^r))^\times; \begin{pmatrix}
    a&b\\c&d
\end{pmatrix} \mapsto a^{m} \bmod{p^r}$. 
Then  Hida also showed in \cite[Corollary 4.5 and (6.8)]{Hida86} 
that the $\Gamma_0(p^r)$-homomorphism  $i_r \colon \cM_{m, p}/p^r \longrightarrow L_{m, r}; f(X_1,X_2) \mapsto f(1, 0)$ induces  Hecke-equivariant isomorphisms
    \begin{align}\label{eq:isom-hida-U_p-ordinary-n-and-n'}
    \begin{split}
             e_{U_p'}H^{1}(\Gamma_0(p^r), \cM_{m, p}/p^r) &\stackrel{\sim}{\longrightarrow} e_{U_p'}H^{1}(\Gamma_0(p^r), L_{m, r}); x \mapsto i_{r, *}(x), 
     \\
    e_{U_p'}H^{1}_P(\Gamma_0(p^r), \cM_{m, p}/p^r) &\stackrel{\sim}{\longrightarrow} e_{U_p'}H^{1}_P(\Gamma_0(p^r), L_{m, r}); x \mapsto i_{r, *}(x). 
        \end{split}
    \end{align}
 Since $n \equiv n' \pmod{(p-1)p^r}$, we have $L_{n, r} = L_{n', r}$ as $\Gamma_0(p^r)$-modules, by combining the isomorphisms \eqref{eq:isom-T_p-ordinary-and-U_p-ordinary} and \eqref{eq:isom-hida-U_p-ordinary-n-and-n'} for $m = n,n'$, we obtain the following commutative diagram:  
\small
 \begin{align*}\label{eq:isom-hida-T_p-ordinary-n-and-n'}
    \xymatrix@C=18pt{
        H^{1}_{!, \ord}(Y^\BS, \cM_{n, p}/p^r) \ar[r]^-{\cong} \ar@{^{(}->}[d]  & 
e_{U_p'}H^{1}_{P}(\Gamma_0(p^{r}), L_{n,r}) = e_{U_p'}H^{1}_{P}(\Gamma_0(p^{r}), L_{n',r}) \ar@{^{(}->}[d] 
& \ar[l]_-{\cong} H^{1}_{!, \ord}(Y^\BS, \cM_{n', p}/p^r) \ar@{^{(}->}[d]
        \\
        H^{1}_{\ord}(Y^\BS, \cM_{n, p}/p^r)
        \ar[r]^-{\cong} & 
e_{U_p'}H^{1}(\Gamma_0(p^{r}), L_{n,r}) =  e_{U_p'}H^{1}(\Gamma_0(p^{r}), L_{n',r}) & 
\ar[l]_-{\cong} H^{1}_{\ord}(Y^\BS, \cM_{n', p}/p^r), 
    } 
\end{align*}  
        \normalsize
        where horizontal arrows are isomorphisms and $T_\ell'$-equivariant for any prime number $\ell \neq p$. 
            This completes the proof. 
\end{proof}

\subsection{Another expression for $\Delta_p(\Eis_n)$}

Let $p \geq 5$ be a prime number and take a prime number $\ell \neq p$. 
Let 
\[
\cH_{\ell, p} := \bZ_{p}[X]
\]
be the polynomial ring over $\bZ_{p}$, and by using the Hecke operator $T_{\ell}'$ at $\ell$, we regard cohomology groups that appear in the present paper as $\cH_{\ell,p}$-modules. 
For notational simplicity, we put 
\[
x_{\ell, n} := X - (1+ \ell^{n+1}) \,\,\, \textrm{ and } \,\,\, \mathcal{B}_{\ell, p, n} := \cH_{\ell, p}/(x_{\ell, n}). 
\]
Note that by Corollary \ref{cor:H^1-ordinary-boundary}, we have $\mathcal{B}_{\ell, p, n} \stackrel{\sim}{\longrightarrow} H^{1}_{\ord}(\partial Y^\BS, \cM_{n, p}); 1 \mapsto e_{T_p'}[e_{n}]$ as $\cH_{\ell, p}$-modules.

\begin{lem}\label{lem:corollary-of-Ramanujam-conjecture-T_p-eigen-implies-Hecke-eigen}
Let  $c \in H^{1}(Y^\BS, \cM_{n}) \otimes \bC$ be a cohomology class.  
If $c|T_{\ell}' = (1 + \ell^{n+1})c$, then we have $c|T_{\ell'}' = (1 + \ell'^{n+1})c$ for any prime number $\ell'$, that is, the cohomology class $c$ is a scalar multiple of $\Eis_{n}$. 
\end{lem}
\begin{proof}
It is well-known that one can take a $T_\ell'$-Hecke-eigen basis $f_{1}, \ldots, f_{t} \in H^{1}(Y^\BS, \cM_{n}) \otimes \bC$ such that 
$f_{1} = \Eis_n$ and that the elements $f_2, \ldots, f_t$ correspond to either cusp forms or their complex conjugates via the Eichler--Shimura homomorphism. 
Then the Ramanujan conjecture proved by Deligne shows that the absolute value of the $T_\ell'$-eigenvalue of $f_i$ ($2 \leq i \leq t$) is less than $1 + \ell^{n+1}$, which implies this lemma. 
\end{proof}

\begin{lem}\label{lem:denominator-ordinary-Ann-x_p}
We have $H^{1}_{\ord}(Y^\BS, \cM_{n, p})[x_{\ell,n}] = \bZ_p\Delta_p(\Eis_{n})\Eis_{n}$.  
\end{lem}

\begin{proof}
    By Proposition \ref{prop:H^1-ord-torsion-free} and Lemma \ref{lem:corollary-of-Ramanujam-conjecture-T_p-eigen-implies-Hecke-eigen}, we have 
    $H^{1}_{\ord}(Y^\BS, \cM_{n, p})[x_{\ell,n}] \subset \bQ_p \cdot \Eis_n$. 
    Hence this lemma follows from the definition of the denominator $\Delta_p(\Eis_{n})$ of the Eisenstein class $\Eis_{n}$ and Lemmas \ref{lem:e_p-pairing-ordinary} and \ref{lem:ordinary-flat-integral=non-flat-ntegral}. 
\end{proof}

\begin{dfn}
We define $[\cE_{\ell, p, n}] \in \Ext^{1}_{\cH_{\ell, p}}(\mathcal{B}_{\ell, p, n}, H^{1}_{!, \ord}(Y^\BS, \cM_{n,p}))$ to be the element corresponding to the exact sequence of $\cH_{\ell, p}$-modules in Corollary \ref{cor:exact-seq-ordinary-H^1-inner-boundary} for $\cM = \cM_{n,p}$:  
\[
0 \longrightarrow H^{1}_{!, \ord}(Y^\BS, \cM_{n,p})   \longrightarrow H^{1}_{\ord}(Y^\BS, \cM_{n,p})    \longrightarrow  \mathcal{B}_{\ell, p, n} \longrightarrow 0. 
\]
\end{dfn}

The following lemma follows directly from Lemma \ref{lem:denominator-ordinary-Ann-x_p}. 

\begin{lem}\label{lem:ann-exact seq-cE_n,p}
$\Ann_{\bZ_{p}}([\cE_{\ell, p, n}]) = \Delta_{p}(\Eis_{n})\bZ_{p}$. 
\end{lem}

\begin{lem}\label{lem:ext1-identification}
We have a natural identification 
\[
H^{1}_{!, \ord}(Y^\BS, \cM_{n,p}) \otimes_{\cH_{\ell, p}} \mathcal{B}_{\ell, p, n} = \Ext^{1}_{\cH_{\ell, p}}(\mathcal{B}_{\ell, p, n}, H^{1}_{!, \ord}(Y^\BS, \cM_{n,p})).
\]
\end{lem}
\begin{proof}
Since $x_{\ell, n}$ is an regular element of $\cH_{\ell, p}$, 
we have an exact sequence of $\cH_{\ell, p}$-modules: 
\[
0 \longrightarrow \cH_{\ell, p} \xrightarrow{\times x_{\ell}} \cH_{\ell, p} \longrightarrow \mathcal{B}_{\ell, p, n} \longrightarrow 0.  
\]
Applying the functor $\Hom_{\cH_{\ell, p}}(-, H^{1}_{!, \ord}(Y^\BS, \cM_{n,p}))$ to this short exact sequence,  we obtain  the desired identification. 
\end{proof}

\begin{lem}\label{lem:harder-hida}
For any positive integers $r$ and $n'$ with $n \equiv n' \pmod{(p-1)p^{r-1}}$, we have a natural isomorphism of $\cH_{\ell, p}$-modules: 
\[
H^{1}_{!, \ord}(Y^\BS, \cM_{n,p}) \otimes_{\cH_{\ell, p}} \mathcal{B}_{\ell, p, n}/(p^r) \cong H^{1}_{!, \ord}(Y^\BS, \cM_{n',p}) \otimes_{\cH_{\ell, p}} \mathcal{B}_{\ell, p, n'}/(p^r). 
\]
Moreover, the image of $[\cE_{\ell, p, n}] \bmod{p^r}$ is $[\cE_{\ell, p, n'}] \bmod{p^r} $ under this isomorphism (and the identification in Lemma \ref{lem:ext1-identification}). 
\end{lem}

\begin{proof}
    This lemma follows from Theorem \ref{thm:harder-hida} and Corollary \ref{cor:conpatible-reduction-ordinary-inner}.
\end{proof}

\begin{dfn}
    We define a polynomial $\Phi_{\ell, n}(t) \in \bZ_p[t]$ to be the characteristic polynomial associated with  $T_\ell' \colon H^1_{!, \ord}(Y^\BS, \cM_{n, p}) \longrightarrow H^1_{!, \ord}(Y^\BS, \cM_{n, p})$: 
    \[
    \Phi_{\ell, n}(t) := \det(t \cdot \mathrm{id} - T_\ell' \mid H^1_{!, \ord}(Y^\BS, \cM_{n, p})). 
    \]
\end{dfn}

\begin{lem}\label{lem:annihilated by Phi_(ell, n)(1+ell^(n+1))}
    The $\bZ_p$-module $H^{1}_{!, \ord}(Y^\BS, \cM_{n,p}) \otimes_{\cH_{\ell, p}} \mathcal{B}_{\ell, p, n}$ is annihilated by $\Phi_{\ell, n}(1+\ell^{n+1})$. 
\end{lem}
\begin{proof}
     By the Cayley--Hamilton theorem, the Hecke module $H^{1}_{!, \ord}(Y^\BS, \cM_{n,p})$ is annihilated by  $\Phi_{\ell, n}(T_\ell')$. 
    Hence the $\bZ_p$-module $H^{1}_{!, \ord}(Y^\BS, \cM_{n,p}) \otimes_{\cH_{\ell, p}} \mathcal{B}_{\ell, p, n}$ is annihilated by $\Phi_{\ell, n}(1+\ell^{n+1})$ since $\mathcal{B}_{\ell, p, n} = \cH_{\ell, p}/(X - (1+\ell^{1+n}))$. 
\end{proof}

\begin{lem}\label{lem:ord-equal-Phi-n-and-n'}
    Let $r$ be a positive integer satisfying $r > \ord_p(\Phi_{\ell, n}(1+\ell^{n+1}))$. 
    Then for any even integer $n' \geq 2$ with $n \equiv n' \pmod{(p-1)p^{r-1}}$, we have 
     \[
    \ord_p(\Phi_{\ell, n}(1+\ell^{n+1})) = \ord_p(\Phi_{\ell, n'}(1+\ell^{n'+1})). 
     \]
\end{lem}
\begin{proof}
By Theorem \ref{thm:harder-hida}, we have 
\[
\Phi_{\ell, n}(t) \equiv \Phi_{\ell, n'}(t) \pmod{p^r}. 
\]
The fact that $n \equiv n' \pmod{(p-1)p^{r-1}}$ implies that  $\ell^{n+1} \equiv \ell^{n'+1} \pmod{p^r}$, and we obtain $\Phi_{\ell, n}(1+ \ell^{n+1}) \equiv \Phi_{\ell, n'}(1+ \ell^{n'+1}) \pmod{p^r}$. 
Hence the assumption that $r > \ord_p(\Phi_{\ell, n}(1+\ell^{n+1}))$ shows that 
$\ord_p(\Phi_{\ell, n}(1+\ell^{n+1})) = \ord_p(\Phi_{\ell, n'}(1+\ell^{n'+1}))$. 
\end{proof}

\begin{prop}\label{prop:change-delta}
     Let $r$ be a positive integer satisfying $r > \ord_p(\Phi_{\ell, n}(1+\ell^{n+1}))$. 
    Then for any even integer $n' \geq 2$ with $n \equiv n' \pmod{(p-1)p^{r-1}}$, we have 
     \[
     \Delta_{p}(\Eis_n) = \Delta_{p}(\Eis_{n'}). 
     \]
\end{prop}
\begin{proof}
    By Lemmas \ref{lem:harder-hida}, \ref{lem:annihilated by Phi_(ell, n)(1+ell^(n+1))}, and \ref{lem:ord-equal-Phi-n-and-n'}, we have the  natural isomorphism 
    \begin{align*}
    H^{1}_{!, \ord}(Y^\BS, \cM_{n,p}) \otimes_{\cH_{\ell, p}} \mathcal{B}_{\ell, p, n} 
     &=
        H^{1}_{!, \ord}(Y^\BS, \cM_{n,p}) \otimes_{\cH_{\ell, p}} \mathcal{B}_{\ell, p, n}/(p^r) 
        \\
        &\cong H^{1}_{!, \ord}(Y^\BS, \cM_{n',p}) \otimes_{\cH_{\ell, p}} \mathcal{B}_{\ell, p, n'}/(p^r)
        \\
        &= H^{1}_{!, \ord}(Y^\BS, \cM_{n',p}) \otimes_{\cH_{\ell, p}} \mathcal{B}_{\ell, p, n'}.
    \end{align*}
    Moreover, the image of $[\cE_{\ell, p, n}]$ under this isomorphism is $[\cE_{\ell, p, n'}]$, and Lemma \ref{lem:ann-exact seq-cE_n,p} implies that 
    \begin{align*}
        \Delta_{p}(\Eis_n)\bZ_p = \Ann_{\bZ_p}([\cE_{\ell, p, n}]) = \Ann_{\bZ_p}([\cE_{\ell, p, n'}]) = \Delta_{p}(\Eis_{n'})\bZ_p. 
    \end{align*}
\end{proof}

\section{Kubota--Leopoldt $p$-adic $L$-function}\label{sec:p-adic L}

Let $p$ be a prime number. 
In this section, we introduce the Kubota--Leopoldt $p$-adic $L$-functions and prove certain congruence properties that will be used in the proof of Theorem \ref{thm:main result}.

Let $\omega \colon \Gal(\overline{\bQ}/\bQ) \longrightarrow \bZ_p^\times$ denote the Teichm\"uller character, and
let $\mathbbm{1}$ denote the trivial character. 
For any Dirichlet character $\chi$, we denote by $L_{p}(s, \chi) \in \bC_{p}[[s]]$ the Kubota--Leopoldt $p$-adic $L$-function attached to $\chi$. 

\begin{prop}[{\cite[Theorems 5.11 and 5.12, Exercises 5.11(1)]{LW97}
}]\label{prop:p-adic-zeta}\
\begin{itemize}
\item[(1)] For any  Dirichlet character $\chi$, the $p$-adic $L$-function $L_{p}(s, \chi)$ converges on $\bZ_p \setm \{1\}$. 
Moreover, for any integer $m \geq 2$, we have 
\[
L_{p}(1-m, \chi) = (1-\chi \omega^{-m}(p) p^{m-1})L(1-m, \chi\omega^{-m}). 
\]
In particular, we have $\mathrm{ord}_{p}(L_{p}(1-m, \omega^{m})) = \mathrm{ord}_{p}(\zeta(1-m))$. 
\item[(2)] We have 
\[
L_{p}(s, \mathbbm{1}) \in \frac{p-1}{p(s-1)} + \bZ_p[[s-1]]
\]
\item[(3)] If $m \not\equiv 0 \pmod{p-1}$, then we have 
\[
L_{p}(s, \omega^{m}) \in \bZ_p + p\bZ_p[[s-1]]. 
\]
\end{itemize}
\end{prop}

By using the Kubota--Leopoldt $p$-adic $L$-functions, Theorem \ref{thm:Tp^m and Eis}  can be restated as follows.

\begin{cor}\label{cor:reformulation-of-THeorem-Tp^m-and-Eis-pairing}
If we put 
   \begin{align*}
D_p(n,\nu) := \frac{L_{p}(-\nu, \omega^{1+\nu})L_{p}(\nu-n, \omega^{n-\nu+1})}{L_{p}(-1-n, \omega^{n+2})} - L_{p}(-\nu, \omega^{1+\nu}) - L_{p}(\nu-n, \omega^{n-\nu+1}),   
\end{align*}
then for any integer $\nu \in \{1, \ldots, n-1\}$ we have 
   \begin{align*}
   \lim_{m \to \infty}\brk{\Eis_n, \widetilde{T_p^{m!}(C_{\nu}(\tau))}}  = \frac{1-p^{n+1}}{(1-p^\nu)(1-p^{n-\nu})} D_p(n,\nu). 
   \end{align*}
\end{cor}

For any even integer $m$, we define a positive integer $N_m$ by  
\[
N_{m} := \textrm{the numerator of $\zeta(1-m)$}. 
\]

\begin{cor}\label{cor:change-zeta}
Let $m \geq 2$ be an even integer.  
\begin{itemize}
    \item[(1)] If  $m \not\equiv 0 \pmod{p-1}$, then we  have $\mathrm{ord}_{p}(\zeta(1-m)) = \mathrm{ord}_{p}(N_{m})$. 
    \item[(2)] If  $m \equiv 0 \pmod{p-1}$, then we  have $\mathrm{ord}_{p}(N_{m}) = 0$. 
 \item[(3)] 
Let $r$ and $m'$ be positive integers  with $m \equiv m' \pmod{(p-1)p^{r-1}}$. 
If $r > \mathrm{ord}_{p}(N_{m})$, then 
\[
\mathrm{ord}_{p}(N_{m}) = \mathrm{ord}_{p}(N_{m'}). 
\] 
\end{itemize}

\end{cor}
\begin{proof}
Claims (1) and (2)  follows immediately from Proposition \ref{prop:p-adic-zeta}. 
When $m \equiv 0 \pmod{p-1}$, claim (3) follows from claim (2). 
When $m \not\equiv 0 \pmod{p-1}$, claim (3) follows from 
Proposition \ref{prop:p-adic-zeta}. 
\end{proof}

\begin{cor}\label{cor:case1}
Let $x$ be an integer with $x \not\equiv 0 \pmod{p-1}$.
For any integer $y$, we have 
\begin{align*}
\frac{L_{p}(1-x, \omega^{x})L_{p}(1-y, \mathbbm{1})}{L_{p}(1-x-y, \omega^{x})} 
\in L_{p}(1-y, \mathbbm{1}) + \frac{\bZ_{p}}{L_{p}(1-x-y, \omega^{x})}. 
\\
\end{align*}
\end{cor}
\begin{proof}
Since $x \not\equiv 0 \pmod{p-1}$, by Proposition \ref{prop:p-adic-zeta}(3), we have 
\[
L_{p}(1-x, \omega^{x}) \in L_{p}(1-x-y, \omega^{x}) + py\bZ_{p}. 
\]
Moreover, by Proposition \ref{prop:p-adic-zeta}(2), we have $py L_{p}(1-y, \mathbbm{1}) \in 1 + p\bZ_{p}$, which shows 
\begin{align*}
\frac{L_{p}(1-x, \omega^{x})L_{p}(1-y, \mathbbm{1})}{L_{p}(1-x-y, \omega^{x})} 
\in L_{p}(1-y, \mathbbm{1}) + \frac{\bZ_{p}}{L_{p}(1-x-y, \omega^{x})}. 
\end{align*}
\end{proof}

\begin{cor}\label{cor:case2}
For any integers $x$ and $y$, we have 
\begin{align*}
\frac{L_{p}(1-x, \mathbbm{1})L_{p}(1-y, \mathbbm{1})}{L_{p}(1-x-y, \mathbbm{1})} &\equiv \frac{p-1}{px} +  \frac{p-1}{py} \pmod{\bZ_{p}}
\\
&\equiv L_{p}(1-x, \mathbbm{1}) + L_{p}(1-y, \mathbbm{1}) \pmod{\bZ_{p}}. 
\end{align*}
\end{cor}

\begin{proof}
For notational simplicity, we put 
\[
R(s) :=  -\frac{p-1}{ps} \,\,\, \textrm{ and } \,\,\, H(s-1) :=  L_{p}(s, \mathbbm{1}) - R(1-s). 
\]
By Proposition \ref{prop:p-adic-zeta}(2), we have $H(s) \in \bZ_{p}[[s]]$, $xR(x), yR(y), (x+y)R(x+y) \in p^{-1}\bZ_{p}^{\times}$, and $R(x+y)^{-1} = R(x)^{-1} + R(y)^{-1}$.  
Since $H(s) \in \bZ_{p}[[s]]$, we have 
\begin{align*}
H(x) 
\in H(x+y) + y\bZ_{p}, \quad 
H(y) 
\in H(x+y) + x\bZ_{p}. 
\end{align*}
Put $\alpha := 1 + R(x+y)^{-1}H(x+y) \in 1 + p\bZ_{p}$. 
Then  
\begin{align*}
\frac{L_{p}(1-x, \mathbbm{1})L_{p}(1-y, \mathbbm{1})}{L_{p}(1-x-y, \mathbbm{1})} 
\in \frac{(R(x) + H(x+y) + y\bZ_{p})(R(y) + H(x+y) + x\bZ_{p})}{R(x+y) + H(x+y)}
\end{align*}
and we have 
\begin{align*}
&\quad \frac{(R(x) + H(x+y) + y\bZ_{p})(R(y) + H(x+y) + x\bZ_{p})}{R(x+y) + H(x+y)}
\\
&\subset \frac{R(x)^{-1} + R(y)^{-1}}{\alpha} (R(x)R(y) + (R(x) + R(y))H(x+y) + p^{-1}\bZ_{p}) 
\\
&= R(x) + R(y)  + \bZ_p. 
\end{align*}
\end{proof}

\section{Proof of Theorem \ref{thm:main result}}\label{sec:proof of main thm}

Let $p$ be a prime number. 
As in Corollary \ref{cor:reformulation-of-THeorem-Tp^m-and-Eis-pairing}, for any integer $1 \leq \nu \leq n-1$, we define 
\begin{align*}
D_p(n,\nu) := \frac{L_{p}(-\nu, \omega^{1+\nu})L_{p}(\nu-n, \omega^{n-\nu+1})}{L_{p}(-1-n, \omega^{n+2})} - L_{p}(-\nu, \omega^{1+\nu}) - L_{p}(\nu-n, \omega^{n-\nu+1}) 
\end{align*}
and set 
\begin{align*}
\delta_p(n,\nu) :=  \max \left\{- \mathrm{ord}_{p}(D_p(n,\nu)), 0 \right\}. 
\end{align*}

\begin{prop}\label{prop:delta-p-Eisenstein=max-delta-p-n-nu}
We have 
\[
\delta_{p}(\Eis_n) = \max_{1 \leq \nu \leq n-1} \delta_p(n,\nu). 
\]
\end{prop}
\begin{proof}
By Corollary \ref{cor:reformulation-of-THeorem-Tp^m-and-Eis-pairing}, we have 
   \begin{align*}
  \lim_{m \to \infty}\brk{\Eis_n, \widetilde{T_p^{m!}(C_{\nu}(\tau))}} = \frac{1-p^{n+1}}{(1-p^\nu)(1-p^{n-\nu})}D_p(n,\nu). 
      \end{align*}  
     Hence, for any sufficiently large integer $m$, we have 
    \[
    \mathrm{ord}_p\left(\brk{\Eis_n, \widetilde{T_p^{m!}(C_{\nu}(\tau))}}\right) = 
    \mathrm{ord}_{p}(D_p(n,\nu)), 
    \]
and this proposition follows from Corollary \ref{cor:ordinary Eisenstein denominator}. 
\end{proof}

Recall that $N_{m}$ denotes the numerator of $\zeta(1-m)$.

\begin{prop}\label{prop:pre-theorem}
Let $p$ be a prime number. 
\begin{itemize}
\item[(1)] $\delta_{p}(\Eis_n) \leq \mathrm{ord}_{p}(N_{n+2})$. 
\item[(2)] If $p<n$, then $\delta_{p}(\Eis_n) = \mathrm{ord}_{p}(N_{n+2})$. 
\end{itemize}
\end{prop}

The proof of Proposition \ref{prop:pre-theorem} is given in \S\ref{sec:proof of Proposition pre-thm}. 
First, we give the proof of Theorem \ref{thm:main result} assuming Proposition \ref{prop:pre-theorem}, that is, we show that $\Delta(\Eis_n) = N_{n+2}$.

\begin{proof}[Proof of Theorem \ref{thm:main result}]
Take a prime number $p$. It suffices to show that $\delta_p(\Eis_n) = \ord_p(N_{n+2})$. 
When $p-1 \mid n+2$, by Proposition \ref{prop:pre-theorem} we have  $0 \leq \delta_{p}(\Eis_n)  \leq \mathrm{ord}_{p}(N_{n+2}) = 0$, and hence  we may assume that $n \not\equiv -2 \pmod{p-1}$. 
Note that $p \geq 5$ in this case. 
Take a prime number $\ell \neq p$, and positive integers $r$ and $n'$ satisfying 
\begin{itemize}
\item $r  > \mathrm{ord}_{p}(\Phi_{\ell, n}(1+\ell^n))$, 
\item $p < n'$, 
\item $n \equiv n' \pmod{p^{r-1}(p-1)}$. 
\end{itemize}
Then by Propositions \ref{prop:pre-theorem}(2), we have $\delta_{p}(\Eis_{n'}) =  \mathrm{ord}_{p}(N_{n'+2})$, and  
Proposition \ref{prop:change-delta} implies that 
\[
\delta_{p}(\Eis_n) = \delta_{p}(\Eis_{n'}) =  \mathrm{ord}_{p}(N_{n'+2}). 
\]
Hence Corollary \ref{cor:change-zeta} shows that $\delta_{p}(\Eis_n) =  \mathrm{ord}_{p}(N_{n'+2}) = \mathrm{ord}_{p}(N_{n+2})$. 
\end{proof}

\subsection{Proof of Proposition \ref{prop:pre-theorem}}\label{sec:proof of Proposition pre-thm}

In this subsection, we prove Proposition \ref{prop:pre-theorem}. 
The proof is divided into the following two cases: 
\begin{itemize}
    \item $p-1 \nmid n+2$, 
    \item $p-1 \mid n+2$. 
\end{itemize}

\subsubsection{$p-1 \nmid n+2$}

\begin{lem}\label{lem:case1-p-1-nmid-n+2}
If $p-1 \nmid n+2$,  then we have $\delta_p(\Eis_n) \leq \mathrm{ord}_{p}(N_{n+2})$. 
\end{lem}
\begin{proof}
Take an integer $\nu  \in \{1, \ldots, n-1\}$. 
When $p-1 \nmid 1+\nu$ and  $p-1 \nmid n-\nu + 1$,  both $L_p(-\nu, \omega^{1+\nu})$ and $L_p(\nu-n, \omega^{n-\nu+1})$ are $p$-adic integers, and hence we have 
\[
\delta_p(n,\nu) \leq \mathrm{ord}_{p}(L_p(-1-n, \omega^{n+2})) = \mathrm{ord}_{p}(N_{n+2}). 
\]
Suppose $p-1 \mid 1+\nu$ (resp. $p-1 \mid n-\nu + 1$). Then since $p-1 \nmid n+2$, we see that $n-\nu + 1$ (resp. $1+\nu$) is not divisible by $p-1$. Therefore, 
Corollary \ref{cor:case1} shows that 
\begin{align*}
D_p(n,\nu) \in \frac{\bZ_{p}}{L_{p}(-1-n, \omega^{n+2})} + \bZ_p,  
\end{align*}
which implies that $\delta_p(n,\nu) \leq \mathrm{ord}_{p}(N_{n+2})$. 
Hence Proposition \ref{prop:delta-p-Eisenstein=max-delta-p-n-nu} implies this lemma. 
\end{proof}

Moreover, if $p<n$, the result of Carlitz concerning the index of irregularity of a prime shows the following lemma. 

\begin{lem}\label{lem:case1-p-1-nmid-n+2-equality}
    If $p-1 \nmid n+2$ and $p<n$, there is an (odd) integer $\nu \in \{1, \ldots, n-1\}$ such that $\delta_p(n,\nu)  = \mathrm{ord}_{p}(N_{n+2})$. 
    In particular, we have
    $\delta_{p}(\Eis_n) = \mathrm{ord}_{p}(N_{n+2})$ in this case. 
\end{lem}
\begin{proof}
By Lemma \ref{lem:case1-p-1-nmid-n+2}, for any regular prime $p$ and (odd) integer $\nu \in \{1, \ldots, n-1\}$, we have $\delta_p(n,\nu) = \mathrm{ord}_{p}(N_{n+2}) = 0$. 
Therefore,  we may assume that $p$ is an irregular prime. In particular, $p \geq 37$. 

    We define the index $d(p)$ of irregularity of the prime number $p$ by 
\begin{align*}
d(p) 
&:= \#\{1 \leq t \leq p-3 \mid t \in 2\Z, B_{t} \in p\bZ_{p}\} \\
&=
\#\{1 \leq t \leq p-3 \mid t \in 2\Z, L_p(1-t, \omega^{t}) \in p\bZ_p \}. 
\end{align*}
Then by using the result of Carlitz in \cite[(21)]{Car61}, Skula proved in \cite[Theorem 2.2, Remark 2.3]{Sku80} that 
\[
d(p) < \frac{p+3}{4} - \frac{\log(2)}{\log(p)}\frac{p-1}{4}. 
\]
Hence if $p \geq  47$, then we have 
\[
d(p) < \frac{p-5}{4}. 
\]
Moreover, since the only irregular prime smaller than $47$ is $37$ and $d(37) = 1 < (37-5)/4$, 
the inequality $d(p) < (p-5)/4$ holds true. 

For any integer $a$, we define an integer $[a]_{p-1}$ by 
\[
0 \leq [a]_{p-1} \leq p-2 \,\,\, \textrm{ and } \,\,\, [a]_{p-1} \equiv a \pmod{p-1}. 
\]
Since $d(p) < (p-5)/4$, there is an even integer $t \in \{2, 4, \ldots, p-3\}$ with $t \neq [n+2]_{p-1}$ such that 
\[
L_p(1-t, \omega^{t}) \in \bZ_p^\times \,\,\, \textrm{ and } \,\,\, 
L_p(1-[n+2 - t]_{p-1}, \omega^{[n+2 - t]_{p-1}}) \in \bZ_p^\times. 
\]
Furthermore, Proposition \ref{prop:p-adic-zeta}(3) shows that 
\[
L_p(1-[n+2 - t]_{p-1}, \omega^{[n+2 - t]_{p-1}}) - 
L_p(-1+t-n, \omega^{n+2-t}) \in p\bZ_p, 
\]
and hence we have $L_p(-1+t-n, \omega^{n+2-t}) \in \bZ_p^\times$. 
Therefore, we put $\nu := t-1$ and get 
\[
\delta_p(n, \nu) = \mathrm{ord}_p(N_{n+2}). 
\]
\end{proof}

\subsubsection{$p-1 \mid n+2$}

\begin{lem}\label{lem:case2-p-1-mid-n+2}
    If $p-1 \mid n+2$, we have $\delta_p(\Eis_n) = 0 = \mathrm{ord}_p(N_{n+2})$.  
\end{lem}
\begin{proof}
The fact that $\mathrm{ord}_{p}(N_{n+2}) = 0$ follows from Corollary \ref{cor:change-zeta}(2). 
Hence by Proposition \ref{prop:delta-p-Eisenstein=max-delta-p-n-nu}, it suffices to show that $\delta_p(n,\nu) = 0$ for any integer $1 \leq \nu \leq n-1$. 
Since  $p-1 \mid n+2$, we have $\mathrm{ord}_p(L_p(-1-n, \omega^{n+2})) < 0$. 
If $p-1 \nmid 1+\nu$, then we also have $p-1\nmid n-\nu + 1$, and  hence we get $\delta_p(n,\nu) = 0$ since $L_p(-\nu, \omega^{1+\nu})$ and $L_p(\nu-n, \omega^{n-\nu+1})$ are $p$-adic integers. 
When $p-1 \mid 1+\nu$, we also have $p-1\mid n-\nu + 1$, and 
 Corollary \ref{cor:case2} implies that $D_p(n,\nu) \in \bZ_{p}$. Hence, again, we obtain  $\delta_p(n,\nu) = 0$. 
\end{proof}

 This completes the proof of Proposition \ref{prop:pre-theorem}, and in particular of Theorem \ref{thm:main result}.

\section{Applications}

In this section, we discuss some applications of Theorem \ref{thm:main result}. 
For notational simplicity, in the following, the (co)homology groups will be denoted by $H^{\bullet}(Y, \mcM_n)$ (resp.  $H_{\bullet}(Y, \mcM_n)$) rather than $H^{\bullet}(Y^{\BS}, \mcM_n)$ (resp. $H_{\bullet}(Y^{\BS}, \mcM_n)$) since they are naturally isomorphic. 

First, note that we have the following corollary of Theorem \ref{thm:main result}. 
\begin{cor}\label{cor:pairing with simple cycles}
Let $n\geq 2$ be an  even integer and $\gamma \in \Gamma$ a matrix. 
Take a polynomial $P(X_1, X_2) \in \mcM_n$ such that $\gamma P(X_1, X_2) = P(X_1, X_2)$. 
Then for any element $\tau \in \bbH$, we have
\begin{align*}
N_{n+2}
\int_{\tau}^{\gamma \tau}
E_{n+2}(z) 
P(z,1) \, dz
\in \Z. 
\end{align*}
Here $N_{n+2} > 0$ is the numerator of $\zeta(-1-n)$. 
\end{cor}

\begin{proof}
Since $\gamma P=P$, we have $\partial(\{\tau , \gamma\tau\}\otimes P) =0$, 
and hence $\{\tau , \gamma\tau\}\otimes P$ defines an element in the homology group $H_1(Y, \mcM_n)$. 
Therefore, by Theorem \ref{thm:main result}, we obtain
\begin{align*}
N_{n+2}
\int_{\tau}^{\gamma \tau}
E_{n+2}(z) P(z,1) \, dz=\brk{N_{n+2}\Eis_n, [\{\tau , \gamma\tau\}\otimes P]}\in \Z. 
\end{align*}
\end{proof}

\subsection{Duke's conjecture}\label{sec:duke's conjecture} 

In the paper \cite{Duke23}, Duke defined a certain map called the higher Rademacher symbol
\[
\Psi_{k} \colon  \Gamma \longrightarrow \Q 
\]
for each integer $k \in \Z_{\geq 2}$ which is a generalization of the classical Rademacher symbol and gave a conjecture concerning the integrality of the higher Rademacher symbol $\Psi_{k}$. 
\begin{conj}[{\cite[Conjecture, p. 4]{Duke23}}]\label{conj:duke}
For any integer  $k \in \Z_{\geq 2}$ and matrix $\gamma \in \Gamma$, we have 
\[
\Psi_k(\gamma) \in \Z. 
\]
\end{conj}

In the following, we show that Duke's Conjecture \ref{conj:duke} follows from Theorem \ref{thm:main result}.

\begin{rmk}
Conjecture \ref{conj:duke} is recently proved also by O'Sullivan using a more direct method (see \cite{Sullivan23}). 
\end{rmk}

Here, instead of giving the original definition of the higher Rademacher symbols, we recall an integral representation of the higher Rademacher symbols, also given by Duke in \cite{Duke23}, which is equivalent to the original definition and more suitable for our purpose.

\begin{prop}[{\cite[Definition (2.4) and Lemma 6]{Duke23}}]\label{prop:rademacher integral}
Let $k \in \Z_{\geq 2}$ be an integer. 
For any matrix $\gamma :=
\begin{pmatrix}
    a&b\\
    c&d
\end{pmatrix}
\in \Gamma$, 
we define a binary quadratic polynomial $Q_{\gamma}(X_1, X_2) \in \mcM_2$ associated with $\gamma$ by
\[
Q_{\gamma}(X_1, X_2) := -\frac{\sgn(a+d)}{\mathrm{gcd}(c,a-d,b)}(cX_1^2-(a-d)X_1X_2-bX_2^2). 
\]
Then for any element $\tau \in \bbH$, we have
\[
\Psi_k(\gamma) = N_{2k} \int_{\tau}^{\gamma \tau} E_{2k}(z)Q_{\gamma}(z,1)^{k-1}\, dz, 
\]
where $N_{2k}>0$ is the numerator of $\zeta(1-2k)$. 
\end{prop}

\begin{cor}\label{cor:duke's conjecture}
Duke's Conjecture \ref{conj:duke} holds true. 
\end{cor}
\begin{proof}
By definition, the binary quadratic  polynomial  $Q_{\gamma}(X_1,X_2)$ defined in Proposition \ref{prop:rademacher integral} is  $\gamma$-invariant. 
Hence  Corollary \ref{cor:pairing with simple cycles} and Proposition \ref{prop:rademacher integral} imply that  $\Psi_k(\gamma) \in \Z$. 
\end{proof}

\subsection{Partial zeta functions of real quadratic fields}\label{sec:denominator of partial zeta functions}

In this subsection, we discuss an application to the denominators of the special values of the partial zeta functions of real quadratic fields.

Let $F$ be a real quadratic field, and let $\mcO \subset F$ be an order of $F$ with  discriminant $D_{\mcO}$. 
We denote by $I_{\mcO}$ the group of proper fractional $\mcO$-ideals and  $P_{\mcO}^+ \subset I_{\mcO}$ the subgroup of totally positive principal ideals. 
We define the narrow ideal class group $Cl_{\mcO}^+$ of $\mcO$ by 
\[
Cl_{\mcO}^+:=I_{\mcO}/P_{\mcO}^+. 
\]
See \cite[\S 7]{Cox13}. 
We fix an embedding $F \subset \R$, and for any element $\alpha \in F \subset \R$, we denote by $\alpha' \in F \subset \R$ its conjugate over $\Q$. 

Moreover, let $\mcO_{+}^{\times}$ denote the group  of totally positive units in $\mcO$, and let $\varepsilon_0 \in \mcO_{+}^{\times}$ denote the generator of $\mcO_{+}^{\times}$ such that $\varepsilon_0> 1$. 

\begin{dfn}\label{def:definition of mfz}
We define a map
\begin{align*}
\mfz_{\mcO,k} \colon Cl_{\mcO}^+ \longrightarrow  H_1(Y, \mcM_{2k-2})
\end{align*}
as follows: 
Let $\mcA \in Cl_{\mcO}^+$, and take a representative $\mfa \in I_{\mcO}$ of $\mcA$. 
We also take a basis $\alpha_1, \alpha_2 \in \mfa$ over $\Z$ such that $\alpha_1 \alpha'_2 - \alpha'_1\alpha_2>0$, and let $\gamma_0 \in \Gamma$ be a matrix such that
\begin{align*}
\gamma_0 \begin{pmatrix}
    \alpha_1\\ \alpha_2
\end{pmatrix} = \begin{pmatrix}
    \varepsilon_0\alpha_1\\ \varepsilon_0\alpha_2
\end{pmatrix}.  
\end{align*}
Moreover, set
\begin{align*}
N_{\alpha_1, \alpha_2}(X_1, X_2):=
-
\frac{1}{N\mfa}
(\alpha_2 X_1-\alpha_1 X_2)
(\alpha'_2X_1-\alpha'_1X_2). 
\end{align*}
We see that $N_{\alpha_1, \alpha_2}(X_1, X_2) \in \mcM_2$ and that $\gamma_0 N_{\alpha_1, \alpha_2}(X_1, X_2) = N_{\alpha_1, \alpha_2}(X_1, X_2)$.  
We then define
\begin{align*}
\mfz_{\mcO,k}(\mcA) := [\{\tau, \gamma_0 \tau\}\otimes N_{\alpha_1, \alpha_2}(X_1, X_2)^{k-1}] \in H_1(Y, \mcM_{2k-2}), 
\end{align*}
where $\tau$ is an arbitrary element in $\bbH$. 
\end{dfn}

\begin{lem}\label{lem:well-definedness of the map fmz}
The homology class $\mfz_{\mcO,k}(\mcA)$ does not depend on the choices we made. 
\end{lem}
\begin{proof}
The independence of $\tau \in \bbH$ is clear. 
Let $\mfb \in \mcA$ be another representative. Then there exists a totally positive element $\alpha \in F^\times$ such that $\mfb = \alpha \mfa$. 
Take a basis $\beta_1, \beta_2 \in \mfa$ over $\bZ$ with $\beta_1 \beta'_2 - \beta'_1\beta_2>0$. 
Then we obtain a matrix $\gamma_{0, \mfb} \in \Gamma$ and a binary quadratic polynomial $N_{\alpha\beta_1, \alpha\beta_2}(X_1, X_2)$ from the basis $\alpha\beta_1, \alpha\beta_2$ of $\mfb$. 
Note that since $\alpha$ is totally positive, we have $(\alpha\beta_1) (\alpha'\beta'_2) - (\alpha'\beta'_1)(\alpha\beta_2)>0$. 
Let $\gamma \in \mathrm{GL}_2(\bZ)$ be a matrix satisfying 
\[
\begin{pmatrix}\beta_1 \\ \beta_2 \end{pmatrix} = \gamma \begin{pmatrix} \alpha_1 \\ \alpha_2 \end{pmatrix}. 
\]
Then the facts that $\alpha_1 \alpha'_2 - \alpha'_1\alpha_2>0$ and $\beta_1 \beta'_2 - \beta'_1\beta_2>0$ imply that $\gamma \in \Gamma$. 
Since $\gamma_{0, \mfb} \gamma \begin{pmatrix}\alpha_1 \\ \alpha_2 \end{pmatrix} = \gamma \begin{pmatrix} \varepsilon_0 \alpha_1 \\ \varepsilon_0\alpha_2 \end{pmatrix}$, we have 
$\gamma^{-1}\gamma_{0, \mfb}\gamma = \gamma_0$. 
Moreover, we have 
\[
\begin{pmatrix}X_1 & X_2 \end{pmatrix} {}^t\widetilde{\gamma} \begin{pmatrix}0&1\\-1&0\end{pmatrix} \begin{pmatrix}\alpha_1 \\ \alpha_2 \end{pmatrix} 
= \begin{pmatrix}X_1 & X_2 \end{pmatrix}  \begin{pmatrix}0&1\\-1&0\end{pmatrix} \gamma  \begin{pmatrix}\alpha_1 \\ \alpha_2 \end{pmatrix} 
= \begin{pmatrix}X_1 & X_2 \end{pmatrix}  \begin{pmatrix}0&1\\-1&0\end{pmatrix}  \begin{pmatrix}\beta_1 \\ \beta_2 \end{pmatrix}, 
\]
which implies that $N_{\alpha\beta_1, \alpha\beta_2}(X_1, X_2) = \gamma N_{\alpha_1, \alpha_2}(X_1, X_2)$. Therefore, we have
\begin{align*}
    [\{\tau, \gamma_{0, \mfb} \tau\}\otimes N_{\alpha\beta_1, \alpha\beta_2}(X_1, X_2)^{k-1}] 
    &= [\{\tau, \gamma \gamma_0 \gamma^{-1} \tau\}\otimes  \gamma N_{\alpha_1, \alpha_2}(X_1, X_2)^{k-1}]
    \\
    &= [\{\gamma^{-1} \tau, \gamma_0 \gamma^{-1} \tau\}\otimes N_{\alpha_1, \alpha_2}(X_1, X_2)^{k-1}]
    \\
    &= [\{\tau, \gamma_0\tau\}\otimes N_{\alpha_1, \alpha_2}(X_1, X_2)^{k-1}]
\end{align*}
as elements of $H_1(Y, \mcM_{2k-2})$. 
\end{proof}

\begin{rmk}\label{rem:Gauss's theory hyperbolic elements}\
\begin{itemize}
    \item[(1)] Since the matrix $\gamma_0$ in Definition \ref{def:definition of mfz} is hyperbolic  (i.e., $|\mathrm{trace}(\gamma_0)| > 2$), we have $\dim_{\bQ} \{Q \in \cM_{2} \otimes \bQ \mid \gamma Q = Q\} = 1$. 
    This fact together with \cite[(7.6)]{Cox13} shows that $N_{\alpha_1, \alpha_2}(X_1, X_2) = Q_{\gamma_0}(X_1, X_2)$. 

    \item[(2)] Gauss's theory concerning binary quadratic forms (see \cite[Exercise 7.21]{Cox13} for example) shows that for any hyperbolic element $\gamma \in \Gamma$, there is an order $\cO$ of a real quadratic field and a narrow ideal class $\mcA \in Cl_\cO^+$ such that
    \[
    [\{z, \gamma z\} \otimes Q_\gamma(X_1, X_2)] \in \bZ \mfz_{\cO, 2}(\mcA). 
    \]
\end{itemize}    
\end{rmk}

\begin{dfn}
   For each ideal class $\mcA \in Cl_{\mcO}^+$, the partial zeta function $\zeta_{\mcO}(\mcA,s)$ associated with $\mcA$ is defined by 
\begin{align*}
\zeta_{\mcO}(\mcA,s)
:=
\sum_{\mfa \subset \mcO, \mfa \in \mcA}
\frac{1}{(N\mfa)^s},
\quad\quad
(\re(s)>1),  
\end{align*}
and it is well-known that $\zeta_{\mcO}(\mcA,s)$ can be continued meromorphically to $s \in \C$ and has a simple pole at $s=1$. 

\end{dfn}

The following integral representation of the special values of the partial zeta function is classically known. 

\begin{prop}\label{prop:Hecke integral formula}
For any integer  $k \in \Z_{\geq 2}$ and ideal class $\mcA \in Cl_{\mcO}^+$, we have 
\begin{align*}
\brk{\Eis_{2k-2}, \mfz_{\mcO,k}(\mcA)}
= (-1)^k \frac{\zeta_{\mcO}(\mcA^{-1},1-k)}{\zeta(1-2k)}. 
\end{align*}
\end{prop}

Before we give a proof of Proposition \ref{prop:Hecke integral formula}, we recall (a special case of) the so-called Feynman parametrization. 

\begin{lem}\label{lem:feynman}
Let $x_1, x_2, a, b \in \C$  be complex numbers such that $x_1a+x_2\neq 0$ and $x_1b+x_2 \neq 0$. Then for any non-negative integers $k_1$ and $k_2$, we have
\[
\int_a^b \frac{\left(b-z\right)^{k_1} \left(z-a\right)^{k_2}}{(x_1z+x_2)^{2+k_1+k_2}} \, dz 
= \frac{k_1!k_2!}{(k_1+k_2+1)!}\frac{(b-a)^{k_1+k_2+1}}{(x_1 a+x_2)^{k_1+1}(x_1 b+x_2)^{k_2+1}}. 
\]
\end{lem}
\begin{proof}
We may assume that $a \neq b$. By setting $y_1=x_1a+x_2$ and $y_2=x_1b+x_2$, it suffices to prove that 
\[
(b-a)
\int_a^b \frac{\left(b-z\right)^{k_1}\left(z-a\right)^{k_2}}{((b-z)y_1+(z-a)y_2)^{2+k_1+k_2}} \, dz 
= \frac{k_1!k_2!}{(k_1+k_2+1)!}\frac{1}{y_1^{k_1+1}y_2^{k_2+1}}.
\]
The case where $k_1=k_2=0$ is clear, i.e., we have
\[
(b-a)\int_a^b \frac{1}{((b-z)y_1+(z-a)y_2)^{2}} \, dz = \frac{1}{y_1y_2}.
\]
Then by viewing the both sides as holomorphic functions in $(y_1, y_2) \in \C^{\times}\times \C^{\times}$ and applying the differential operator $\lr{\frac{\partial}{\partial y_1}}^{k_1}\lr{\frac{\partial}{\partial y_2}}^{k_2}$, we obtain the desired identity. 
\end{proof}
\begin{proof}[{Proof of Proposition \ref{prop:Hecke integral formula}}]
We use the same notations as in Definition \ref{def:definition of mfz}. 
Since 
\[
2\zeta(2k)E_{2k}(z) = \sum_{(0,0) \neq (m,n) \in \Z^2} \frac{1}{(mz+n)^{2k}}
\]
and $\Eis_{2k-2} = r(E_{2k})$, for any element $\tau \in \bbH$ we have
\begin{align*}
2\zeta(2k)\brk{\Eis_{2k-2}, \mfz_{\mcO,k}(\mcA)} = \int_{\tau}^{\gamma_0 \tau}\sum_{(0,0) \neq (m,n)\in \Z^2}  \frac{N_{\alpha_1, \alpha_2}(z, 1)^{k-1}}{(mz+n)^{2k}} \, dz. 
\end{align*}
Since $N_{\alpha_1, \alpha_2}(\gamma_0  z, 1) = j(\gamma_0, z)^{-2}N_{\alpha_1, \alpha_2}(z, 1)$, where $j(\begin{pmatrix}a&b\\c&d\end{pmatrix}, z) := cz+d$ is the factor of automorphy, if we fix a complete set $S_{\gamma_0}$ of representatives of $(\Z^2  \setm \{(0, 0)\} )/\gamma_0^{\Z}$, then we have 
\begin{align*}
\int_{\tau}^{\gamma_0 \tau}\sum_{(0,0) \neq (m,n)\in \Z^2} \frac{N_{\alpha_1, \alpha_2}(z, 1)^{k-1}}{(mz+n)^{2k}} \, dz 
&= \int_{\tau}^{\gamma_0 \tau} \sum_{l \in \Z} \sum_{(m,n)\in S_{\gamma_0}} \frac{N_{\alpha_1, \alpha_2}(z, 1)^{k-1}}{j(\gamma_0^l, z)^{2k}(m(\gamma_0^l  z)+n)^{2k}} \, dz
\\
&= \sum_{l \in \Z}\int_{\gamma_0^{l} \tau}^{\gamma_0^{l+1} \tau}  \sum_{(m,n)\in S_{\gamma_0}} \frac{N_{\alpha_1, \alpha_2}(\gamma_0^{-l} z, 1)^{k-1}}{j(\gamma_0^l, \gamma_0^{-l} z)^{2k}(mz+n)^{2k}} \, d(\gamma_0^{-l} z)
\\
&= \sum_{l \in \Z}\int_{\gamma_0^{l} \tau}^{\gamma_0^{l+1} \tau}  \sum_{(m,n)\in S_{\gamma_0}} \frac{N_{\alpha_1, \alpha_2}(z, 1)^{k-1}}{(mz+n)^{2k}} \, dz. 
\end{align*}
Set $\alpha_0 := \alpha_1/\alpha_2 \in F \subset \R$. Then the point $\alpha_0 \in \R$ 
(resp. $\alpha'_0$) is the attractive fixed point (resp. repelling fixed point) of the hyperbolic matrix $\gamma_0 \in \Gamma$, i.e., we have $\lim_{l \to \infty} \gamma_0^{l}\tau= \alpha_0$ and $\lim_{l \to \infty} \gamma_0^{-l}\tau= \alpha'_0$ in $\bbP^1(\C)$ for any element $\tau \in \bbH$. 
Hence we obtain 
\begin{align*}
    \sum_{l \in \Z}\int_{\gamma_0^{l} \tau}^{\gamma_0^{l+1} \tau}  \sum_{(m,n)\in S_{\gamma_0}} \frac{N_{\alpha_1, \alpha_2}(z, 1)^{k-1}}{(mz+n)^{2k}} \, dz 
    &= \int_{\alpha'_0}^{\alpha_0}  \sum_{(m,n)\in S_{\gamma_0}} \frac{N_{\alpha_1, \alpha_2}(z, 1)^{k-1}}{(mz+n)^{2k}} \, dz
\\
&= \frac{N_{F/\Q}(\alpha_2)^{k-1}}{(N\mfa)^{k-1}}\sum_{(m,n)\in S_{\gamma_0}}\int_{\alpha'_0}^{\alpha_0}\frac{((\alpha_0-z)(z-\alpha'_0))^{k-1}}{(mz+n)^{2k}}\, dz. 
\end{align*}
By using Lemma \ref{lem:feynman}, we find
\begin{align*}
\int_{\alpha'_0}^{\alpha_0}\frac{((\alpha_0-z)(z-\alpha'_0))^{k-1}}{(mz+n)^{2k}}\, dz = \frac{((k-1)!)^2}{(2k-1)!}\frac{(\alpha_0-\alpha'_0)^{2k-1}}{N_{F/\Q}(m\alpha_0+n)^{k}}. 
\end{align*}
Note that we have the identity $\alpha_1\alpha'_2- \alpha'_1\alpha_2=\sqrt{D_{\mcO}} N\mfa$, and this shows that 
\begin{align*}
\frac{N_{F/\Q}(\alpha_2)^{k-1}}{(N\mfa)^{k-1}}\sum_{(m,n)\in S_{\gamma_0}}\frac{(\alpha_0-\alpha'_0)^{2k-1}}{N_{F/\Q}(m\alpha_0+n)^{k}} 
= 
 D_{\mcO}^{k-1/2}(N\mfa)^k \sum_{\alpha \in (\mfa \setm\{0\})/\mcO_{+}^{\times}}\frac{1}{N_{F/\Q}(\alpha)^{k}}. 
\end{align*}
For any subset $X \subset F$, we put 
\begin{align*}
X_{+} &:=\{\alpha \in X \mid \alpha>0,\, \alpha'>0\}, 
\\
X_{-} &:= \{\alpha \in X \mid \alpha>0,\, \alpha'<0\}. 
\end{align*}
Let $\mcJ \in Cl_{\mcO}^+$ denote the ideal class containing the principal ideal $(\sqrt{D_{\mcO}}) \subset \mcO$. Note that $\mcJ^{-1}=\mcJ$ in $Cl_{\mcO}^+$. 
Then we further compute
\begin{align*}
(N\mfa)^k \sum_{\alpha \in (\mfa \setm\{0\})/\mcO_{+}^{\times}}\frac{1}{N_{F/\Q}(\alpha)^{k}} 
 &= 
 \sum_{\alpha \in \mfa_+/\mcO_{+}^{\times}}\frac{2(N\mfa)^k}{N_{F/\Q}(\alpha)^k} + \sum_{\alpha \in \mfa_-/\mcO_{+}^{\times}}\frac{2(N\mfa)^k}{N_{F/\Q}(\alpha)^k}
\\
&=
\sum_{\alpha \in \mfa_+/\mcO_{+}^{\times}}\frac{2(N\mfa)^{k}}{N_{F/\Q}(\alpha)^k}+(-1)^k\sum_{\alpha \in (\sqrt{D_{\mcO}}\mfa)_+/\mcO_{+}^{\times}}\frac{2 N(\sqrt{D_{\mcO}}\mfa)^{k}}{N_{F/\Q}(\alpha)^k}
\\
&=
2 \left(\zeta_{\mcO}(\mcA^{-1}, k)+(-1)^k\zeta_{\mcO}(\mcJ\mcA^{-1},k)\right). 
\end{align*}
Now, we recall the functional equations of the partial zeta functions. Set
\begin{align*}
\Lambda_{\mcO}^+(\mcA^{-1},s)
&:=
\pi^{-s}
\Gamma\lr{\frac{s}{2}}^2 D_{\mcO}^{\frac{s-1}{2}}
\lr{
\zeta_{\mcO}(\mcA^{-1},s)
+
\zeta_{\mcO}(\mcA^{-1}\mcJ,s)
}, \\
\Lambda_{\mcO}^-(\mcA^{-1}, s)
&:=
\pi^{-s}
\Gamma\lr{\frac{s+1}{2}}^2 D_{\mcO}^{\frac{s}{2}}
\lr{
\zeta_{\mcO}(\mcA^{-1},s)
-
\zeta_{\mcO}(\mcA^{-1}\mcJ,s)
}. 
\end{align*}
Then we have
\begin{align*}
\Lambda_{\mcO}^+(\mcA^{-1},s)
&=
\Lambda_{\mcO}^+(\mcA^{-1},1-s), \\
\Lambda_{\mcO}^-(\mcA^{-1},s)
&=
\Lambda_{\mcO}^-(\mcA^{-1},1-s). 
\end{align*}
See \cite[Equations (59), (60)]{DIT18} or \cite[p. 545]{Sczech93} for example. 
Although \cite{DIT18} deals only with the maximal orders, we can apply the same argument to general orders. See also \cite{Siegel80}, \cite[(4.19)]{Duke23} or \cite[p. 42]{VZ13}. 
Using these functional equations, we find
\[
D_{\mcO}^{k - 1/2}\left(\zeta_{\mcO}(\mcA^{-1}, k)+(-1)^k\zeta_{\mcO}(\mcJ\mcA^{-1},k)\right) = \frac{(2\pi)^{2k}}{2((k-1)!)^2} \zeta_{\mcO}(\mcA^{-1}, 1-k). 
\]
Therefore, by also using the functional equation for the Riemann zeta function (see, for example, \cite[p. 29]{Hida93}), we obtain 
\begin{align*}
\brk{\Eis_{2k-2}, \mfz_{\mcO,k}(\mcA)} = \frac{(2\pi)^{2k}}{2(2k-1)!\zeta(2k)}\zeta_{\mcO}(\mcA^{-1},1-k) = (-1)^k\frac{\zeta_{\mcO}(\mcA^{-1},1-k)}{\zeta(1-2k)}. 
\end{align*}
\end{proof}

We define the positive integer $J_{2k}$ by 
\[
J_{2k} := \text{the denominator of } \zeta(1-2k). 
\]

\begin{cor}\label{cor:partial zeta denominator and J_n}
Let $F$ be a real quadratic field, $\mcO \subset F$ be an order in $F$, and let $\mcA \in Cl_{\mcO}^+$ be a narrow ideal class of $\mcO$. 
Then for any integer $k \geq 2$, we have
\[
J_{2k}\zeta_{\mcO}(\mcA, 1-k) \in \Z. 
\]
\end{cor}

\begin{proof}
By Proposition \ref{prop:Hecke integral formula}, we have  
\[
N_{2k}\brk{\Eis_{2k-2}, \mfz_{\mcO,k}(\mcA^{-1})}
= \pm N_{2k} \frac{\zeta_{\mcO}(\mcA,1-k)}{\zeta(1-2k)} = \pm J_{2k}\zeta_{\mcO}(\mcA,1-k). 
\]
Since $N_{2k}\brk{\Eis_{2k-2}, \mfz_{\mcO,k}(\mcA^{-1})} \in \bZ$ by Theorem \ref{thm:main result}, we obtain $J_{2k}\zeta_{\mcO}(\mcA,1-k) \in \bZ$. 
\end{proof}

\begin{rmk}\label{rem:duke's conjecture and integrality partial zetas are equivalent}
    By Proposition \ref{prop:rademacher integral} (\cite[Lemma 6]{Duke23}) and Proposition \ref{prop:Hecke integral formula}, we see that 
    Duke's conjecture \ref{conj:duke} is equivalent to Corollary \ref{cor:partial zeta denominator and J_n}. 
\end{rmk}

\begin{rmk}
As for the denominator of the special values of the Dedekind zeta functions of real quadratic fields, or more generally of totally real fields, 
  the same (a slightly stronger at $p=2$) universal upper bound was obtained  by Serre in the paper \cite[\S 2, Th\'eor\`eme 6]{Serre73}.  
Moreover, if we fix a totally real field $F$, then a more refined description for the denominators and even for the numerators of the special values of the Dedekind zeta function of $F$ is obtained from the classical Iwasawa main conjecture proved by Wiles in \cite{Wiles90} (see \cite{Kol04}). 
\end{rmk}

\subsection{Sharpness of the universal upper bound in Corollary \ref{cor:partial zeta denominator and J_n}}\label{sec:sharpness}

Let $k \geq 2$ be an integer.  We define a $\Z$-submodule $\mfZ_{k} \subset H_1(Y, \mcM_{2k-2})$ to be the $\Z$-submodule generated by homology classes of the form $\mfz_{\mcO, k}(\mcA)$, that is, 
\[
\mfZ_k := \langle \, \mfz_{\mcO, k}(\mcA) \mid \textrm{$\mcO$ is an order of a real quadratic field and $\mcA \in Cl_\mcO^+$ } \rangle_{\bZ}. 
\]
This subsection is devoted to proving  the following theorem.

\begin{thm}\label{thm:fmZ-pairing-image-is-equal-to-1/N_2k-Z}\
 We have $\displaystyle \left\langle  \Eis_{2k-2}, \mfZ_k  \right\rangle = \frac{1}{N_{2k}}\bZ$. 
\end{thm}

Theorem \ref{thm:fmZ-pairing-image-is-equal-to-1/N_2k-Z} has the following interesting application. 

\begin{cor}\label{cor:sherpness of the universal bound of partial zeta functions}
The universal bound in Corollary \ref{cor:partial zeta denominator and J_n} is sharp, namely, for any prime number $p$, there exist an order $\mcO$ of a real quadratic field  and a narrow ideal class $\mcA \in Cl_\mcO^+$ such that 
\[
\mathrm{ord}_p(J_{2k}\zeta_\mcO(\mcA, 1-k)) = 0. 
\]
In other words, we have
\[
J_{2k}=\min\left\{J \in \Z_{>0} \,\,\middle|\,\, 
\begin{aligned}
J\zeta_\mcO(\mcA, 1-k) \in \Z 
\text{ for all orders $\mcO$ in all real quadratic fields}\\ 
\text{and narrow ideal classes $\mcA \in Cl_{\mcO}^+$}
\end{aligned}\,
\right\}. 
\]
\end{cor}
\begin{proof}
    Let $p$ be a prime number. By the definition of the module $\mfZ_k$ and Theorem \ref{thm:fmZ-pairing-image-is-equal-to-1/N_2k-Z}, one can find an order $\mcO$ of a real quadratic field  and a narrow ideal class $\mcA \in Cl_\mcO^+$ such that 
    \[
    \mathrm{ord}_p(\langle \Eis_{2k-2}, \mfz_{\mcO, k}(\mcA^{-1}) \rangle) = -\mathrm{ord}_p(N_{2k}).
    \]
    Since $\zeta(1-2k) = \pm N_{2k}/J_{2k}$, Proposition \ref{prop:Hecke integral formula} shows that 
    \begin{align*}
            0 &=  \mathrm{ord}_p(N_{2k}) + \mathrm{ord}_p(\langle \Eis_{2k-2}, \mfz_{\mcO, k}(\mcA^{-1}) \rangle) 
            \\
            &=  \mathrm{ord}_p(N_{2k}) - \mathrm{ord}_p(\zeta(1-2k)) +\mathrm{ord}_p(\zeta_\mcO(\mcA, 1-k))
            \\
            &= \mathrm{ord}_p(J_{2k}\zeta_\mcO(\mcA, 1-k)). 
        \end{align*}
\end{proof}

\subsubsection{Preparations for proving Theorem \ref{thm:fmZ-pairing-image-is-equal-to-1/N_2k-Z}}

Let $N \geq 1$ be an integer and define
\[
\Gamma_1(N) := \left\{\begin{pmatrix}
    a&b\\c&d
\end{pmatrix} \in \Gamma \, \middle| \, a -1 \equiv c \equiv d -1 \equiv 0 \pmod{N} \right\}. 
\]
We also put  
\[
Y_1(N) :=  \Gamma_1(N) \bs \bbH, \,\,\,\,\,\, Y_1(N)^\BS :=  \Gamma_1(N) \bs \bbH^\BS, \,\,\,\,\,\,  \partial Y_1(N)^\BS :=Y_1(N)^\BS \setm Y_1(N). 
\]
We note that the similar facts in \S \ref{sec:Definitions of Modular curve and Borel--Serre compactification} and \S \ref{sec:Modular symbols and (co)homology} hold true for the congruence subgroup $\Gamma_1(N)$. 
Moreover, the Hecke operators 
$T_p$ ($p \nmid N$) 
and $U_p$ ($p \mid N$)
act on the homology group $H_1(Y_1(N), \cM_{2k-2})$.

Let $k \geq 1$ be an integer. 
For any hyperbolic matrix $\gamma \in \Gamma_1(N)$ (i.e., $|\mathrm{trace}(\gamma)| > 2$), we set  
\[
\fz_{\Gamma_1(N), k}(\gamma) := [\{z, \gamma  z \} \otimes Q_\gamma(X_1, X_2)^{k-1}] \in H_1(Y_1(N), \cM_{2k-2}). 
\]

\begin{dfn}
    For any integer $k \geq 1$, we define a $\bZ$-submodule 
    \[
    \mfZ_{\Gamma_1(N), k} \subset H_1(Y_1(N), \cM_{2k-2})
    \]
    by 
    \[
    \mfZ_{\Gamma_1(N), k} := \langle \, \fz_{\Gamma_1(N), k}(\gamma) \mid \gamma \in \Gamma_1(N) \textrm{ with } |\mathrm{trace}(\gamma)| > 2 \, \rangle_\bZ. 
    \]
\end{dfn}

\begin{rmk}\label{rem:Zk-and-ZGamma11}
    By Remark \ref{rem:Gauss's theory hyperbolic elements}, we have     $\mfZ_{\Gamma_1(1), k} = \mfZ_{k}$. 
\end{rmk}

\begin{lem}\label{lem:parabolic gamma_1(N) generator}
For any integer $N \geq 1$ and prime number $p$, we have 
\[
 \langle \, [\{z, \gamma  z \}] \mid \gamma \in \Gamma_1(Np)\setm \Gamma(p) \textrm{ with } |\mathrm{trace}(\gamma)| > 2 \,\rangle_\bZ = 
\mfZ_{\Gamma_1(Np), 1} = H_1(Y_1(Np), \bZ).
\]
\end{lem}
\begin{proof}
It suffices to show that 
\[
 \langle \, \gamma \mid \gamma \in \Gamma_1(Np)\setm \Gamma(p) \textrm{ with } |\mathrm{trace}(\gamma)| > 2 \, \rangle = \Gamma_1(Np). 
\]
Here $\brk{~}$ means the group generated by the elements inside the bracket.
Moreover, since  $\Gamma_1(2) = \langle \Gamma_1(6), \Gamma_1(10)\rangle$, we may assume that $Np \geq 3$. 
       Put $\gamma := \begin{pmatrix}
        1&1\\Np&1+Np
    \end{pmatrix}$. Then the quotient group  $\Gamma_1(Np)/(\Gamma(p) \cap \Gamma_1(Np))$ is generated by the image of $\gamma$. 
    For any matrix $\gamma' \in \Gamma(p) \cap \Gamma_1(Np)$, we have $\gamma'\gamma^{1+ap} \not\in \Gamma(p)$ for any integer $a$. 
    Since $\mathrm{trace}(\gamma) = Np +2 > 2$ and $\det(\gamma) = 1$, the matrix $\gamma$ is hyperbolic, and one can take  a matrix $Q \in \mathrm{GL}_2(\bbR)$ such that $Q^{-1} \gamma Q = \begin{pmatrix}
        \alpha&\\&\alpha^{-1}
    \end{pmatrix}$. 
    If we put $Q^{-1}\gamma'Q =: \begin{pmatrix}
        x&y\\z&w
    \end{pmatrix}$, then we have $\mathrm{trace}(\gamma'\gamma^{1+ap}) = x \alpha^{1+ap} + w \alpha^{-1-ap}$. 
    Since $\mathrm{trace}(\gamma'\gamma^{1+ap}) \equiv 2 \pmod{Np}$ and $Np \geq 3$, we have $\mathrm{trace}(\gamma'\gamma^{1+ap}) \neq 0$. Hence the set $\{\mathrm{trace}(\gamma'\gamma^{1+ap}) \mid a \in \bZ\} \subset \bZ$ is infinite. 
    Therefore we can find an integer $a$ such that $|\mathrm{trace}(\gamma'\gamma^{1+ap})| > 2$. 
    \end{proof}

Next, we recall an important result proved by Hida (see \cite[Corollary 4.5]{Hida86} or \cite[Corollary 8.2]{Hida88}).  
Take an integer $N \geq 1$ and  a prime number $p$, and put $q := p^{\ord_p(Np)}$. 
Then we have a $\Gamma_1(q)$-homomorphism 
    \[
    j \colon \bZ/(q) \longrightarrow \mcM_{2k-2} \otimes \bZ/(q); b \mapsto bX_2^{2k-2},
    \]
    which induces a Hecke-equivariant homomorphism  
    \[
    j_* \colon   H_1(Y_1(Np), \bZ/(q)) \longrightarrow H_1(Y_1(Np), \mcM_{2k-2, p} \otimes \bZ/(q)). 
    \]

\begin{prop}\label{prop:hida-redction-to-the-constant-sheaf-case} 
When $Np \geq 4$,  the homomorphism $j$ induces a Hecke-equivariant isomorphism 
    \[
    j_* \colon   H_1^{\ord}(Y_1(Np), \bZ/(q)) \stackrel{\sim}{\longrightarrow} H_1^{\ord}(Y_1(Np), \mcM_{2k-2, p} \otimes \bZ/(q)) . 
    \]
Here we define $H_1^{\ord}(Y_1(Np), -) := e_{U_p}H_1(Y_1(Np), -)$. 
\end{prop}
\begin{proof}
The proof of this proposition is essentially the same as that of \cite[Theorem 2 in \S7.2]{Hida93} for cohomology groups. 
We note that $\Gamma_1(Np)$ is torsion-free since $Np \geq 4$.  
Hence any short exact sequence of $\Gamma_1(Np)$-modules induces a long exact sequence in homology.

If we put 
\[
\mcC := (\mcM_{2k-2}/ \bZ X_2^{2k-2}) \otimes \bZ/(q), 
\]
then the short exact sequence $0 \longrightarrow \bZ/(q) \overset{j}{\longrightarrow} \mcM_{2k-2} \otimes \bZ/(q) \longrightarrow \mcC \longrightarrow 0$ induces 
an exact sequence of $\bZ/(q)$-modules
\begin{align*}
    H_2(Y_1(Np), \mcC ) \longrightarrow H_1(Y_1(Np), \bZ/(q)) \stackrel{j_*}{\longrightarrow} H_1(Y_1(N), \mcM_{2k-2} \otimes\Z/(q))  \longrightarrow H_1(Y_1(Np), \mcC ). 
\end{align*}
Since the operator $U_p$ is defined by 
$U_p = \sum_{u=0}^{p-1}
\begin{pmatrix}
    1&u\\
    0&p
\end{pmatrix}$, 
 we have 
\begin{align*}
    \begin{pmatrix}
    1&u\\
    0&p
\end{pmatrix} \cdot P(X_1, X_2) &\equiv P(-uX_2, X_2) \pmod{p} 
\\
&\in \bF_p  X_2^{2k-2}
\end{align*}
for any polynomial $P \in \mcM_{2k-2}$. 
In other words, we have
\[
    \begin{pmatrix}
    1&u\\
    0&p
\end{pmatrix}
\lr{
\mcM_{2k-2}/\Z X_2^{2k-1}
}
\subset
p 
\lr{
\mcM_{2k-2}/\Z X_2^{2k-1}
},  
\]
and this shows that $H_1^{\ord}(Y_1(N), \mcC ) = H_2^{\ord}(Y_1(N), \mcC ) = 0$,  which completes the proof. 
\end{proof}

\begin{lem}\label{lem:mod-q-compatibility-different-weights}
Let $N$ be an integer and  let $p$ be a prime number. 
Set $q := p^{\ord_p(Np)}$. 
Then for any matrix $\gamma \in \Gamma_1(Np)\setm \Gamma(p)$, we have 
    \[
       j_*(\fz_{\Gamma_1(Np), 1}(\gamma) \bmod{q}) \in (\bZ/(q))^{\times} \cdot \fz_{\Gamma_1(Np), k}(\gamma). 
    \]
\end{lem}
\begin{proof}
    If we put $\gamma :=  \begin{pmatrix}
        a&b\\c&d
    \end{pmatrix}$, then we have 
    \begin{align*}
             Q_{\gamma}(X_1, X_2) = -\frac{\sgn(a+d)}{\mathrm{gcd}(c,a-d,b)}(cX_1^2-(a-d)X_1X_2-bX_2^2). 
        \end{align*}
    Since $\gamma \in \Gamma_1(Np)\setm \Gamma(p)$, we have 
    $c \equiv a-d \equiv 0 \pmod{q}$ and $b \not\equiv 0 \pmod{p}$, which shows that 
   \[
   Q_{\gamma}(X_1, X_2) \equiv \pm \frac{b}{\mathrm{gcd}(c,a-d,b)} X_2^2 \pmod{q}
   \]
   and 
\begin{align*}
j_*(\fz_{\Gamma_1(Np), 1}(\gamma) \bmod{q}) 
&= \left(\pm \frac{\mathrm{gcd}(c,a-d,b)}{b} \right)^{k-1} \fz_{\Gamma_1(Np), k}(\gamma) 
\pmod{q}
\\
&\in (\bZ/(q))^{\times} \cdot \fz_{\Gamma_1(Np), k}(\gamma). 
\end{align*}
\end{proof}

\begin{thm}\label{thm:N geq 5 and p mid N case fZ = H_1-ordinary}
    For any integer $N \geq 1$ and prime number $p$ satisfying $Np \geq 4$, we have 
    \[
    e_{U_p}(\fZ_{\Gamma_1(Np),k} \otimes \bZ_p) = H_1^\ord(Y_1(Np), \cM_{2k-2, p}). 
    \]
\end{thm}

\begin{proof}
       We first note that $e_{U_p}H_0(Y_1(Np), \mcM_{2k-2, p}) = e_{U_p}((\mcM_{2k-2, p})_{\Gamma_1(Np)})$ vanishes 
       since $U_p \cdot X_2^{2k-2} = pX_2^{2k-2}$ and $\begin{pmatrix}
    1&u\\
    0&p
\end{pmatrix} \cdot P(X_1, X_2) \equiv P(-uX_2, X_2) \pmod{p}$.
    This fact implies that 
    \[
    H_1^{\ord}(Y_1(Np), \mcM_{2k-2, p} ) \otimes \bZ/(q)
\stackrel{\sim}{\longrightarrow} 
    H_1^{\ord}(Y_1(Np), \mcM_{2k-2, p} \otimes \bZ/(q)). 
    \]
    Here $q := p^{\ord_p(Np)}$. 
    Hence this theorem follows from Proposition \ref{prop:hida-redction-to-the-constant-sheaf-case} and Lemmas \ref{lem:parabolic gamma_1(N) generator} and \ref{lem:mod-q-compatibility-different-weights}.  
\end{proof}

\subsubsection{Proof of  Theorem \ref{thm:fmZ-pairing-image-is-equal-to-1/N_2k-Z}}

Let $k \geq 2$ be an integer. 
For any positive integers $M$ and $N$ with $M \mid N$, we denote by 
\begin{align*}
\pi_*^{N,M} &\colon H_1(Y_1(N), \cM_{2k-2}) \longrightarrow H_1(Y_1(M), \cM_{2k-2}), 
\\
\pi^*_{M,N} &\colon H^1(Y_1(M), \cM_{2k-2}) \longrightarrow H^1(Y_1(N), \cM_{2k-2})
\end{align*}
the homomorphisms induced by the natural projection $Y_1(N) \longrightarrow Y_1(M); z \mapsto z$.

\begin{cor}\label{cor:pairing-image-of-Np-zeta-homology-classes-containes-Zp}
    For any integer $N \geq 1$ and prime number $p$, we have 
    \[
    \bZ_p \subset \langle \Eis_{2k-2}, \pi^{Np,1}_*(e_{U_p}(\fZ_{\Gamma_{1}(Np), k} \otimes \bZ_p)) \rangle. 
    \]
\end{cor}

\begin{proof}
Take an element $\tau \in \bbH$. 
Since $U_p([\{\tau, \tau + 1\} \otimes X_2^{2k-2}]) = [\{\frac{\tau}{p}, \frac{\tau}{p} + 1\} \otimes X_2^{2k-2}] = [\{\tau, \tau + 1\} \otimes X_2^{2k-2}]$, 
we have 
   \[
    [\{\tau, \tau + 1\} \otimes X_2^{2k-2}] \in H_1^\ord(\partial Y_1(Np)^{\BS}, \mcM_{2k-2, p}) \subset H_1^\ord(Y_1(Np), \mcM_{2k-2, p}).  
    \]
        On the other hand, since 
    \[
    \pi^{Np,1}_*([\{\tau, \tau + 1\}  \otimes X_2^{2k-2}]) =
        [\{\tau, \tau + 1\}  \otimes X_2^{2k-2}] \in H_1(Y, \mcM_{2k-2}), 
    \]
    we have
    \[
    \langle \Eis_{2k-2}, \pi^{Np,1}_*( [\{\tau, \tau + 1\}  \otimes X_2^{2k-2}]) \rangle = 1, 
    \]
    Therefore, first when $Np \geq 4$, Theorem \ref{thm:N geq 5 and p mid N case fZ = H_1-ordinary} implies this lemma. 
When $Np \leq 3$, we have $N=1$ and $p|6$. Then this case follows from the case $N=3$ and $p \mid 6$ 
    since $\pi^{3p,1}_*(e_{U_p}(\fZ_{\Gamma_{1}(3p), k} \otimes \bZ_p) \subset \pi^{p,1}_*(e_{U_p}(\fZ_{\Gamma_{1}(p), k} \otimes \bZ_p)$. 
\end{proof}

\begin{lem}\label{lem:ordinary-homology-compatibility}
Let $N \geq 1$ be an integer and let $p$ be a prime number. 
Then for any homology class $x \in H_1(Y_1(Np), \mcM_{2k-2}) \otimes \bC_p$ we have 
\[
e_{T_p}\pi_*^{Np,1}(e_{U_p}x) = \pi_*^{Np, 1}(e_{U_p}x). 
\]
\end{lem}

\begin{proof}
By using the formal duality, it suffices to show that 
\[
e_{U_p'}\pi_{1, Np}^*(e_{T_p'}y) = e_{U_p'}\pi_{1,Np}^*(y)
\]
for any cohomology class $y \in H^1(Y, \mcM_{2k-2}) \otimes \bC_p$. 
This claim is well-known (see \cite[Lemma 2]{Fer92} for example). 
\end{proof}

\begin{cor}\label{cor:pairing-Eis-and-fZ_Gamma_1(p)=1/N_2k-Z_p}
For any prime number $p \geq 5$, we have 
\[
\langle \Eis_{2k-2}, \pi_*^{p, 1}(e_{U_p}(\fZ_{\Gamma_1(p), k} \otimes \bZ_p)) \rangle = \frac{1}{N_{2k}}\bZ_p.  
\]
\end{cor}

\begin{proof}
    Take a prime number $p \geq 5$. Then Theorem \ref{thm:N geq 5 and p mid N case fZ = H_1-ordinary} shows that 
    \[
    \langle \Eis_{2k-2}, \pi_*^{p, 1}(e_{U_p}(\fZ_{\Gamma_1(p), k} \otimes \bZ_p)) \rangle = \langle \Eis_{2k-2}, \pi_*^{p, 1}(H_1^\ord(Y_1(p), \mcM_{2k-2, p})) \rangle. 
    \]
By Lemma \ref{lem:ordinary-homology-compatibility}, we have a natural homomorphism 
\[
\pi_*^{p,1} \colon H_1^\ord(Y_1(p), \mcM_{2k-2, p})/(\mathrm{torsion}) \longrightarrow H_1^\ord(Y, \mcM_{2k-2, p})/(\mathrm{torsion}), 
\]
which is the dual of the homomorphism 
\[
H^1_\ord(Y, \mcM_{2k-2, p}) \longrightarrow H^1_\ord(Y_1(p), \mcM_{2k-2, p}); y \mapsto e_{U_p'}\pi^{*}_{1, p}(y). 
\]
The fact that the index $[\Gamma_0(p) \colon \Gamma_1(p)] = p - 1$ is relatively prime to $p$ together with the isomorphism \eqref{eq:isom-T_p-ordinary-and-U_p-ordinary}  implies that $\pi_*^{p,1} \colon H_1^\ord(Y_1(p), \mcM_{2k-2, p})/(\mathrm{torsion}) \longrightarrow H_1^\ord(Y, \mcM_{2k-2, p})/(\mathrm{torsion})$ is surjective. 
Hence we have 
\begin{align*}
\langle \Eis_{2k-2}, \pi_*^{p, 1}(H_1^\ord(Y_1(p), \mcM_{2k-2, p})) \rangle = 
\langle \Eis_{2k-2}, H_1^\ord(Y, \mcM_{2k-2, p}) \rangle. 
\end{align*}
Therefore, this corollary follows from Theorem \ref{thm:main result} and Lemma \ref{lem:e_p-pairing-ordinary}. 
\end{proof}

We  note that for any prime number $p$, the operators 
$V_p := \begin{pmatrix}p&0\\0&1\end{pmatrix}$ and $V_p' := \widetilde{\begin{pmatrix}p&0\\0&1\end{pmatrix}}$ induce homomorphisms 
\begin{align*}
V_p &\colon H_1(Y_1(Np), \cM_{2k-2}) \longrightarrow  H_1(Y_1(N), \cM_{2k-2}), 
\\
V_p' &\colon H^1(Y_1(N), \cM_{2k-2}) \longrightarrow  H^1(Y_1(Np), \cM_{2k-2}).   
\end{align*}

\begin{lem}\label{lem:eisenstein class U_p-ordinary relation}
    For any prime number $p$ and  integer $N \geq 1$, we have 
    \[
    e_{U_p'}(\pi^*_{1, Np}(\Eis_{2k-2})) = \frac{1}{1-p^{2k-1}}\left(\pi^*_{1, Np}(\Eis_{2k-2}) - \pi^*_{1, N}(\Eis_{2k-2})|V_p' \right). 
    \]
\end{lem}
\begin{proof}
We put 
\begin{align*}
    \Eis_{2k-2}^{(1)} &:= \frac{1}{1-p^{2k-1}} \left(\pi^*_{1, Np}(\Eis_{2k-2}) -\pi^*_{1,N}(\Eis_{2k-2})|V_p' \right), 
    \\
    \Eis_{2k-2}^{(p^{2k-1})} &:= \frac{1}{1-p^{2k-1}} \left(-p^{2k-1}\pi^*_{1,Np}(\Eis_{2k-2}) +  \pi^*_{1,N}(\Eis_{2k-2})|V_p' \right). 
\end{align*}
    Since $\pi_{1,Np}^*T_p' = U_p'\pi_{1,Np}^* + V_p'\pi_{1,N}^*$ and $U_p'V_p'\pi_{1,N}^* = p^{2k-1}\pi_{1,Np}^*$, the relation that $\Eis_{2k-2}|T_p' = (1 + p^{2k-1})\Eis_{2k-2}$ shows that 
    \begin{align*}
            \Eis_{2k-2}^{(1)}|U_p'  = \Eis_{2k-2}^{(1)} \,\,\, \textrm{ and } \,\,\, \Eis_{2k-2}^{(p^{2k-1})}|U_p'  =  p^{2k-1}\Eis_{2k-2}^{(p^{2k-1})}.  
    \end{align*}
    Since $\pi^*_{1, Np}(\Eis_{2k-2}) = \Eis_{2k-2}^{(1)} +  \Eis_{2k-2}^{(p^{2k-1})}$, these facts imply this lemma. 
\end{proof}

\begin{lem}\label{lem:inclusion-N-and-1-zeta-homology-classes}
    For any integer $N \geq 1$ and a prime number $p$, we have 
    \[
    \langle \Eis_{2k-2}, \pi^{Np, 1}_*(e_{U_p}(\fZ_{\Gamma_{1}(Np), k} \otimes \bZ_p))  \rangle \subset 
    \langle \Eis_{2k-2}, \fZ_{\Gamma_1(1), k} \rangle \bZ_p. 
    \]
\end{lem}
\begin{proof}
Lemma \ref{lem:eisenstein class U_p-ordinary relation} implies that 
    \begin{align*}
        \langle e_{U_p'}\pi^*_{Np, 1}(\Eis_{2k-2}), \fZ_{\Gamma_{1}(Np), k}  \rangle &\subset  \langle \pi^*_{1, Np}(\Eis_{2k-2}), \fZ_{\Gamma_{1}(Np), k}  \rangle\bZ_p +  \langle V_p'\pi^*_{1, N}(\Eis_{2k-2}), \fZ_{\Gamma_{1}(Np), k}  \rangle\bZ_p
        \\
        &\subset \langle \Eis_{2k-2}, \fZ_{\Gamma_1(1), k}  \rangle\bZ_p +  \langle \Eis_{2k-2}, \pi_*^{N, 1}(V_p\fZ_{\Gamma_{1}(Np), k})  \rangle\bZ_p. 
    \end{align*}
Hence we have 
\begin{align*}
 \langle \Eis_{2k-2}, \pi^{Np,1}_*(e_{U_p}(\fZ_{\Gamma_{1}(Np), k} \otimes \bZ_p))  \rangle &= 
  \langle e_{U_p'}\pi^*_{1, Np}(\Eis_{2k-2}), \fZ_{\Gamma_{1}(N), k}  \rangle\bZ_p
  \\
  &\subset \langle \Eis_{2k-2}, \fZ_{\Gamma_1(1), k}  \rangle\bZ_p +  \langle \Eis_{2k-2}, \pi_*^{N, 1}(V_p\fZ_{\Gamma_{1}(Np), k})  \rangle\bZ_p. 
    \end{align*}
    Therefore, it suffices to show that $V_p\fZ_{\Gamma_{1}(Np), k} \subset \fZ_{\Gamma_{1}(N), k}$.
    Let $\gamma \in \Gamma_1(Np)$ be a matrix. Then we have $\gamma_p := \begin{pmatrix}
        p&0\\0&1
    \end{pmatrix}\gamma\begin{pmatrix}
        p&0\\0&1
    \end{pmatrix}^{-1} \in \Gamma_1(N)$. 
Moreover, by the definitions of $\mfz_{\Gamma_1(N), k}$ and  $\mfz_{\Gamma_1(Np), k}$, we obtain 
\[
V_p \cdot \mfz_{\Gamma_1(Np), k}(\gamma)  \in \bZ \mfz_{\Gamma_1(N), k}(\gamma_p). 
\]
In particular, we have $V_p  \fZ_{\Gamma_{1}(Np), k} \subset \fZ_{\Gamma_{1}(N), k}$. 
\end{proof}

\begin{proof}[Proof of Theorem \ref{thm:fmZ-pairing-image-is-equal-to-1/N_2k-Z}]
By Theorem \ref{thm:main result} and Remark \ref{rem:Zk-and-ZGamma11}, we only need to show that 
\[
\frac{1}{N_{2k}}\bZ_p \subset  \langle \Eis_{2k-2}, \mfZ_{\Gamma_1(1), k} \rangle\bZ_p
\]
for any prime number $p$. 
When $p \geq 5$, this claim follows from Corollary \ref{cor:pairing-Eis-and-fZ_Gamma_1(p)=1/N_2k-Z_p} and Lemma \ref{lem:inclusion-N-and-1-zeta-homology-classes} applied to $N=1$. 
Suppose $p=2$ or $p=3$. Then
    Since the primes $2$ and $3$ are regular, these primes does not divide $N_{2k}$. 
Hence, Corollary \ref{cor:pairing-image-of-Np-zeta-homology-classes-containes-Zp} and Lemma \ref{lem:inclusion-N-and-1-zeta-homology-classes} show that 
\[
\frac{1}{N_{2k}}\bZ_p  = \bZ_p \subset \langle \Eis_{2k-2}, \mfZ_{\Gamma_1(1), k} \rangle\bZ_p. 
\]
\end{proof}

\bibliography{references}

\providecommand{\bysame}{\leavevmode\hbox to3em{\hrulefill}\thinspace}
\providecommand{\MR}{\relax\ifhmode\unskip\space\fi MR }
\providecommand{\MRhref}[2]{%
  \href{http://www.ams.org/mathscinet-getitem?mr=#1}{#2}
}
\providecommand{\href}[2]{#2}
\begin{thebibliography}{BHYY22}

\bibitem[BCG20]{BCG20}
Nicolas Bergeron, Pierre Charollois, and Luis~E. Garcia, \emph{Transgressions
  of the {E}uler class and {E}isenstein cohomology of {${\rm GL}_N({\bf Z})$}},
  Jpn. J. Math. \textbf{15} (2020), no.~2, 311--379. \MR{4120422}

\bibitem[Bel21]{Bel21}
Jo\"{e}l Bella\"{\i}che, \emph{The eigenbook---eigenvarieties, families of
  {G}alois representations, {$p$}-adic {$L$}-functions}, Pathways in
  Mathematics, Birkh\"{a}user/Springer, Cham, [2021] \copyright 2021.
  \MR{4306639}

\bibitem[Ber08]{Berger08}
Tobias Berger, \emph{Denominators of {E}isenstein cohomology classes for {${\rm
  GL}_2$} over imaginary quadratic fields}, Manuscripta Math. \textbf{125}
  (2008), no.~4, 427--470. \MR{2392880}

\bibitem[Ber09]{Berger09}
\bysame, \emph{On the {E}isenstein ideal for imaginary quadratic fields},
  Compos. Math. \textbf{145} (2009), no.~3, 603--632. \MR{2507743}

\bibitem[BHYY22]{BHYY22}
Kenichi Bannai, Kei Hagihara, Kazuki Yamada, and Shuji Yamamoto,
  \emph{{$p$}-adic polylogarithms and {$p$}-adic {H}ecke {$L$}-functions for
  totally real fields}, J. Reine Angew. Math. \textbf{791} (2022), 53--87.
  \MR{4489625}

\bibitem[BKL18]{BKL18}
Alexander Beilinson, Guido Kings, and Andrey Levin, \emph{Topological
  polylogarithms and {$p$}-adic interpolation of {$L$}-values of totally real
  fields}, Math. Ann. \textbf{371} (2018), no.~3-4, 1449--1495. \MR{3831278}

\bibitem[Bra23]{Branchereau23}
Romain Branchereau, \emph{An upper bound on the denominator of {E}isenstein
  classes in {B}ianchi manifolds}, Preprint, arXiv:2305.11341, 2023.

\bibitem[Car61]{Car61}
L.~Carlitz, \emph{A generalization of {M}aillet's determinant and a bound for
  the first factor of the class number}, Proc. Amer. Math. Soc. \textbf{12}
  (1961), 256--261. \MR{121354}

\bibitem[CDG15]{CDG15}
Pierre Charollois, Samit Dasgupta, and Matthew Greenberg, \emph{Integral
  {E}isenstein cocycles on {$\bold{GL}_n$}, {II}: {S}hintani's method},
  Comment. Math. Helv. \textbf{90} (2015), no.~2, 435--477. \MR{3351752}

\bibitem[CN79]{Cassou-Nogues}
Pierrette Cassou-Nogu\`es, \emph{Valeurs aux entiers n\'{e}gatifs des fonctions
  z\^{e}ta et fonctions z\^{e}ta {$p$}-adiques}, Invent. Math. \textbf{51}
  (1979), no.~1, 29--59. \MR{524276}

\bibitem[Cox13]{Cox13}
David~A. Cox, \emph{Primes of the form {$x^2 + ny^2$}}, second ed., Pure and
  Applied Mathematics (Hoboken), John Wiley \& Sons, Inc., Hoboken, NJ, 2013,
  Fermat, class field theory, and complex multiplication. \MR{3236783}

\bibitem[CS74a]{CS74-Stickelberger}
J.~Coates and W.~Sinnott, \emph{An analogue of {S}tickelberger's theorem for
  the higher {$K$}-groups}, Invent. Math. \textbf{24} (1974), 149--161.
  \MR{369322}

\bibitem[CS74b]{CS74-p-adic}
\bysame, \emph{On {$p$}-adic {$L$}-functions over real quadratic fields},
  Invent. Math. \textbf{25} (1974), 253--279. \MR{354615}

\bibitem[CS77]{CS77}
\bysame, \emph{Integrality properties of the values of partial zeta functions},
  Proc. London Math. Soc. (3) \textbf{34} (1977), no.~2, 365--384. \MR{439815}

\bibitem[DIT18]{DIT18}
W.~Duke, \"{O}. Imamo\={g}lu, and \'{A}. T\'{o}th, \emph{Kronecker's first
  limit formula, revisited}, Res. Math. Sci. \textbf{5} (2018), no.~2, Paper
  No. 20, 21. \MR{3782448}

\bibitem[DR80]{DR80}
Pierre Deligne and Kenneth~A. Ribet, \emph{Values of abelian {$L$}-functions at
  negative integers over totally real fields}, Invent. Math. \textbf{59}
  (1980), no.~3, 227--286. \MR{579702}

\bibitem[Duk23]{Duke23}
William Duke, \emph{Higher rademacher symbols}, Experimental Mathematics
  \textbf{0} (2023), no.~0, 1--20.

\bibitem[Gor05]{Gor05}
Mark Goresky, \emph{Compactifications and cohomology of modular varieties},
  Harmonic analysis, the trace formula, and {S}himura varieties, Clay Math.
  Proc., vol.~4, Amer. Math. Soc., Providence, RI, 2005, pp.~551--582.
  \MR{2192016}

\bibitem[Gou92]{Fer92}
Fernando~Q. Gouv\^{e}a, \emph{On the ordinary {H}ecke algebra}, J. Number
  Theory \textbf{41} (1992), no.~2, 178--198. \MR{1164797}

\bibitem[Hab83]{Haberland}
Klaus Haberland, \emph{Perioden von {M}odulformen einer {V}ariabler and
  {G}ruppencohomologie. {I}, {II}, {III}}, Math. Nachr. \textbf{112} (1983),
  245--282, 283--295, 297--315. \MR{726861}

\bibitem[Har]{CAG}
G\"{u}nter Harder, \emph{{C}ohomology of {A}rithmetic {G}ruops}, available on
  Harder's webpage,
  \url{http://www.math.uni-bonn.de/people/harder/Manuscripts/buch/}.

\bibitem[Har81]{Harder81}
\bysame, \emph{Period integrals of cohomology classes which are represented by
  {E}isenstein series}, Automorphic forms, representation theory and arithmetic
  ({B}ombay, 1979), Tata Inst. Fund. Res. Studies in Math, vol.~10, Tata
  Institute of Fundamental Research, Bombay, 1981, pp.~pp 41--115. \MR{633658}

\bibitem[Har82]{Harder82}
\bysame, \emph{Period integrals of {E}isenstein cohomology classes and special
  values of some {$L$}-functions}, Number theory related to {F}ermat's last
  theorem ({C}ambridge, {M}ass., 1981), Progr. Math., vol.~26, Birkh\"{a}user
  Boston, Boston, MA, 1982, pp.~103--142. \MR{685293}

\bibitem[Har11]{Harder11}
\bysame, \emph{Interpolating coefficient systems and {$p$}-ordinary cohomology
  of arithmetic groups}, Groups Geom. Dyn. \textbf{5} (2011), no.~2, 393--444.
  \MR{2782179}

\bibitem[Har21]{harder_can_we}
\bysame, \emph{{C}an we compute denominators of {E}isenstein classes?}, Adv.
  Stud. Euro-Tbil. Math. J. \textbf{Special Issue} (2021), no.~9, 131--165,
  Special issue on Cohomology, Geometry, Explicit Number Theory.

\bibitem[Hid86]{Hida86}
Haruzo Hida, \emph{Galois representations into {${\rm GL}_2({\bf Z}_p[[X]])$}
  attached to ordinary cusp forms}, Invent. Math. \textbf{85} (1986), no.~3,
  545--613. \MR{848685}

\bibitem[Hid88]{Hida88}
\bysame, \emph{On {$p$}-adic {H}ecke algebras for {${\rm GL}_2$} over totally
  real fields}, Ann. of Math. (2) \textbf{128} (1988), no.~2, 295--384.
  \MR{960949}

\bibitem[Hid93]{Hida93}
\bysame, \emph{Elementary theory of {$L$}-functions and {E}isenstein series},
  London Mathematical Society Student Texts, vol.~26, Cambridge University
  Press, Cambridge, 1993. \MR{1216135}

\bibitem[HP92]{HP92}
G\"{u}nter Harder and Richard Pink, \emph{Modular konstruierte unverzweigte
  abelsche {$p$}-{E}rweiterungen von {${\bf Q}(\zeta_p)$} und die {S}truktur
  ihrer {G}aloisgruppen}, Math. Nachr. \textbf{159} (1992), 83--99.
  \MR{1237103}

\bibitem[Kai90]{Kaiser90}
Christian Kaiser, \emph{Die nenner von eisensteinklassen f\"ur gewisse
  kongruenzgruppen}, Diploma thesis, 1990, Rheinische
  Friedrich-Wilhelms-Universit\"{a}t Bonn.

\bibitem[Kol04]{Kol04}
Manfred Kolster, \emph{{$K$}-theory and arithmetic}, Contemporary developments
  in algebraic {$K$}-theory, ICTP Lect. Notes, vol.~XV, Abdus Salam Int. Cent.
  Theoret. Phys., Trieste, 2004, pp.~191--258. \MR{2175640}

\bibitem[Mae93]{Maennel93}
Hartmut Maennel, \emph{Nenner von {E}isensteinklassen auf {H}ilbertschen
  {M}odulvariet\"{a}ten und die {$p$}-adische {K}lassenzahlformel}, Bonner
  Mathematische Schriften [Bonn Mathematical Publications], vol. 247,
  Universit\"{a}t Bonn, Mathematisches Institut, Bonn, 1993, Dissertation,
  Rheinische Friedrich-Wilhelms-Universit\"{a}t Bonn, Bonn, 1992. \MR{1290786}

\bibitem[Mah00]{Mahnkopf00}
Joachim Mahnkopf, \emph{Eisenstein cohomology and the construction of
  {$p$}-adic analytic {$L$}-functions}, Compositio Math. \textbf{124} (2000),
  no.~3, 253--304. \MR{1809337}

\bibitem[O'S23]{Sullivan23}
Cormac O'Sullivan, \emph{{I}ntegrality of the higher {R}ademacher symbols},
  Preprint, arXiv:2306.08831, 2023.

\bibitem[Scz93]{Sczech93}
Robert Sczech, \emph{Eisenstein group cocycles for {${\rm GL}_n$} and values of
  {$L$}-functions}, Invent. Math. \textbf{113} (1993), no.~3, 581--616.
  \MR{1231838}

\bibitem[Ser73]{Serre73}
Jean-Pierre Serre, \emph{Congruences et formes modulaires [d'apr\`es {H}. {P}.
  {F}. {S}winnerton-{D}yer]}, S\'{e}minaire {B}ourbaki, 24e ann\'{e}e
  (1971/1972), {E}xp. {N}o. 416, Lecture Notes in Math, vol. Vol. 317,
  Springer, Berlin, 1973, pp.~pp 319--338. \MR{466020}

\bibitem[Shi71]{Shimura-book-arith-autom-funct}
Goro Shimura, \emph{Introduction to the arithmetic theory of automorphic
  functions}, vol. No. 1., Iwanami Shoten Publishers, Tokyo; Princeton
  University Press, Princeton, N.J., 1971, Publications of the Mathematical
  Society of Japan, No. 11. \MR{314766}

\bibitem[Shi76]{Shintani76}
Takuro Shintani, \emph{On evaluation of zeta functions of totally real
  algebraic number fields at non-positive integers}, J. Fac. Sci. Univ. Tokyo
  Sect. IA Math. \textbf{23} (1976), no.~2, 393--417. \MR{427231}

\bibitem[Sie80]{Siegel80}
Carl~Ludwig Siegel, \emph{Advanced analytic number theory}, second ed.,
  vol.~9., Tata Institute of Fundamental Research, Bombay, 1980. \MR{659851}

\bibitem[Sku80]{Sku80}
Ladislav Skula, \emph{Index of irregularity of a prime}, J. Reine Angew. Math.
  \textbf{315} (1980), 92--106. \MR{564526}

\bibitem[VZ13]{VZ13}
Maria Vlasenko and Don Zagier, \emph{Higher {K}ronecker ``limit'' formulas for
  real quadratic fields}, J. Reine Angew. Math. \textbf{679} (2013), 23--64.
  \MR{3065153}

\bibitem[Wan89]{Wang89}
Xiang~Dong Wang, \emph{Die {E}isensteinklasse in {$H^1({\rm SL}_2({\bf
  Z}),M_n({\bf Z}))$} und die {A}rithmetik spezieller {W}erte von
  {$L$}-{F}unktionen}, Bonner Mathematische Schriften [Bonn Mathematical
  Publications], vol. 202, Universit\"{a}t Bonn, Mathematisches Institut, Bonn,
  1989, Dissertation, Rheinische Friedrich-Wilhelms-Universit\"{a}t Bonn, Bonn,
  1989. \MR{1015506}

\bibitem[Was97]{LW97}
Lawrence~C. Washington, \emph{Introduction to cyclotomic fields}, second ed.,
  Graduate Texts in Mathematics, vol.~83, Springer-Verlag, New York, 1997.
  \MR{1421575}

\bibitem[Wes88]{Weselmann88}
Uwe Weselmann, \emph{Eisensteinkohomologie und {D}edekindsummen f\"{u}r {${\rm
  GL}_2$} \"{u}ber imagin\"{a}r-quadratischen {Z}ahlk\"{o}rpern}, J. Reine
  Angew. Math. \textbf{389} (1988), 90--121. \MR{953667}

\bibitem[Wil90]{Wiles90}
Andrew Wiles, \emph{The {I}wasawa conjecture for totally real fields}, Ann. of
  Math. (2) \textbf{131} (1990), no.~3, 493--540. \MR{1053488}

\bibitem[Zag76]{Zagier76}
Don Zagier, \emph{On the values at negative integers of the zeta-function of a
  real quadratic field}, Enseign. Math. (2) \textbf{22} (1976), no.~1-2,
  55--95. \MR{406957}

\bibitem[Zag77]{Zagier77}
\bysame, \emph{Valeurs des fonctions z\^{e}ta des corps quadratiques r\'{e}els
  aux entiers n\'{e}gatifs}, Journ\'{e}es {A}rithm\'{e}tiques de {C}aen
  ({U}niv. {C}aen, {C}aen, 1976), Ast\'{e}risque No. 41--42, Soc. Math. France,
  Paris, 1977, pp.~pp 135--151. \MR{441925}

\end{thebibliography}
\bibliographystyle{amsalpha}

\end{document}